\documentclass[a4paper,10pt]{article}
 \usepackage[
  a4paper,
  left=18mm,
  right=18mm,
  top=17mm,
  bottom=19mm
]{geometry}
\usepackage[english]{babel}
\usepackage{amsthm,amsmath,amssymb}
\usepackage{natbib}

\newtheorem{thm}{Theorem}[section]

\newtheorem{claim}[thm]{Claim}
\newtheorem{lemma}[thm]{Lemma}

\newtheorem{prop}[thm]{Proposition}
\newtheorem{defin}[thm]{Definition}
\newtheorem{rema}[thm]{Remark}
\newtheorem{corollary}[thm]{Corollary}

\newcommand{\E}{\mathbb E}
\renewcommand{\P}{\mathbb P}
\newcommand{\N}{\mathbb N}
\allowdisplaybreaks
\title{Stability of compensator integrals and nonlinear operators under
canonical \(p\)-th power variation convergence}
\author{Philip Kennerberg
\footnote{Philip Kennerberg  
Independent researcher, Sweden  
Email: pkennerberg@gmail.com  
ORCID: 0000-0003-3366-2046}
}  
\date{}

\begin{document}
\maketitle

\begin{abstract}
Let \(X\) and \((X^n)_{n\ge1}\) be càdlàg processes with predictable
compensators \(\nu\) and \(\nu^n\). For \(p>1\), assume that

$$
        [X^n-X]^{(p)}_t
        \longrightarrow0
        \qquad\text{in probability},
$$

where \([\,\cdot\,]^{(p)}\) denotes canonical \(p\)-th power
variation.

Our principal result proves canonical \(p\)-th power-variation
stability for a class of nonlinear compensator operators. Each atom
response is the sum of a stable global Lipschitz cylinder component
generated by finitely many compensator integrals and a nonlinear
residual evaluated on the restriction of the corresponding mark
measure to a compact set. A stopped compensation estimate controls
the gauge energy, while weighted cylinder approximation transfers
stability to the full operator and accommodates weighted response
perturbations. No topological convergence of \(\nu^n\) to \(\nu\) is
assumed. Examples include covariance, cumulant, kernel and regularised
transport responses.

The linear foundation is a stability theorem for compensator
integrals under local growth, locally uniform convergence of the
integrands, jump continuity and barycentric tail control. Two
structural results drive its proof. A threshold isolation principle
constructs spatial bands eventually avoided by the limiting and
approximating jump magnitudes. Within the compensator mass-control
argument, a separate matching step identifies sufficiently large
jumps at predictable atom times, and localisation produces an almost
surely summable envelope for the annular atom masses. Subject only to
the standing process and finiteness assumptions, both structural
results follow from canonical \(p\)-th power-variation convergence
itself, independently of the additional hypotheses used for
compensator-integral stability. We also give sufficient conditions for ucp stability of compensator
integrals and, for \(p=2\), an equivalence criterion for
ucp stability of the corresponding compensated jump integrals.
\end{abstract}


\setcounter{tocdepth}{2}
\tableofcontents

\section{Introduction}

Predictable compensators provide the link between the realised jumps of
a stochastic process and their conditional structure. If \(X\) is a
càdlàg process with jump measure \(\mu\), its predictable compensator
\(\nu\) determines both finite-variation integrals
\[
        f*\nu
\]
and compensated jump integrals
\[
        f*(\mu-\nu).
\]
These objects form a basic part of the general theory of stochastic
processes and occur in semimartingale decompositions, stochastic
differential equations with jumps, filtering, stochastic control and
limit theory; see, for example,
\cite{MeyerBook,JacodShiryaev,HeWangYan,IkedaWatanabe,Protter}.

Classical stability results for stochastic integrals are commonly
formulated through semimartingale convergence, convergence of
characteristics, tightness conditions, or additional structure of the
driving processes; see
\cite{KurtzProtter,Memin,LiptserShiryaev,Protter}. Such frameworks
control the predictable characteristics directly or place the
integrators in a topology designed to guarantee stability of stochastic
integration. The question considered here is different. We ask how much
stability of predictable compensator measures can be recovered from
convergence of the underlying càdlàg processes in \(p\)-th power
variation, without assuming convergence of their characteristics.

Throughout the paper, canonical \(p\)-th power variation refers to the
limit of sums of \(p\)-th powers of increments along partitions whose
mesh tends to zero. It is distinct from the classical total
\(p\)-variation defined by taking a supremum over all partitions. For
\(p=2\), canonical \(p\)-th power variation reduces to quadratic
variation. A continuous finite-variation process has zero canonical
\(p\)-th power variation whenever \(p>1\), whereas the contribution of
its jumps is
\[
        \sum_{0<s\le t}|\Delta Y_s|^p.
\]
This makes canonical \(p\)-th power variation particularly natural for
studying compensator integrals: their continuous finite-variation
parts disappear, while their predictable atomic contributions remain
visible.

Let \(X\) and \(X^n\) have jump measures \(\mu,\mu^n\) and predictable
compensators \(\nu,\nu^n\), respectively. The linear foundation of the
paper, Theorem~\ref{thm:CMC-pvar-compensator-stability}, shows that, for every \(p>1\),
\[
        [X^n-X]^{(p)}_t
        \longrightarrow0
        \qquad\text{in probability}
\]
implies
\[
        [f_n*\nu^n-f*\nu]^{(p)}_t
        \longrightarrow0
        \qquad\text{in probability},
\]
under explicit local growth assumptions, locally uniform convergence
of \(f_n\) to \(f\), continuity of \(f(s,\cdot)\) at the realised
jumps of \(X\), and barycentric tail conditions involving
\(f_n,f,\nu^n,\nu\). Thus canonical \(p\)-th power-variation
convergence of the realised processes is transferred to predictable
compensator integrals without first imposing a convergence topology on
the compensator measures.

The hypothesis is genuinely different from convergence of predictable
characteristics. Canonical \(p\)-th power-variation convergence need
not imply convergence of the corresponding predictable
\(p\)-energies. Nor does it control accumulated first-order
compensator mass near the origin. Indeed, \(n\) predictable jumps of
size \(1/n\) contribute
\[
        n\left(\frac1n\right)^p
        =
        n^{1-p}
        \longrightarrow0
\]
to canonical \(p\)-th power variation, while their accumulated
first-order contribution is
\[
        n\frac1n=1.
\]
This phenomenon may occur even when the compensators agree eventually
on every fixed compact annulus. Thus the result controls the
\(\ell^p\)-geometry of the compensator responses without imposing the
corresponding \(\ell^1\)-control at the origin.

This transfer is not a direct consequence of convergence of the
realised jumps. Although
\[
        [X^n-X]^{(p)}_t\longrightarrow0
\]
forces the individual jump discrepancies to vanish, it does not
directly provide convergence or uniform summability of the atoms of
\(\nu^n\). A predictable compensator records conditional jump mass
rather than realised jump mass, and the relevant predictable atom
times may depend on \(n\). Consequently, atomwise convergence alone is
insufficient for passing to the complete \(p\)-th power-variation sum.

Theorem~\ref{thm:CMC-pvar-compensator-stability} does not remove all
assumptions involving the compensators: local integrability and the
barycentric tail condition \((\mathrm{BE}_p)\) remain among its
hypotheses. These are integrability and uniform
barycentric-domination conditions, rather than assumptions that
\(\nu^n\) converges to \(\nu\). In particular,
\((\mathrm{BE}_p)\) requires a common locally bounded envelope for the
entire sequence, but no limiting relation between the compensator
measures. For compactly supported cylinder coordinates, this
domination holds with a deterministic constant.

What the theorem removes is the need to assume convergence or uniform
summability of the predictable compensator atoms on bounded annuli. The required
annular control is instead derived from canonical \(p\)-th
power-variation convergence of the realised processes.

This distinction is particularly relevant for coupled approximations.
A coupling is naturally designed to control the realised path
difference, whereas the predictable compensators induced by the
common filtration may be difficult to identify or compare directly.
The theorem permits the available pathwise control to be used without
first establishing an independent convergence theorem for the
annular compensator atoms. Remark~\ref{rema:small-jump-cloud} shows
that this replacement is not vacuous: canonical \(p\)-th
power-variation convergence may hold even when the accumulated
first-order small-jump characteristic does not converge.

The passage from realised jumps to predictable compensator mass is
obtained through two structural results. The Threshold Isolation
Principle selects pathwise generic spatial bands which avoid the
limiting jump magnitudes and are eventually avoided by the
approximating jump magnitudes. It thereby stabilises the classification
of jumps by spatial threshold regions, but does not itself identify
the predictable jump times.

The Compensator Mass-Control lemma contains a separate matching
argument for sufficiently large jumps at predictable atom times.
After localisation, it converts this matching into an almost surely
summable envelope for the compensator masses on a fixed annulus,
uniformly over a tail of the approximating sequence. This envelope
provides the domination needed to pass from convergence at individual
predictable times to convergence of the entire atomic series.

The structural results are independent of the integrands. In
particular, they require none of \({\rm (LG}_0{\rm )}\),
\({\rm (LU)}\), \({\rm (JC)}\) or
\({\rm (BE}_p{\rm )}\). Apart from the standing process and finiteness
assumptions, their only asymptotic input is canonical \(p\)-th
power-variation convergence of the realised processes. The additional
growth, continuity and barycentric tail assumptions enter only when
this structural control is used to prove stability of particular
compensator integrals.

Three distinct mechanisms enter the proof of
Theorem~\ref{thm:CMC-pvar-compensator-stability}. Near the origin,
local linear growth reduces the problem to the small-jump
\(p\)-energy of \(X^n\) and \(X\). The bounded-annulus regime is the
structural core of the argument: threshold isolation and compensator
mass control upgrade convergence at individual predictable times to
convergence of the full atomic series. The unbounded tail is treated
by the barycentric estimate \((\mathrm{BE}_p)\), which prevents a
small amount of compensator mass at large marks from producing a
non-negligible compensator response.

Canonical \(p\)-th power-variation convergence controls the jumps of a
compensator integral but does not, by itself, control its continuous
finite-variation component. We therefore also give conditions under
which the preceding result can be strengthened to ucp convergence.
The atomless-in-time component is treated through the joint jump
measure of \((X^n,X)\). Uniform integrability transfers convergence of
the realised joint \(p\)-energy to convergence of its predictable
compensator, while a barycentric Radon--Nikodym estimate converts this
energy convergence into total-variation convergence of the diffuse
compensator response. For the atomic component, an interpolation
between the established \(\ell^p\)-convergence and a locally bounded
\(\ell^r\)-quantity, with \(0<r<1\), yields the required
\(\ell^1\)-convergence.

Compensated jump integrals are treated separately. The exponent
\(p=2\) is distinguished because the corresponding compensated
integrals are local martingales and quadratic variation is connected
to their maximal processes through the Burkholder--Davis--Gundy
inequalities. Under the assumptions of the compensator-integral
theorem with \(p=2\), we prove
\[
\left[
f_n*(\mu^n-\nu^n)
-
f*(\mu-\nu)
\right]_t
\longrightarrow0
\qquad\text{in probability}.
\]
The realised-jump contributions are controlled by the same threshold
separation and jump-matching mechanisms, while the predictable
contribution is supplied by the quadratic-variation stability of the
compensator integrals. Moreover, outside a Lebesgue-null set of
truncation levels, ucp convergence of the compensated integrals is
equivalent to ucp convergence of their large-jump compensator
contributions. This isolates the precise obstruction which remains
after the quadratic-variation convergence has been established.

The principal nonlinear result places the compensator-integral theorem
inside a common operator framework. The first is a stable global
cylinder component. Let
\[
        I_n^1,\ldots,I_n^r
        \qquad\text{and}\qquad
        I^1,\ldots,I^r
\]
be integral compensator operators along \(\nu^n\) and \(\nu\),
respectively, whose integrands satisfy the hypotheses of
Theorem~\ref{thm:CMC-pvar-compensator-stability}. If
\[
        \Phi:\mathbb R^r\longrightarrow H
\]
is globally \(\ell^p\)-Lipschitz and \(\Phi(0)=0\), the corresponding
atom responses of this component are
\[
\begin{aligned}
\Pi_n(s)
&=
\Phi\left(
        \Delta Y_s^{I_n^1,\nu^n},
        \ldots,
        \Delta Y_s^{I_n^r,\nu^n}
\right),
\\
\Pi(s)
&=
\Phi\left(
        \Delta Y_s^{I^1,\nu},
        \ldots,
        \Delta Y_s^{I^r,\nu}
\right).
\end{aligned}
\]
The linear stability theorem and the Lipschitz property of \(\Phi\)
give canonical \(p\)-th power-variation stability of this component.

The second component is a nonlinear residual defined on restricted
compensator atoms. Let \(E\subset\mathbb R\) be compact and let
\[
        w:E\longrightarrow[0,\infty)
\]
be continuous. If \(0\in E\), assume that \(w(0)=0\) and that
\[
        w(x)\le c_w|x|
\]
in a neighbourhood of the origin. We refer to \((E,w)\) as a locally
linear compact control gauge. Let
\[
K(E)
:=
\left\{
        \lambda\in\mathcal M_+(E):
        \lambda(E)\le1
\right\}.
\]
For a predictable compensator measure \(\eta\), set
\[
        \lambda_s^\eta
        :=
        \eta(\{s\},\cdot)|_E
        \in K(E),
\]
and, for \(\lambda\in K(E)\), define
\[
        q_w(\lambda)
        :=
        \int_Ew(x)^p\,\lambda(dx).
\]

Let
\[
        R:K(E)\longrightarrow H
\]
be a continuous response which depends only on the restriction of
\(\lambda\) to the set where \(w\) is nonzero. Assume that
\[
        R(\lambda)=0
        \qquad\text{whenever}\qquad
        q_w(\lambda)=0
\]
and
\[
\lim_{\delta\downarrow0}
\sup_{\substack{\lambda\in K(E)\\
                 0<q_w(\lambda)\le\delta}}
\frac{\lVert R(\lambda)\rVert_H}
     {q_w(\lambda)^{1/p}}
=
0.
\]
The weighted approximation theorem shows that \(R\) can be
approximated by gauge-controlled cylinder responses with error small
relative to \(q_w(\lambda)^{1/p}\).

The approximating residuals \(R_n\) may vary with \(n\), provided
\[
\sup_{\substack{\lambda\in K(E)\\q_w(\lambda)>0}}
\frac{\lVert R_n(\lambda)-R(\lambda)\rVert_H}
     {q_w(\lambda)^{1/p}}
\longrightarrow0.
\]
The total atom responses are then
\[
        \Pi_n(s)+R_n(\lambda_s^{\nu^n})
        \qquad\text{and}\qquad
        \Pi(s)+R(\lambda_s^\nu).
\]
For admissible compensator operators with these atom responses, the
principal nonlinear theorem gives
\[
        [Y^{B_n,\nu^n}-Y^{B,\nu}]^{(p)}_t
        \longrightarrow0
        \qquad\text{in probability}.
\]

The accumulated gauge energies appearing in the approximation error
are not imposed as an additional tightness assumption. A stopped
compensation estimate shows that they are finite and bounded in
probability. Indeed, compactness of \(E\) and the local linear control
of \(w\) imply
\[
        w(x)\le c_{E,w}|x|,
        \qquad x\in E,
\]
so the gauge energies are dominated by bounded-mark predictable
\(p\)-energies.

Theorem~\ref{thm:CMC-pvar-compensator-stability} is recovered by
taking \(r=1\), \(\Phi\) equal to the identity and \(R_n=R=0\).
Conversely, the previous compact-gauge formulation is obtained by
using its cylinder response as the global component and retaining the
remaining nonlinear response as the compact-gauge residual.
Continuous finite-variation residuals of the operator processes may
be retained, since their differences make no contribution to
canonical \(p\)-th power variation.

Two important regimes are contained in this formulation. On a compact
annulus one may take \(w\equiv1\), so that
\[
        q_w(\lambda)=\lambda(E).
\]
On a compact neighbourhood of the origin one may instead take
\[
        w(x)=|x|,
        \qquad
        q_w(\lambda)=\int_E|x|^p\,\lambda(dx).
\]
In both cases, the required accumulated gauge-energy control follows
from the bounded-mark predictable \(p\)-energy estimate. The nonlinear
stability theory therefore treats both separated large-jump regions
and compact small-jump regions without an additional compensator-energy
hypothesis.

The resulting operator class contains both finite-dimensional
moment-based responses and functionals of the entire compensator atom.
Gauge-weighted covariance and cumulant responses, together with
finite-rank kernel discrepancies, are polynomial cylinder operators.
General continuous kernel discrepancies and regularised transport
responses are obtained through the weighted closure and need not be
determined by finitely many prescribed moments. The constant-gauge
specialisation recovers the corresponding unweighted forms on compact
annuli. History-dependent operators additionally allow the response
to the current compensator atom to depend on state variables
accumulated from the preceding compensator history. Their stability
combines canonical \(p\)-th power-variation control of the atomic
responses with ucp stability of the compensator-integral state
variables.

\section{Overview of the paper}

Section~\ref{sec:preliminaries} introduces the notation for jump
measures and predictable compensators and recalls the relevant
properties of canonical \(p\)-th power variation. In particular, it
records the domination of the \(p\)-th jump energy by canonical
\(p\)-th power variation and the fact that continuous
finite-variation processes have zero canonical \(p\)-th power
variation whenever \(p>1\).

Section~\ref{sec:integrals} contains the stability results for
compensator and compensated jump integrals. Its principal theorem
shows that, for every \(p>1\), convergence of the underlying
càdlàg processes in canonical \(p\)-th power variation implies
convergence of the corresponding compensator integrals in the same
variation. Examples show that the conclusion may fail
without the local linear-growth, locally uniform convergence and
barycentric tail conditions. They also show that canonical
\(p\)-th power-variation convergence neither forces convergence of
predictable \(p\)-energies nor controls first-order small-jump
accumulation.

Section~\ref{sec:integrals} also gives two complementary extensions.
Under additional joint-energy and barycentric assumptions, convergence
of the compensator integrals is strengthened to ucp convergence. In
the case \(p=2\), the results extend to compensated jump integrals.
These integrals are then local martingales, so the
Burkholder--Davis--Gundy inequalities connect their quadratic
variation with ucp convergence. Outside a Lebesgue-null set of
truncation levels, ucp convergence of the compensated integrals is
equivalent to ucp convergence of their large-jump compensator
contributions.

Section~\ref{sec:structure} develops two structural consequences of
canonical \(p\)-th power-variation convergence. The Threshold
Isolation Principle constructs spatial bands which separate the
limiting jump magnitudes and are eventually avoided by the
approximating jump magnitudes. Within the compensator mass-control
argument, a separate matching step identifies the sufficiently large
jumps at predictable atom times. The resulting lemma yields an almost
surely summable envelope for the annular compensator atoms, uniformly
over a localised tail of the approximating sequence. Neither
structural result uses the integrand-dependent assumptions of
Section~\ref{sec:integrals}.

Section~\ref{sec:operators} develops the nonlinear stability theory
for Hilbert-valued compensator operators. Operators whose continuous
residual has finite variation admit an atomic representation of their
canonical \(p\)-th power variation. This leads to a natural
equivalence relation and an \(\ell^p\)-geometry on their atomic
response maps.

The principal operator theorem combines two components. A global
cylinder component is generated by finitely many compensator integrals
covered by Theorem~\ref{thm:CMC-pvar-compensator-stability} and a
finite-dimensional Lipschitz map. A compact-gauge residual may depend
nonlinearly on the entire restricted compensator atom. Weighted
cylinder approximation controls this residual relative to the
\(p\)-th root of the gauge energy, while a stopped compensation
estimate derives the required accumulated-energy bounds from
canonical \(p\)-th power-variation convergence.

The zero-residual, one-coordinate case recovers the compensator-
integral theorem. When the cylinder component is supported on the
compact mark set, the result gives the compact-gauge theory. The
constant gauge yields nonlinear stability on compact annuli, whereas
the choice \(w(x)=|x|\) gives nonlinear stability on compact
neighbourhoods of the origin.

The examples in Section~\ref{sec:operators} include gauge-weighted
covariance and cumulant responses, kernel discrepancies, regularised
transport responses, their unweighted annular specialisations and
history-dependent compensator operators.

Finally, Section~\ref{sec:proof} proves the structural results stated
in Section~\ref{sec:structure} and combines them with the small-jump,
bounded-annulus and tail estimates to complete the proof of
Theorem~\ref{thm:CMC-pvar-compensator-stability}.

\section{Preliminaries}\label{sec:preliminaries}

Whenever a statement is formulated on \([0,t]\), the number \(t>0\)
denotes a fixed finite time horizon. A global ucp conclusion is made
only when the corresponding assumptions hold on every finite time
interval.

We assume that all processes are defined on a common filtered
probability space
\[
        (\Omega,\mathcal F,(\mathcal F_u)_{u\ge0},\mathbb P)
\]
satisfying the usual hypotheses, and that they are adapted to the
filtration \((\mathcal F_u)_{u\ge0}\). Given a càdlàg process \(X\)
and a stopping time \(T\), we write
\[
        X^T_s:=X_{s\wedge T},
        \qquad s\ge0.
\]
We also write
\[
        X^*_s:=\sup_{0\le u\le s}|X_u|,
        \qquad s\ge0,
\]
for the running supremum of \(X\).

\begin{defin}\label{loc}
A property of a stochastic process is said to hold locally if there exists a
sequence of stopping times \(T_k\uparrow\infty\) such that the property holds
for the stopped process \(X^{T_k}\) for each \(k\). It is said to hold
pre-locally if the corresponding property holds for \(X^{T_k-}\) for each
\(k\).
\end{defin}

The following notation for jump measures will be used frequently. Whenever a
càdlàg process \(X\) with jump measure \(\mu\) is considered, \(\nu\) denotes
a predictable compensator of \(\mu\).

\begin{defin}\label{calA}
Given a càdlàg process \(X\) with jump measure \(\mu\) and predictable
compensator \(\nu\), we define
\[
\mathcal A
:=
\left\{
(\omega,s)\in\Omega\times(0,\infty):
\nu(\omega,\{s\},\mathbb R)>0
\right\}.
\]
Thus \(\mathcal A\) is the predictable time support of the atoms of
\(\nu\).
\end{defin}

The following proposition is a special case of Proposition ~1.17(b) in
\cite{JacodShiryaev}.

\begin{prop}
There exists a version of \(\nu\) such that
\[
        \nu(\{s\},\mathbb R)\leq1,
        \qquad s>0,
\]
outside an evanescent set, and such that \(\mathcal A\) is a thin set
exhausted by a sequence of predictable stopping times.
\end{prop}

Throughout the paper we work with such a version of the compensator.

\begin{rema}
Throughout the paper we use the convention \(\inf\emptyset=\infty\).
\end{rema}

\begin{defin}[Canonical \(p\)-th power variation]
Let \(p>1\), let \(t>0\), let \(H\) be a separable Hilbert space, and
let \(Y\) be
an \(H\)-valued càdlàg process. For a partition
\[
        \pi=\{0=t_0<t_1<\cdots<t_N=t\},
\]
set
\[
V_\pi^{(p)}(Y)_u
:=
\sum_{i=1}^N
\left\lVert
Y_{t_i\wedge u}-Y_{t_{i-1}\wedge u}
\right\rVert_H^p,
\qquad 0\le u\le t.
\]
We say that \(Y\) possesses finite canonical \(p\)-th power variation
on \([0,t]\) if there exists a finite increasing càdlàg process
\([Y]^{(p)}\) such that, for every sequence of partitions
\((\pi_m)_{m\ge1}\) with mesh tending to zero,
\[
        V_{\pi_m}^{(p)}(Y)
        \longrightarrow
        [Y]^{(p)}
        \qquad\text{ucp on }[0,t].
\]
\end{defin}
For an \(H\)-valued càdlàg process \(Y\), we write
\[
\operatorname{Var}(Y)_u
:=
\sup_{\pi}
\sum_{i=1}^{N}
\left\lVert
Y_{t_i}-Y_{t_{i-1}}
\right\rVert_H,
\qquad 0\le u\le t,
\]
where the supremum is taken over all finite partitions
\[
        \pi=\{0=t_0<t_1<\cdots<t_N=u\}
\]
of \([0,u]\). Thus \(Y\) has finite variation on \([0,t]\) if
\[
        \operatorname{Var}(Y)_t<\infty.
\]
\begin{lemma}
Let \(p>1\), let \(t>0\), let \(H\) be a separable Hilbert space, and
let \(Y\) be an \(H\)-valued càdlàg finite-variation process on
\([0,t]\). Then \(Y\) possesses finite canonical
\(p\)-th power variation and
\[
        [Y]^{(p)}_u
        =
        \sum_{0<s\le u}
        \lVert\Delta Y_s\rVert_H^p,
        \qquad 0\le u\le t.
\]
\end{lemma}

\begin{proof}
Write \(Y=Y^c+Y^d\), where \(Y^c\) is continuous and
\[
        Y_u^d
        :=
        \sum_{0<s\le u}\Delta Y_s.
\]
The jump series converges absolutely because \(Y\) has finite
variation. For every partition \(\pi\),
\[
V_\pi^{(p)}(Y^c)_t
\le
\left(
\max_{[u,v]\in\pi}
\lVert Y_v^c-Y_u^c\rVert_H
\right)^{p-1}
\operatorname{Var}(Y^c)_t,
\]
which tends to zero as the mesh of \(\pi\) tends to zero.

Fix \(\varepsilon>0\). Since
\[
        \sum_{0<s\le t}\lVert\Delta Y_s\rVert_H<\infty,
\]
there exists a finite set \(F\subset(0,t]\) of jump times such that
\[
        \sum_{\substack{0<s\le t\\s\notin F}}
        \lVert\Delta Y_s\rVert_H
        <
        \varepsilon.
\]
Set
\[
        J_u^F
        :=
        \sum_{\substack{s\in F\\s\le u}}\Delta Y_s,
        \qquad
        R_u^F
        :=
        Y_u^d-J_u^F.
\]
Then
$
        \operatorname{Var}(R^F)_t<\varepsilon.
$

For every partition \(\pi\) and every \(u\le t\), Minkowski's
inequality applied to the vector of partition increments gives
\[
\left|
V_\pi^{(p)}(Y)_u^{1/p}
-
V_\pi^{(p)}(Y^c+J^F)_u^{1/p}
\right|
\le
V_\pi^{(p)}(R^F)_u^{1/p}
\le
\operatorname{Var}(R^F)_t
<
\varepsilon.
\]
Likewise,
\[
\left|
\left(
\sum_{0<s\le u}\lVert\Delta Y_s\rVert_H^p
\right)^{1/p}
-
\left(
\sum_{\substack{s\in F\\s\le u}}
\lVert\Delta Y_s\rVert_H^p
\right)^{1/p}
\right|
<
\varepsilon.
\]

Let
\[
        \omega_{Y^c}(\delta)
        :=
        \sup_{\substack{u,v\le t\\|u-v|\le\delta}}
        \lVert Y_u^c-Y_v^c\rVert_H.
\]
For every partition \(\pi\),
\[
\sup_{u\le t}V_\pi^{(p)}(Y^c)_u
\le
\omega_{Y^c}(|\pi|)^{p-1}
\operatorname{Var}(Y^c)_t,
\]
which tends to zero as \(|\pi|\to0\).

Since \(F\) is finite, every sufficiently fine partition has at most
one time from \(F\) in each partition interval. For such partitions,
\[
        V_\pi^{(p)}(J^F)_u
        =
        \sum_{\substack{s\in F\\s\le u}}
        \lVert\Delta Y_s\rVert_H^p,
        \qquad 0\le u\le t.
\]
Another application of Minkowski's inequality therefore gives
\[
\begin{aligned}
&\sup_{u\le t}
\left|
V_\pi^{(p)}(Y^c+J^F)_u^{1/p}
-
\left(
\sum_{\substack{s\in F\\s\le u}}
\lVert\Delta Y_s\rVert_H^p
\right)^{1/p}
\right|
\le
\sup_{u\le t}V_\pi^{(p)}(Y^c)_u^{1/p}
\longrightarrow0.
\end{aligned}
\]

Combining the preceding estimates and then letting
\(\varepsilon\downarrow0\) yields
\[
\sup_{u\le t}
\left|
V_\pi^{(p)}(Y)_u^{1/p}
-
\left(
\sum_{0<s\le u}
\lVert\Delta Y_s\rVert_H^p
\right)^{1/p}
\right|
\longrightarrow0.
\]
Both quantities inside the absolute value are bounded by
\(\operatorname{Var}(Y)_t\). Hence the same conclusion holds after
raising them to the \(p\)-th power:
\[
\sup_{u\le t}
\left|
V_\pi^{(p)}(Y)_u
-
\sum_{0<s\le u}
\lVert\Delta Y_s\rVert_H^p
\right|
\longrightarrow0.
\]
The convergence is pathwise and uniform on \([0,t]\), and therefore
also ucp.
\end{proof}

\begin{lemma}[Basic properties of canonical \(p\)-th power variation]
\label{lem:canonical-p-power-basic}
Let \(p>1\) and let \(H\) be a separable Hilbert space.

If \(Y\) possesses finite canonical \(p\)-th power variation on
\([0,t]\), then
\[
\sum_{0<s\le t}
\lVert\Delta Y_s\rVert_H^p
\le
[Y]^{(p)}_t
\qquad\text{almost surely}.
\]

If \(Y\), \(Z\), and \(Y+Z\) possess finite canonical \(p\)-th power
variation on \([0,t]\), then
\[
[Y+Z]^{(p)}_t
\le
2^{p-1}
\left(
[Y]^{(p)}_t+[Z]^{(p)}_t
\right)
\qquad\text{almost surely}.
\]
\end{lemma}

\begin{proof}
For every partition \(\pi\),
\[
V_\pi^{(p)}(Y+Z)_t
\le
2^{p-1}
\left(
V_\pi^{(p)}(Y)_t+V_\pi^{(p)}(Z)_t
\right).
\]
Passing to the limits in probability along any sequence of partitions
whose mesh tends to zero gives
\[
[Y+Z]^{(p)}_t
\le
2^{p-1}
\left(
[Y]^{(p)}_t+[Z]^{(p)}_t
\right)
\]
almost surely.

To prove the first assertion, fix a sequence of partitions whose mesh
tends to zero. By passing to a subsequence, we may assume that the
corresponding partition sums converge almost surely to
\([Y]^{(p)}_t\). Fix a path on which this convergence holds, and fix
finitely many distinct jump times
\[
        s_1,\ldots,s_N\in(0,t].
\]
For all sufficiently fine partitions, these jump times belong to
distinct partition intervals. By the càdlàg property, the increment
over the interval containing \(s_j\) converges to \(\Delta Y_{s_j}\).
Consequently,
\[
[Y]^{(p)}_t
\ge
\sum_{j=1}^N
\lVert\Delta Y_{s_j}\rVert_H^p.
\]
Taking the supremum over all finite collections of jump times gives
\[
[Y]^{(p)}_t
\ge
\sum_{0<s\le t}
\lVert\Delta Y_s\rVert_H^p.
\]
\end{proof}

\section{Stability of compensator and compensated jump integrals}\label{sec:integrals}

This section establishes the linear stability theory on which the
nonlinear operator results of Section~\ref{sec:operators} are built.
Our first main result transfers convergence in canonical \(p\)-th
power variation of the realised processes to the corresponding
compensator integrals, without assuming convergence of \(\nu^n\) to
\(\nu\) in a topology on predictable measures. These integrals later
serve as the finite-dimensional coordinates of the cylinder
operators from which the nonlinear class is obtained by weighted
approximation. The remainder of the present section gives an ucp
upgrade for compensator integrals and treats compensated jump
integrals in the quadratic case.

\begin{defin}[Barycentric tail condition]
\label{def:BEp}

Fix \(p>1\). Let \(f,f_n\) be predictable functions and let
\(\nu,\nu^n\) be predictable compensator measures. For \(a>0\), set
\[
\begin{aligned}
\alpha_n^{(a)}(s)
&:=
\nu^n(\{s\},\{|x|>a\}),
&
\alpha^{(a)}(s)
&:=
\nu(\{s\},\{|x|>a\}),
\\
\beta_n^{(a)}(s)
&:=
\int_{\{|x|>a\}}
f_n(s,x)\,\nu^n(\{s\},dx),
&
\widehat\beta_n^{(a)}(s)
&:=
\int_{\{|x|>a\}}
f_n(s,x)\,\nu(\{s\},dx),
\\
\beta^{(a)}(s)
&:=
\int_{\{|x|>a\}}
f(s,x)\,\nu(\{s\},dx).
\end{aligned}
\]

We say that \(f,f_n,\nu,\nu^n\) satisfy the barycentric tail
condition \((\mathrm{BE}_p)\) if, for every \(c>0\), there exists a
nonnegative predictable process \(K^{(c)}\), locally bounded in time,
such that, outside an evanescent set,
\[
\begin{aligned}
|\beta_n^{(a)}(s)|^p
&\le
K_s^{(c)}\alpha_n^{(a)}(s),
\\
|\widehat\beta_n^{(a)}(s)|^p
&\le
K_s^{(c)}\alpha^{(a)}(s),
\\
|\beta^{(a)}(s)|^p
&\le
K_s^{(c)}\alpha^{(a)}(s),
\end{aligned}
\qquad
a\ge c,\quad n\ge1.
\]
\end{defin}

\begin{thm}
[\(p\)-th power-variation stability of compensator integrals]
\label{thm:CMC-pvar-compensator-stability}

Fix \(p>1\) and \(t<\infty\). Let \(X\) and
\((X^n)_{n\ge1}\) be real-valued adapted
càdlàg processes, defined on the same filtered probability space,
with jump measures \(\mu\) and \(\mu^n\) and predictable compensators
\(\nu\) and \(\nu^n\), respectively. Assume that \(X\), \(X^n\), and
\(X^n-X\) possess finite canonical \(p\)-th power variation on
\([0,t]\), and that
\[
        \bigl[X^n-X\bigr]^{(p)}_t
        \longrightarrow0
\]
in probability.

Let \(f\) and \(f_n\), \(n\ge1\), be real-valued predictable functions
such that
\[
        |f|*\nu\in\mathcal V_{\mathrm{loc}},
        \qquad
        |f_n|*\nu^n\in\mathcal V_{\mathrm{loc}},
\]
and define
\[
        A:=f*\nu,
        \qquad
        A^n:=f_n*\nu^n.
\]

Assume the following conditions.

\begin{enumerate}

\item[\textnormal{(LG$_0$)}]
For every \(R>0\), there exists a finite-valued nonnegative predictable
process \(C^R\) such that, outside an evanescent set,
\[
        |f_n(s,x)|+|f(s,x)|
        \le C^R_s|x|,
        \qquad |x|\le R,\quad n\ge1.
\]
Moreover, there exists \(R_0>0\) such that \(C^{R_0}\) is locally
bounded in time.

\item[\textnormal{(LU)}]
Outside a \(\P\)-null set, for every \(R>0\),
\[
        \sup_{|x|\le R}
        |f_n(s,x)-f(s,x)|
        \longrightarrow0
\]
as \(n\to\infty\), simultaneously for all \(s\in[0,t]\).

\item[\textnormal{(JC)}]
\[
\P\left(
        \exists s\in(0,t]:
        \Delta X_s\in
        \operatorname{disc}\bigl(f(s,\cdot)\bigr)
\right)=0.
\]
\item[\textnormal{(\(\mathrm{BE}_p\))}]
The barycentric tail condition of
Definition~\ref{def:BEp} holds.

\end{enumerate}

Then
\[
        [A^n-A]^{(p)}_t
        \longrightarrow0
        \qquad\text{in probability}.
\]
\end{thm}

\noindent
The proof is given in Section~\ref{sec:proof}.

A convenient sufficient condition for \((\mathrm{BE}_p)\) can be
formulated in terms of conditional \(L^{p'}\)-bounds, where
\[
        p':=\frac{p}{p-1}.
\]
Suppose that, for every \(b>0\), there exists a nonnegative predictable
process \(\Xi^{(b)}\), locally bounded in time, such that, outside an
evanescent set,
\[
\begin{aligned}
&\sup_{n\ge1}
\int_{\{|x|>b\}}
|f_n(s,x)|^{p'}\,\nu^n(\{s\},dx)
+
\sup_{n\ge1}
\int_{\{|x|>b\}}
|f_n(s,x)|^{p'}\,\nu(\{s\},dx)
+
\int_{\{|x|>b\}}
|f(s,x)|^{p'}\,\nu(\{s\},dx)
\le
\Xi_s^{(b)}.
\end{aligned}
\]
Then, for every \(a\ge b\), Hölder's inequality gives
\[
\begin{aligned}
|\beta_n^{(a)}(s)|^p
&\le
\bigl(\Xi_s^{(b)}\bigr)^{p-1}
\alpha_n^{(a)}(s),
\\
|\widehat\beta_n^{(a)}(s)|^p
&\le
\bigl(\Xi_s^{(b)}\bigr)^{p-1}
\alpha^{(a)}(s),
\\
|\beta^{(a)}(s)|^p
&\le
\bigl(\Xi_s^{(b)}\bigr)^{p-1}
\alpha^{(a)}(s).
\end{aligned}
\]
Consequently, \((\mathrm{BE}_p)\) holds with
\[
        K_s^{(b)}
        :=
        \bigl(\Xi_s^{(b)}\bigr)^{p-1}.
\]

\begin{rema}[Stability without convergence of predictable \(p\)-energies]
\label{rema:separating-example-p-energy}

The hypotheses of
Theorem~\ref{thm:CMC-pvar-compensator-stability} do not require
convergence of the predictable \(p\)-energies. The following example
also shows that such convergence cannot be recovered from
\(p\)-th power-variation convergence of the realised processes.

Fix \(p>1\), \(t>0\), and a deterministic time \(T\in(0,t]\). Let
\((a_n)_{n\ge1}\) be a deterministic sequence such that
\[
        1<a_n\uparrow\infty.
\]
On a common probability space, let \((\xi_n)_{n\ge1}\) be random
variables satisfying
\[
        \P(\xi_n=a_n)=a_n^{-p},
        \qquad
        \P(\xi_n=0)=1-a_n^{-p}.
\]
Let \((\mathcal F_u)_{u\ge0}\) be the completed right-continuous
filtration generated by revealing all the variables \(\xi_n\) at time
\(T\). Thus \(\mathcal F_u\) is trivial for \(u<T\), while
\[
        \mathcal F_u
        =
        \sigma(\xi_1,\xi_2,\ldots)
        \qquad\text{for }u\ge T.
\]
In particular, \(T\) is predictable and \(\mathcal F_{T-}\) is
trivial.

Define
\[
        X_u:=0,
        \qquad
        X_u^n:=\xi_n1_{\{T\le u\}},
        \qquad 0\le u\le t.
\]
Each \(X^n\) is a finite-jump process and hence possesses finite
canonical \(p\)-th power variation. Moreover,
\[
        [X^n-X]^{(p)}_t
        =
        a_n^p1_{\{\xi_n=a_n\}}.
\]
Consequently, for every \(\varepsilon>0\) and all sufficiently large
\(n\),
\[
\begin{aligned}
\P\left([X^n-X]^{(p)}_t>\varepsilon\right)
&=
\P(\xi_n=a_n)\\
&=
a_n^{-p}
\longrightarrow0.
\end{aligned}
\]
Thus
\[
        [X^n-X]^{(p)}_t
        \longrightarrow0
        \qquad\text{in probability}.
\]

The jump measure and its predictable compensator are
\[
        \mu^n(ds,dx)
        =
        1_{\{\xi_n=a_n\}}
        \delta_T(ds)\delta_{a_n}(dx)
\]
and
\[
        \nu^n(ds,dx)
        =
        a_n^{-p}\delta_T(ds)\delta_{a_n}(dx),
\]
respectively. Since \(X\equiv0\), its jump measure and compensator
both vanish:
\[
        \mu=0,
        \qquad
        \nu=0.
\]
Nevertheless,
\[
\begin{aligned}
\int_0^t\int_{\mathbb R}|x|^p\,\nu^n(ds,dx)
&=
a_n^{-p}a_n^p=
1
\end{aligned}
\]
for every \(n\). Hence the predictable \(p\)-energies do not converge
to the predictable \(p\)-energy of \(X\), which is zero.

Moreover,
\[
\begin{aligned}
\mathbb E\left[
[X^n-X]^{(p)}_t
\right]
&=
a_n^p\P(\xi_n=a_n)=
1
\end{aligned}
\]
for every \(n\). Since the nonnegative random variables
\([X^n-X]^{(p)}_t\) converge to zero in probability but their
expectations remain equal to one, the family
\[
        \left\{[X^n-X]^{(p)}_t:n\ge1\right\}
\]
is not uniformly integrable.

We now verify that Theorem~
\ref{thm:CMC-pvar-compensator-stability} nevertheless applies. Define
the deterministic continuous function
\[
f_p(x)
:=
\begin{cases}
|x|,
& |x|\le1,\\[1mm]
|x|^{p-1},
& |x|>1,
\end{cases}
\]
and set
\[
        f_n(s,x):=f_p(x),
        \qquad
        f(s,x):=f_p(x).
\]
For every \(R>0\),
\[
        |f_n(s,x)|+|f(s,x)|
        \le
        C^R|x|,
        \qquad |x|\le R,
\]
where one may take
\[
        C^R
        :=
        2\max\left\{1,R^{(p-2)_+}\right\}.
\]
Thus \({\rm (LG_0)}\) holds, with deterministic locally bounded
processes \(C^R\). Condition \({\rm (LU)}\) is immediate because
\(f_n=f\), while \({\rm (JC)}\) follows from the continuity of
\(f_p\).

The required compensator integrability also holds:
\[
\begin{aligned}
(|f_n|*\nu^n)_t
&=
a_n^{-p}f_p(a_n)\\
&=
a_n^{-p}a_n^{p-1}\\
&=
a_n^{-1}
<
\infty,
\end{aligned}
\]
whereas
\[
        |f|*\nu=0.
\]

It remains to verify \((\mathrm{BE}_p)\). For every truncation level
\(r>0\),
\[
\alpha_n^{(r)}(T)
=
\begin{cases}
a_n^{-p},& r<a_n,\\
0,&r\ge a_n,
\end{cases}
\]
and
\[
\beta_n^{(r)}(T)
=
\begin{cases}
a_n^{-1},& r<a_n,\\
0,&r\ge a_n.
\end{cases}
\]
Consequently,
\[
        |\beta_n^{(r)}(T)|^p
        =
        \alpha_n^{(r)}(T).
\]
Both sides vanish away from \(T\). Moreover, since \(\nu=0\),
\[
        \widehat\beta_n^{(r)}(s)=0,
        \qquad
        \beta^{(r)}(s)=0
\]
for every \(s\). Therefore \((\mathrm{BE}_p)\) holds with
\[
        K^{(c)}_s:=1,
        \qquad c>0.
\]

Finally,
\[
\begin{aligned}
A_u^n
&=
(f_n*\nu^n)_u\\
&=
a_n^{-p}f_p(a_n)1_{\{T\le u\}}=
a_n^{-1}1_{\{T\le u\}},
\end{aligned}
\]
whereas \(A\equiv0\). Hence
\[
        [A^n-A]^{(p)}_t
        =
        a_n^{-p}
        \longrightarrow0.
\]

The identity
\[
        |\beta_n^{(r)}(T)|^p
        =
        \alpha_n^{(r)}(T)
\]
shows that the barycentric estimate is attained at equality. Thus
\((\mathrm{BE}_p)\) is saturated in this example.

The conclusion of
Theorem~\ref{thm:CMC-pvar-compensator-stability} therefore holds even
though the predictable \(p\)-energies remain equal to one and the
realised \(p\)-energies are not uniformly integrable. This does not
constitute a failure of convergence away from the origin. Indeed, for
every compact annulus
\[
        E=\{x\in\mathbb R:b\le |x|\le a\},
        \qquad 0<b<a<\infty,
\]
we have
\[
        \nu^n|_{(0,t]\times E}
        =
        \nu|_{(0,t]\times E}
        =
        0
\]
for every sufficiently large \(n\).
\end{rema}

\begin{rema}[Small-jump accumulation and characteristic convergence]
\label{rema:small-jump-cloud}

Canonical \(p\)-th power-variation convergence does not imply
convergence of the accumulated first-order small-jump characteristic.

Fix \(p>1\) and work on \([0,1]\). Let \(X\) be a Poisson process with
unit intensity, considered in its completed right-continuous natural
filtration. For \(n\ge1\), set
\[
        T_{n,k}:=\frac{k}{n},
        \qquad k=1,\ldots,n,
\]
and
\[
        D_u^n
        :=
        \frac1n
        \sum_{k=1}^n
        1_{\{T_{n,k}\le u\}},
        \qquad
        X_u^n:=X_u+D_u^n.
\]
Since \(X\) has almost surely no jumps at the deterministic times
\(T_{n,k}\), the jump measure of \(X^n\) is, almost surely, the sum of
the jump measure of \(X\) and the jump measure of \(D^n\). The
processes \(X\), \(X^n\), and \(X^n-X\) possess finite canonical
\(p\)-th power variation on \([0,1]\). Moreover,
\[
        X^n-X=D^n
\]
and hence
\[
\begin{aligned}
[X^n-X]^{(p)}_1
&=
\sum_{k=1}^n
\left|\frac1n\right|^p\\
&=
n^{1-p}
\longrightarrow0.
\end{aligned}
\]

The predictable compensators are
\[
        \nu(ds,dx)
        =
        ds\,\delta_1(dx)
\]
and
\[
        \nu^n(ds,dx)
        =
        ds\,\delta_1(dx)
        +
        \sum_{k=1}^n
        \delta_{T_{n,k}}(ds)\delta_{1/n}(dx).
\]
Set
\[
        f_n(s,x)=f(s,x)=x.
\]
The local linear-growth, local uniform convergence and
discontinuity-avoidance conditions are immediate. Furthermore,
\[
        (|f|*\nu)_1=1,
        \qquad
        (|f_n|*\nu^n)_1=2.
\]

We verify \((\mathrm{BE}_p)\) directly. For every \(a>0\), at each
time \(T_{n,k}\),
\[
        \alpha_n^{(a)}(T_{n,k})
        =
        1_{\{1/n>a\}}
\]
and
\[
        \beta_n^{(a)}(T_{n,k})
        =
        \frac1n1_{\{1/n>a\}}.
\]
Therefore
\[
        |\beta_n^{(a)}(T_{n,k})|^p
        =
        \frac1{n^p}1_{\{1/n>a\}}
        \le
        \alpha_n^{(a)}(T_{n,k}).
\]
The compensator \(\nu\) is atomless in time, so
\[
        \widehat\beta_n^{(a)}(s)
        =
        \beta^{(a)}(s)
        =
        0
\]
at every time \(s\). Hence \((\mathrm{BE}_p)\) holds with
\[
        K^{(c)}_s:=1,
        \qquad c>0.
\]

Let
\[
        A^n:=f_n*\nu^n,
        \qquad
        A:=f*\nu.
\]
Then
\[
        A_u=u
\]
and
\[
        A_u^n=u+D_u^n.
\]
Consequently,
\[
        [A^n-A]^{(p)}_1
        =
        [D^n]^{(p)}_1
        =
        n^{1-p}
        \longrightarrow0,
\]
in accordance with
Theorem~\ref{thm:CMC-pvar-compensator-stability}. Nevertheless,
\[
        \sup_{0\le u\le1}|A_u^n-A_u|
        =
        D_1^n
        =
        1.
\]
Thus the conclusion cannot in general be strengthened to ucp
convergence.

The failure is entirely concentrated at the origin. For every compact
annulus
\[
        E=\{x\in\mathbb R:b\le |x|\le a\},
        \qquad 0<b<a<\infty,
\]
we have
\[
        \nu^n|_{(0,1]\times E}
        =
        \nu|_{(0,1]\times E}
\]
whenever \(n>b^{-1}\). On the other hand, if \(h\) is a truncation
function satisfying \(h(x)=x\) in a neighbourhood of the origin, then
\[
\begin{aligned}
\bigl(h*(\nu^n-\nu)\bigr)_1
&=
\sum_{k=1}^n h(1/n)&=
1
\end{aligned}
\]
for every sufficiently large \(n\). Hence canonical \(p\)-th
power-variation convergence does not imply convergence of the
accumulated truncated first-order compensator term. The distinction is
the elementary identity
\[
        \sum_{k=1}^n\left(\frac1n\right)^p
        =
        n^{1-p}
        \longrightarrow0,
        \qquad
        \sum_{k=1}^n\frac1n=1.
\]

Equality of predictable compensators does not, conversely, imply
\(p\)-th power-variation convergence under a prescribed common
coupling. Let
\(\widetilde X,\widetilde X^1,\widetilde X^2,\ldots\)
be independent Poisson processes with the same intensity, considered
in their joint natural filtration. Their predictable compensators are
identical. Since two independent Poisson processes have almost surely
no simultaneous jumps,
\[
        [\widetilde X^n-\widetilde X]^{(p)}_1
        =
        \widetilde X_1^n+\widetilde X_1,
\]
which does not converge to zero in probability. This observation
concerns convergence under the prescribed coupling and does not
contradict the fact that all these processes have the same law.
\end{rema}

\begin{rema}[Sharpness of the local linear-growth condition]
The local linear-growth condition at the origin cannot in general be
replaced by a genuinely sublinear bound.

Fix \(\alpha\in(0,1)\) and \(t>0\). For \(n\ge1\), set
\[
        a_n:=\frac1n,
        \qquad
        N_n:=\lfloor n^{\alpha p}\rfloor,
\]
and define the deterministic times
\[
        T_{n,k}:=\frac{kt}{2N_n+1},
        \qquad k=1,\ldots,2N_n.
\]
Let \(X\equiv0\) and define
\[
X_u^n
:=
a_n
\sum_{k=1}^{2N_n}
(-1)^{k+1}
1_{\{T_{n,k}\le u\}},
\qquad 0\le u\le t.
\]
Thus \(X^n\) has \(2N_n\) alternating jumps of magnitudes \(a_n\).
Since the processes are deterministic, their jump measures coincide
with their predictable compensators. Moreover,
\[
\begin{aligned}
[X^n-X]^{(p)}_t
&=
\sum_{k=1}^{2N_n}
|\Delta X^n_{T_{n,k}}|^p
\\
&=
2N_na_n^p
=
2\lfloor n^{\alpha p}\rfloor n^{-p}
\longrightarrow0.
\end{aligned}
\]

Now set
\[
        f_n=f,
        \qquad
        f(x):=\operatorname{sgn}(x)|x|^\alpha.
\]
The function \(f\) is continuous, the local uniform convergence
condition is trivially satisfied, and all the compensator integrals
are well defined because every \(X^n\) has only finitely many jumps.
The tail condition \((\mathrm{BE}_p)\) is also satisfied. However,
\(f\) does not satisfy the local linear-growth condition at the
origin.

Since
\[
        \Delta(f*\nu^n)_{T_{n,k}}
        =
        (-1)^{k+1}a_n^\alpha,
\]
we have
\[
\begin{aligned}
\bigl[f*\nu^n-f*\nu\bigr]^{(p)}_t
&=
\sum_{k=1}^{2N_n}
\left|
\Delta(f*\nu^n)_{T_{n,k}}
\right|^p
\\
&=
2N_na_n^{\alpha p}
=
2\lfloor n^{\alpha p}\rfloor n^{-\alpha p}
\longrightarrow2.
\end{aligned}
\]
Thus convergence of the underlying processes in \(p\)-th power
variation need not be preserved by a sublinear transformation of their
small jumps. In particular, for \(p=2\), the example shows the
necessity of the corresponding quadratic-variation scaling.

This example and Remark~\ref{rema:small-jump-cloud} describe two
different effects of small-jump accumulation. Here the failure of
local linear growth prevents the integrand responses from inheriting
the underlying \(\ell^p\)-control. In
Remark~\ref{rema:small-jump-cloud}, local linear growth holds and the
\(\ell^p\)-conclusion remains valid, but the corresponding
first-order \(\ell^1\)-accumulation does not vanish.
\end{rema}

\begin{rema}[Pointwise convergence of the integrands is insufficient]
The local uniform convergence condition in the mark variable cannot
in general be replaced by pointwise convergence.

Fix a deterministic time \(T\in(0,t]\), and define
\[
        X_u:=1_{\{T\le u\}},
        \qquad
        X_u^n
        :=
        \left(1+\frac1n\right)1_{\{T\le u\}}.
\]
Then
\[
        [X^n-X]^{(p)}_t
        =
        \frac1{n^p}
        \longrightarrow0.
\]
The jump measures and their predictable compensators are
\[
        \mu^n=\nu^n
        =
        \delta_T\otimes\delta_{1+1/n},
        \qquad
        \mu=\nu
        =
        \delta_T\otimes\delta_1.
\]

Define
\[
        f\equiv0,
        \qquad
        f_n(x):=1_{\{1+1/n\}}(x).
\]
For every fixed \(x\in\mathbb R\),
\[
        f_n(x)\longrightarrow f(x).
\]
The common local linear-growth condition is satisfied because
$
        |f_n(x)|\le |x|
$
whenever \(f_n(x)\ne0\). The discontinuity-avoidance and tail
conditions are also satisfied. Nevertheless, the convergence is not
locally uniform, since
\[
        \sup_{|x|\le2}|f_n(x)-f(x)|=1.
\]

The corresponding compensator integrals are
\[
        (f_n*\nu^n)_u=1_{\{T\le u\}},
        \qquad
        (f*\nu)_u=0.
\]
Consequently,
\[
        [f_n*\nu^n-f*\nu]^{(p)}_t=1
\]
for every \(n\). Hence pointwise convergence does not control
integrand values evaluated at mark locations which vary with \(n\).

Notice also that
\[
        \delta_{1+1/n}
        \longrightarrow
        \delta_1
\]
weakly, whereas
\[
        \left\lVert
        \delta_{1+1/n}-\delta_1
        \right\rVert_{\mathrm{TV}}
        =
        2
\]
for every \(n\). Thus a total-variation convergence assumption on the
compensator kernels would be strictly stronger than the weak
mark-space convergence present in this example.
\end{rema}

\begin{rema}[Failure without uniform barycentric tail control]
The remaining assumptions do not by themselves control the
barycentric contributions of rare large jumps. Convergence in probability of the underlying
\(p\)-th power variations does not by itself control the predictable
barycentres of rare large jumps.

Fix a deterministic time \(T\in(0,t]\). For every \(n\ge1\), let
\(B_n\) be a Bernoulli random variable satisfying
\[
        \P(B_n=1)=\frac1n,
\]
and suppose that \(B_n\) is revealed at time \(T\), so that it is not
\(\mathcal F_{T-}\)-measurable. Define
\[
        X\equiv0,
        \qquad
        X_u^n:=nB_n1_{\{T\le u\}}.
\]
Then
\[
        [X^n-X]^{(p)}_t=n^pB_n.
\]
For every \(\varepsilon>0\) and all sufficiently large \(n\),
\[
\begin{aligned}
\P\left(
        [X^n-X]^{(p)}_t>\varepsilon
\right)
&=
\P(B_n=1)
\\
&=
\frac1n
\longrightarrow0.
\end{aligned}
\]
Thus
\[
        [X^n-X]^{(p)}_t\longrightarrow0
\]
in probability.

The jump measure of \(X^n\) has predictable compensator
\[
        \nu^n(\{T\},dx)
        =
        \frac1n\delta_n(dx),
\]
whereas \(\nu=0\). Set
\[
        f_n(s,x)=f(s,x)=x.
\]
The local linear-growth, local uniform convergence and
discontinuity-avoidance conditions are then satisfied. Moreover,
\[
        |f_n|*\nu^n
\]
is finite, since
\[
        \int_{\mathbb R}|x|\,\nu^n(\{T\},dx)=1.
\]

Nevertheless,
\[
\begin{aligned}
\Delta(f_n*\nu^n)_T
&=
\int_{\mathbb R}x\,\nu^n(\{T\},dx)
\\
&=
\frac1n n
=
1,
\end{aligned}
\]
while \(f*\nu=0\). Consequently,
\[
        [f_n*\nu^n-f*\nu]^{(p)}_t=1
\]
for every \(n\).

The failure is precisely detected by the barycentric tail estimate.
For every fixed \(a>0\) and every \(n>a\),
\[
        \alpha_n^{(a)}(T)=\frac1n
\]
and
\[
        \beta_n^{(a)}(T)=1.
\]
Hence
\[
        |\beta_n^{(a)}(T)|^p=1,
\]
whereas
\[
        \alpha_n^{(a)}(T)=\frac1n.
\]
There can therefore be no finite-valued predictable process
\(K^{(c)}\), independent of \(n\), such that
\[
        |\beta_n^{(a)}(T)|^p
        \le
        K_T^{(c)}\alpha_n^{(a)}(T)
\]
for all \(a\ge c\) and all \(n\).

Thus the conclusion of the theorem need not hold if
\((\mathrm{BE}_p)\) is removed while all the remaining assumptions are
left unchanged. The example does not assert that
\((\mathrm{BE}_p)\) is necessary for the convergence of every
particular sequence.
\end{rema}

The preceding examples distinguish three forms of control.
Convergence in canonical \(p\)-th power variation controls the
\(\ell^p\)-size of the realised jump discrepancies. Convergence of
predictable \(p\)-energies additionally requires an integrability
transfer from realised jumps to their compensators. Convergence of
first-order characteristics controls accumulated small-jump mass on
an \(\ell^1\)-scale. Neither of the latter two conclusions follows
from canonical \(p\)-th power-variation convergence alone.

We next give sufficient conditions under which the preceding
\(p\)-th power-variation convergence can be strengthened to ucp
convergence.

Fix \(p>1\), and let
\[
        p':=\frac{p}{p-1}.
\]
Let \(X\) and \((X^n)_{n\ge1}\) be real-valued adapted càdlàg
processes with jump measures \(\mu\) and \(\mu^n\) and predictable
compensators \(\nu\) and \(\nu^n\), respectively. Let \(f\) and
\((f_n)_{n\ge1}\) be real-valued predictable functions satisfying
\[
        f(s,0)=f_n(s,0)=0,
        \qquad n\ge1.
\]
For every \(n\ge1\), let \(\rho^n\) denote the jump measure of the
two-dimensional process \((X^n,X)\), with marks \((x,y)\), and let
\(\kappa^n\) be its predictable compensator. Define the thin predictable set
\[
\mathcal D_{\kappa^n}
:=
\left\{
(\omega,s):
\kappa^n(\omega,\{s\},\mathbb R^2)>0
\right\}.
\]
Set
\[
\kappa^{n,c}
:=
1_{\mathcal D_{\kappa^n}^c}\kappa^n,
\qquad
\rho^{n,c}
:=
1_{\mathcal D_{\kappa^n}^c}\rho^n.
\]
Since \(\mathcal D_{\kappa^n}\) is predictable,
\(\kappa^{n,c}\) is the predictable compensator of
\(\rho^{n,c}\). Denote by \(\nu^{n,c}\) and \(\nu^c\) the components
of \(\nu^n\) and \(\nu\), respectively, which are atomless in the time
variable.

Define the nonnegative predictable measure
\[
        Q_{n,p}^c(ds)
        :=
        \int_{\mathbb R^2}
        |x-y|^p\,\kappa^{n,c}(ds,dx,dy)
\]
and the signed predictable measure
\[
        B_n^c(ds)
        :=
        \int_{\mathbb R^2}
        \bigl(f_n(s,x)-f(s,y)\bigr)
        \kappa^{n,c}(ds,dx,dy).
\]

\begin{thm}
[Ucp stability for compensator integrals]
\label{thm:CMC-p-variation-compensator-ucp}

Suppose that, for every \(t<\infty\), all assumptions of
Theorem~\ref{thm:CMC-pvar-compensator-stability}
are satisfied on \([0,t]\). Set
\[
        A^n:=f_n*\nu^n,
        \qquad
        A:=f*\nu.
\]

Assume \((\mathrm{DUI}_p^c)\): for every \(t<\infty\), the family
\[
\left\{
        \int_0^t\int_{\mathbb R^2}
        |x-y|^p\,\rho^{n,c}(ds,dx,dy)
        :n\ge1
\right\}
\]
is uniformly integrable.

Assume moreover \((\mathrm{CBE}_p^c)\): almost surely,
\[
        B_n^c\ll Q_{n,p}^c.
\]
Writing
\[
        H_{n,p}^c
        :=
        \frac{dB_n^c}{dQ_{n,p}^c},
\]
assume that, for every \(t<\infty\),
\[
        \mathcal E_{n,p,t}^c
        :=
        \int_{(0,t]}
        |H_{n,p}^c(s)|^{p'}\,Q_{n,p}^c(ds)
        =
        O_{\mathbb P}(1).
\]

Finally, assume that there exists \(r\in(0,1)\) such that, for every
\(t<\infty\),
\[
\sum_{0<s\le t}
|\Delta A^n_s-\Delta A_s|^r
=
O_{\mathbb P}(1).
\]

Then,
\[
        A^n-A\longrightarrow0
        \qquad\text{ucp}.
\]

\end{thm}

\begin{proof}
Fix \(t<\infty\).

By Theorem~\ref{thm:CMC-pvar-compensator-stability},
\[
\sum_{0<s\le t}
|\Delta A^n_s-\Delta A_s|^p
\longrightarrow0
\qquad\text{in probability}.
\]

Set
\[
        \theta:=\frac{p-1}{p-r},
        \qquad
        1-\theta=\frac{1-r}{p-r}.
\]
Then
\[
        r\theta+p(1-\theta)=1.
\]
Applying Hölder's inequality with conjugate exponents
\(1/\theta\) and \(1/(1-\theta)\) gives
\[
\begin{aligned}
\sum_{0<s\le t}
|\Delta A^n_s-\Delta A_s|
&=
\sum_{0<s\le t}
\left(
|\Delta A^n_s-\Delta A_s|^r
\right)^\theta
\left(
|\Delta A^n_s-\Delta A_s|^p
\right)^{1-\theta}
\\
&\le
\left(
\sum_{0<s\le t}
|\Delta A^n_s-\Delta A_s|^r
\right)^\theta
\left(
\sum_{0<s\le t}
|\Delta A^n_s-\Delta A_s|^p
\right)^{1-\theta}.
\end{aligned}
\]
The first factor is bounded in probability, whereas the second
converges to zero in probability. Consequently,
\begin{equation}
\label{eq:p-ucp-atomic-total-variation}
\sum_{0<s\le t}
|\Delta A^n_s-\Delta A_s|
\longrightarrow0
\qquad\text{in probability}.
\end{equation}

We next control the component which is atomless in the time
variable. Set
\[
        Z_{n,p,t}^c
        :=
        \int_0^t\int_{\mathbb R^2}
        |x-y|^p\,\rho^{n,c}(ds,dx,dy).
\]
At every jump time \(s\) of \((X^n,X)\), the joint mark satisfies
\[
        x-y
        =
        \Delta X^n_s-\Delta X_s
        =
        \Delta(X^n-X)_s.
\]
Consequently,
\[
\begin{aligned}
0\le Z_{n,p,t}^c
&\le
\int_0^t\int_{\mathbb R^2}
|x-y|^p\,\rho^n(ds,dx,dy)
\\
&=
\sum_{0<s\le t}
|\Delta(X^n-X)_s|^p
\le
[X^n-X]^{(p)}_t.
\end{aligned}
\]
The assumed \(p\)-th power-variation convergence therefore gives
\[
        Z_{n,p,t}^c\longrightarrow0
        \qquad\text{in probability}.
\]
By \((\mathrm{DUI}_p^c)\) and Vitali's convergence theorem,
\begin{equation}
\label{eq:p-ucp-diffuse-realized-energy-expectation}
        \E[Z_{n,p,t}^c]\longrightarrow0.
\end{equation}

Since \(\kappa^{n,c}\) is the predictable compensator of
\(\rho^{n,c}\), the compensation identity gives
\[
\begin{aligned}
\E\left[Q_{n,p}^c((0,t])\right]
&=
\E\left[
\int_0^t\int_{\mathbb R^2}
|x-y|^p\,\kappa^{n,c}(ds,dx,dy)
\right]
\\
&=
\E\left[
\int_0^t\int_{\mathbb R^2}
|x-y|^p\,\rho^{n,c}(ds,dx,dy)
\right]
=
\E[Z_{n,p,t}^c]
\longrightarrow0.
\end{aligned}
\]
Thus
\begin{equation}
\label{eq:p-ucp-diffuse-compensator-energy}
        Q_{n,p}^c((0,t])
        \longrightarrow0
        \qquad\text{in }L^1,
\end{equation}
and hence also in probability.

By \((\mathrm{CBE}_p^c)\),
\[
        B_n^c(ds)
        =
        H_{n,p}^c(s)\,Q_{n,p}^c(ds).
\]
Hölder's inequality with conjugate exponents \(p\) and \(p'\)
therefore gives
\[
\begin{aligned}
|B_n^c|((0,t])
&=
\int_{(0,t]}
|H_{n,p}^c(s)|\,Q_{n,p}^c(ds)
\\
&\le
Q_{n,p}^c((0,t])^{1/p}
\left(
\int_{(0,t]}
|H_{n,p}^c(s)|^{p'}\,Q_{n,p}^c(ds)
\right)^{1/p'}
\\
&=
Q_{n,p}^c((0,t])^{1/p}
\bigl(\mathcal E_{n,p,t}^c\bigr)^{1/p'}.
\end{aligned}
\]
By \eqref{eq:p-ucp-diffuse-compensator-energy}, the first factor
converges to zero in probability, while the second factor is
bounded in probability. Hence
\begin{equation}
\label{eq:p-ucp-diffuse-barycenter-variation}
        |B_n^c|((0,t])
        \longrightarrow0
        \qquad\text{in probability}.
\end{equation}
We first justify the marginal identities. Let \(g\) be a nonnegative
predictable function satisfying \(g(s,0)=0\). By the definition of the
joint jump measure,
\[
        g(s,x)*\rho^n=g*\mu^n,
        \qquad
        g(s,y)*\rho^n=g*\mu.
\]
Uniqueness of predictable compensators therefore gives
\[
\int_{\mathbb R^2}g(s,x)\,\kappa^n(ds,dx,dy)
=
\int_{\mathbb R}g(s,x)\,\nu^n(ds,dx)
\]
and
\[
\int_{\mathbb R^2}g(s,y)\,\kappa^n(ds,dx,dy)
=
\int_{\mathbb R}g(s,y)\,\nu(ds,dy).
\]

Every time atom of either marginal is a time atom of \(\kappa^n\).
Conversely, at a time atom of \(\kappa^n\) at which the first marginal
has no atom, the integral of \(g(s,x)\) against that atom is zero;
the analogous assertion holds for the second marginal. Hence
restriction to \(\mathcal D_{\kappa^n}^c\) yields
\[
\int_{\mathbb R^2}g(s,x)\,\kappa^{n,c}(ds,dx,dy)
=
\int_{\mathbb R}g(s,x)\,\nu^{n,c}(ds,dx)
\]
and
\[
\int_{\mathbb R^2}g(s,y)\,\kappa^{n,c}(ds,dx,dy)
=
\int_{\mathbb R}g(s,y)\,\nu^c(ds,dy).
\]
By decomposition into positive and negative parts, the same identities
hold for the integrable signed functions used below.
For every \(u\le t\), the marginal identities for the joint
compensator, together with \(f_n(s,0)=f(s,0)=0\), give
\[
\int_0^u\int_{\mathbb R^2}
f_n(s,x)\,\kappa^{n,c}(ds,dx,dy)
=
\int_0^u\int_{\mathbb R}
f_n(s,x)\,\nu^{n,c}(ds,dx)
\]
and
\[
\int_0^u\int_{\mathbb R^2}
f(s,y)\,\kappa^{n,c}(ds,dx,dy)
=
\int_0^u\int_{\mathbb R}
f(s,y)\,\nu^c(ds,dy).
\]
Hence
\[
\begin{aligned}
B_n^c((0,u])
&=
\int_0^u\int_{\mathbb R^2}
\bigl(f_n(s,x)-f(s,y)\bigr)
\kappa^{n,c}(ds,dx,dy)
\\
&=
(f_n*\nu^{n,c})_u-(f*\nu^c)_u.
\end{aligned}
\]
It follows from \eqref{eq:p-ucp-diffuse-barycenter-variation} that
\begin{equation}
\label{eq:p-ucp-diffuse-level-convergence}
\begin{aligned}
\sup_{u\le t}
\left|
(f_n*\nu^{n,c})_u-(f*\nu^c)_u
\right|
&=
\sup_{u\le t}|B_n^c((0,u])|
\\
&\le
|B_n^c|((0,t])
\longrightarrow0
\end{aligned}
\end{equation}
in probability.

Finally, the finite-variation decomposition into the
atomless-in-time and atomic components gives, for every \(u\le t\),
\[
\begin{aligned}
A^n_u-A_u
&=
(f_n*\nu^{n,c})_u-(f*\nu^c)_u
\\
&\quad+
\sum_{0<s\le u}
\left(
        \int_{\mathbb R}f_n(s,x)\,\nu^n(\{s\},dx)
        -
        \int_{\mathbb R}f(s,x)\,\nu(\{s\},dx)
\right).
\end{aligned}
\]
Therefore,
\[
\begin{aligned}
\sup_{u\le t}|A^n_u-A_u|
&\le
\sup_{u\le t}
\left|
(f_n*\nu^{n,c})_u-(f*\nu^c)_u
\right|
\\
&\quad+
\sum_{0<s\le t}
\left|
        \int_{\mathbb R}f_n(s,x)\,\nu^n(\{s\},dx)
        -
        \int_{\mathbb R}f(s,x)\,\nu(\{s\},dx)
\right|.
\end{aligned}
\]
The first term converges to zero in probability by
\eqref{eq:p-ucp-diffuse-level-convergence}, and the second term
converges to zero in probability by
\eqref{eq:p-ucp-atomic-total-variation}. Hence
\[
        \sup_{u\le t}|A^n_u-A_u|
        \longrightarrow0
        \qquad\text{in probability}.
\]
Since \(t<\infty\) was arbitrary, it follows that
\[
        A^n-A\longrightarrow0
        \qquad\text{ucp}.
\]
\end{proof}

The preceding \(p\)-th power-variation theory applies to compensator
integrals for every \(p>1\). For compensated integrals, however, the
quadratic case is distinguished. When \(p=2\), the compensated jump
integrals, after the usual localization, are locally square-integrable
martingales, and their quadratic variation controls their maximal
processes through the Burkholder--Davis--Gundy inequalities. This
martingale structure has no direct analogue based on \(p\)-th power
variation when \(p\ne2\). We therefore restrict the compensated theory
to \(p=2\), combining the preceding stability result for the
compensator contributions with quadratic-variation control of the
realized jump contributions.

\begin{thm}[Quadratic-variation stability of compensated integrals]
\label{thm:CMC-QV-compensated-stability}
Assume that, on every finite time interval, all assumptions of
Theorem~\ref{thm:CMC-pvar-compensator-stability} are satisfied with
\(p=2\), with the same constant \(R_0\) in \({\rm (LG}_0{\rm )}\).
Assume in addition that
\[
M:=f*(\mu-\nu),
\qquad
M^n:=f_n*(\mu^n-\nu^n),
\qquad n\ge1,
\]
are well-defined purely discontinuous local martingales. Then, for every
\(t<\infty\),
\[
        [M^n-M]_t\longrightarrow0
        \qquad\text{in probability}.
\]

Moreover, there exists a Lebesgue-null set
\[
        \mathcal U\subset(0,R_0]
\]
such that, for every \(a\in(0,R_0]\setminus\mathcal U\), the following
assertions are equivalent:
\[
        M^n\longrightarrow M
        \qquad\mathrm{ucp},
\]
and, for every \(t<\infty\),
\[
\sup_{0\le u\le t}
\left|
\int_0^u\int_{\{|x|>a\}}
f_n(s,x)\,\nu^n(ds,dx)
-
\int_0^u\int_{\{|x|>a\}}
f(s,x)\,\nu(ds,dx)
\right|
\longrightarrow0
\]
in probability.

\end{thm}

\begin{proof}
Fix \(t<\infty\). By \({\rm (LG}_0{\rm )}\),
\[
        f_n(s,0)=f(s,0)=0,
        \qquad s\ge0,\quad n\ge1,
\]
outside an evanescent set. Hence, for every \(s\le t\),
\[
\begin{aligned}
\Delta(M^n-M)_s
&=
f_n(s,\Delta X^n_s)-f(s,\Delta X_s)
\\
&\quad-
\left(
\int_{\mathbb R}f_n(s,x)\,\nu^n(\{s\},dx)
-
\int_{\mathbb R}f(s,x)\,\nu(\{s\},dx)
\right).
\end{aligned}
\]
Adding and subtracting \(f_n(s,\Delta X_s)\), and using the
\(\ell^2\)-triangle inequality, we obtain
\begin{align}
[M^n-M]_t^{1/2}
&\le
\left(
\sum_{s\le t}
\left|
f_n(s,\Delta X_s)-f(s,\Delta X_s)
\right|^2
\right)^{1/2}
\nonumber\\
&\quad+
\left(
\sum_{s\le t}
\left|
f_n(s,\Delta X^n_s)-f_n(s,\Delta X_s)
\right|^2
\right)^{1/2}
\nonumber\\
&\quad+
\left(
\sum_{s\le t}
\left|
\int_{\mathbb R}f_n(s,x)\,\nu^n(\{s\},dx)
-
\int_{\mathbb R}f(s,x)\,\nu(\{s\},dx)
\right|^2
\right)^{1/2}.
\label{eq:compensated-three-term-decomposition}
\end{align}
The square of the final term on the right-hand side is precisely
\([A^n-A]_t\), with \(A^n=f_n*\nu^n\) and \(A=f*\nu\). Therefore,
by Theorem~\ref{thm:CMC-pvar-compensator-stability}, applied with
\(p=2\),
\begin{equation}
\sum_{s\le t}
\left|
\int_{\mathbb R}f_n(s,x)\,\nu^n(\{s\},dx)
-
\int_{\mathbb R}f(s,x)\,\nu(\{s\},dx)
\right|^2
\longrightarrow0
\qquad\text{in probability}.
\label{eq:compensated-compensator-contribution}
\end{equation}
It remains to treat the two additional contributions involving the
realized jumps.

Let \((n_j)\) be an arbitrary subsequence. Passing to a further
subsequence, still denoted by \((n_j)\), we may assume that
\[
        [X^{n_j}-X]_t\longrightarrow0
        \qquad\text{a.s.}
\]
We work along this further subsequence and suppress the subsequence
notation. Put
\[
        C^\circ:=C^{R_0},
        \qquad
        C_t^{\circ,*}:=\sup_{0\le u\le t}C^\circ_u.
\]
Since \(C^\circ\) is locally bounded,
\[
        \mathbb P(C_t^{\circ,*}>m)\longrightarrow0
        \qquad\text{as }m\to\infty.
\]

We first consider the fixed-jump-measure integrand contribution. Fix
\(0<r<R_0\) and \(a>r\). On \(\{C_t^{\circ,*}\le m\}\),
\begin{equation}
\begin{aligned}
&\sum_{s\le t}
\left|
f_n(s,\Delta X_s)-f(s,\Delta X_s)
\right|^2
1_{\{|\Delta X_s|\le r\}}
\le
m^2
\sum_{s\le t}
(\Delta X_s)^2
1_{\{|\Delta X_s|\le r\}}.
\end{aligned}
\label{eq:compensated-fixed-realized-small}
\end{equation}
For every fixed \(r>0\) and \(a>r\), the set
\[
        \{s\le t:r<|\Delta X_s|\le a\}
\]
is finite. At each of these times, \({\rm (LU)}\) gives
\[
        f_n(s,\Delta X_s)\longrightarrow f(s,\Delta X_s)
        \qquad\text{a.s.}
\]
Consequently,
\begin{equation}
\sum_{s\le t}
\left|
f_n(s,\Delta X_s)-f(s,\Delta X_s)
\right|^2
1_{\{r<|\Delta X_s|\le a\}}
\longrightarrow0
\qquad\text{a.s.}
\label{eq:compensated-fixed-realized-annulus}
\end{equation}
Moreover, the contribution from \(\{|\Delta X_s|>a\}\) is zero on
$
        \left\{\sup_{s\le t}|\Delta X_s|\le a\right\},
$
and
\[
        \mathbb P\left(\sup_{s\le t}|\Delta X_s|>a\right)
        \longrightarrow0
        \qquad\text{as }a\to\infty.
\]
Since
\[
\sum_{s\le t}
(\Delta X_s)^2
1_{\{|\Delta X_s|\le r\}}
\longrightarrow0
\qquad\text{a.s. as }r\downarrow0,
\]
we may first choose \(m\), then \(r\), then \(a\), and finally let
\(n\to\infty\) in \eqref{eq:compensated-fixed-realized-small} and
\eqref{eq:compensated-fixed-realized-annulus}. It follows that
\begin{equation}
\sum_{s\le t}
\left|
f_n(s,\Delta X_s)-f(s,\Delta X_s)
\right|^2
\longrightarrow0
\qquad\text{in probability}.
\label{eq:compensated-fixed-realized-convergence}
\end{equation}

We next consider the change of the realized jump structure. Fix
\[
        0<r<\min\{1,R_0/2\},
        \qquad k\ge1,
\]
and write \(L=L(r,k)\). On
\[
        \{C_t^{\circ,*}\le m\}
        \cap A_k(r)\cap B_n^{(2)}(k,r),
\]
the Threshold Isolation Principle and \({\rm (LG}_0{\rm )}\) give
\begin{align}
\sum_{s\le t}
\left|
f_n(s,\Delta X^n_s)1_{\{|\Delta X^n_s|\le L\}}
-
f_n(s,\Delta X_s)1_{\{|\Delta X_s|\le L\}}
\right|^2
&\le
2m^2
\sum_{s\le t}
(\Delta X^n_s)^2
1_{\{|\Delta X^n_s|\le L\}}
+
2m^2
\sum_{s\le t}
(\Delta X_s)^2
1_{\{|\Delta X_s|\le L\}}
\nonumber\\
&\le
4m^2[X^n-X]_t
+
6m^2
\sum_{s\le t}
(\Delta X_s)^2
1_{\{|\Delta X_s|\le r\}}.
\label{eq:compensated-realized-small-structure}
\end{align}

For \(a>L\), set
\[
\mathcal G_n(a)
:=
\left\{
\sup_{s\le t}|\Delta X_s|<a,
\quad
\sup_{s\le t}|\Delta X^n_s|<a
\right\}.
\]
On \(A_k(r)\cap B_n^{(2)}(k,r)\cap\mathcal G_n(a)\), the jumps of
\(X^n\) and \(X\) have the same classification relative to \(L\), and
the annulus contribution is
\[
\sum_{s\le t}
\left|
f_n(s,\Delta X^n_s)-f_n(s,\Delta X_s)
\right|^2
1_{\{L<|\Delta X_s|<a\}}.
\]
The sum has only finitely many nonzero terms. On the common
full-measure event on which \({\rm (LU)}\) holds simultaneously for
all \(s\le t\) and \({\rm (JC)}\) holds at every realised limiting
jump, we have, at each of these finitely many times,
\[
        \Delta X^n_s\longrightarrow\Delta X_s.
\]
Moreover,
\[
\begin{aligned}
&|f_n(s,\Delta X^n_s)-f_n(s,\Delta X_s)|
\\
&\qquad\le
2\sup_{|x|\le a}|f_n(s,x)-f(s,x)|
+
|f(s,\Delta X^n_s)-f(s,\Delta X_s)|.
\end{aligned}
\]
The first term converges to zero by \({\rm (LU)}\), and the second by
\({\rm (JC)}\). Therefore
\[
        f_n(s,\Delta X^n_s)-f_n(s,\Delta X_s)
        \longrightarrow0
\]
simultaneously at all the finitely many contributing times.
Therefore,
\begin{align}
&1_{A_k(r)\cap B_n^{(2)}(k,r)\cap\mathcal G_n(a)}
\sum_{s\le t}
\left|
f_n(s,\Delta X^n_s)1_{\{L<|\Delta X^n_s|<a\}}
-
f_n(s,\Delta X_s)1_{\{L<|\Delta X_s|<a\}}
\right|^2
\longrightarrow0
\qquad\text{a.s.}
\label{eq:compensated-realized-annulus-structure}
\end{align}
The tail contribution is zero on \(\mathcal G_n(a)\). Furthermore,
\[
\sup_{s\le t}|\Delta X^n_s|
\le
\sup_{s\le t}|\Delta X_s|
+
[X^n-X]_t^{1/2},
\]
and hence
\begin{equation}
\lim_{a\to\infty}
\limsup_{n\to\infty}
\mathbb P(\mathcal G_n(a)^c)
=0.
\label{eq:compensated-realized-tail-localization}
\end{equation}

We now remove the localizations. First choose \(m\) large. Then choose
\(r\) small enough that the second term on the right-hand side of
\eqref{eq:compensated-realized-small-structure} is small. For this fixed
\(r\), choose \(k\) so large that
$
        \mathbb P(A_k(r)^c)
$
is small. For these fixed parameters,
\[
        \mathbb P(B_n^{(2)}(k,r)^c)\longrightarrow0,
        \qquad
        [X^n-X]_t\longrightarrow0
        \quad\text{in probability}.
\]
Finally choose \(a\) large in
\eqref{eq:compensated-realized-tail-localization}, and then let
\(n\to\infty\) in
\eqref{eq:compensated-realized-annulus-structure}. The
\(\ell^2\)-triangle inequality for the small-jump, annulus and tail
regions therefore yields
\begin{equation}
\sum_{s\le t}
\left|
f_n(s,\Delta X^n_s)-f_n(s,\Delta X_s)
\right|^2
\longrightarrow0
\qquad\text{in probability}.
\label{eq:compensated-realized-structure-convergence}
\end{equation}

Combining
\eqref{eq:compensated-compensator-contribution},
\eqref{eq:compensated-fixed-realized-convergence} and
\eqref{eq:compensated-realized-structure-convergence} in
\eqref{eq:compensated-three-term-decomposition} gives
\[
        [M^n-M]_t\longrightarrow0
        \qquad\text{in probability}
\]
along the chosen further subsequence. Since the initial subsequence was
arbitrary, the subsequence criterion gives the assertion for the original
sequence.

It remains to establish the ucp equivalence and to construct the
exceptional set of truncation levels. For every \(m\in\mathbb N\), set
\[
\mathcal J_m(\omega)
:=
\left\{
|\Delta X_s(\omega)|:
0<s\le m,\ \Delta X_s(\omega)\ne0
\right\}.
\]
Since every càdlàg path has at most countably many jumps,
\(\mathcal J_m(\omega)\) is countable for every \(\omega\). Hence,
by Tonelli's theorem,
\[
\begin{aligned}
\int_{(0,R_0]}
\P\left(
\exists s\in(0,m]:
|\Delta X_s|=a
\right)\,da
&=
\E\left[
\int_{(0,R_0]}
1_{\{a\in\mathcal J_m\}}\,da
\right]
=
0.
\end{aligned}
\]
Consequently, the set
\[
\mathcal U_m
:=
\left\{
a\in(0,R_0]:
\P\left(
\exists s\in(0,m]:
|\Delta X_s|=a
\right)>0
\right\}
\]
has Lebesgue measure zero. Define
\[
        \mathcal U:=\bigcup_{m\ge1}\mathcal U_m.
\]
Then \(\mathcal U\) has Lebesgue measure zero. Moreover, if
\(a\in(0,R_0]\setminus \mathcal U\), then, for every \(t<\infty\),
\[
\P\left(
\exists s\in(0,t]:
|\Delta X_s|=a
\right)=0.
\]
Fix such an \(a\).
For every \(u\le t\), we have
\begin{align}
M^n_u-M_u
&=
\int_0^u\int_{\{|x|\le a\}}
f_n(s,x)(\mu^n-\nu^n)(ds,dx)
-
\int_0^u\int_{\{|x|\le a\}}
f(s,x)(\mu-\nu)(ds,dx)
\nonumber\\
&\quad+
\int_0^u\int_{\{|x|>a\}}
f_n(s,x)\,\mu^n(ds,dx)
-
\int_0^u\int_{\{|x|>a\}}
f(s,x)\,\mu(ds,dx)
\nonumber\\
&\quad-
\left(
\int_0^u\int_{\{|x|>a\}}
f_n(s,x)\,\nu^n(ds,dx)
-
\int_0^u\int_{\{|x|>a\}}
f(s,x)\,\nu(ds,dx)
\right).
\label{eq:compensated-ucp-decomposition}
\end{align}

We first consider the truncated compensated term in the first line.
Apply the quadratic-variation assertion already proved to the
predictable integrands
\[
        f_n(s,x)1_{\{|x|\le a\}}
        \qquad\text{and}\qquad
        f(s,x)1_{\{|x|\le a\}}.
\]
These integrands satisfy \({\rm (LG}_0{\rm )}\) and \({\rm (LU)}\).
Their only possible additional discontinuity points are
\(-a\) and \(a\), which are avoided by the admissibility of \(a\).

Moreover, since \(a\le R_0\), condition
\((\mathrm{BE}_2)\) holds for the truncated integrands with the
locally bounded envelope
$
        a^2(C^\circ)^2.
$
Indeed, using
\[
\nu^n(\{s\},\mathbb R)\le1,
\qquad
\nu(\{s\},\mathbb R)\le1
\]
up to evanescence, we obtain, for every \(c>0\),
\[
\begin{aligned}
\left|
\int_{\{c<|x|\le a\}}
f_n(s,x)\,\nu^n(\{s\},dx)
\right|^2
&\le
a^2(C^\circ_s)^2
\nu^n(\{s\},\{|x|>c\}),\\
\left|
\int_{\{c<|x|\le a\}}
f_n(s,x)\,\nu(\{s\},dx)
\right|^2
&\le
a^2(C^\circ_s)^2
\nu(\{s\},\{|x|>c\}),\\
\left|
\int_{\{c<|x|\le a\}}
f(s,x)\,\nu(\{s\},dx)
\right|^2
&\le
a^2(C^\circ_s)^2
\nu(\{s\},\{|x|>c\}).
\end{aligned}
\]
For \(c\ge a\), all three integrals vanish. Consequently,
\begin{align}
&\left[
\int_0^\cdot\int_{\{|x|\le a\}}
f_n(s,x)(\mu^n-\nu^n)(ds,dx)
-
\int_0^\cdot\int_{\{|x|\le a\}}
f(s,x)(\mu-\nu)(ds,dx)
\right]_t
\longrightarrow0
\qquad\text{in probability}.
\label{eq:truncated-compensated-qv}
\end{align}

We next upgrade
\eqref{eq:truncated-compensated-qv} to ucp. We use the following
bounded-jump consequence of the BDG inequality. Suppose that
\((L^n)_{n\ge1}\) are local martingales satisfying
$
        |\Delta L^n|\le K
$
and
\[
        [L^n]_t\longrightarrow0
        \qquad\text{in probability}.
\]
For \(\delta>0\), define
\[
        \sigma_{n,\delta}
        :=
        \inf\{u>0:[L^n]_u>\delta\}.
\]
Then
\[
[L^n]_{t\wedge\sigma_{n,\delta}}
\le
\delta
+
K^2 1_{\{\sigma_{n,\delta}\le t\}},
\]
and
\[
\P(\sigma_{n,\delta}\le t)
=
\P([L^n]_t>\delta)
\longrightarrow0.
\]
Furthermore,
\[
\begin{aligned}
\P\bigl((L^n)^*_t>\varepsilon\bigr)
&\le
\P(\sigma_{n,\delta}\le t)
+
\P\left(
\bigl((L^n)^{\sigma_{n,\delta}}\bigr)^*_t
>\varepsilon
\right).
\end{aligned}
\]
The BDG and Markov inequalities yield
\[
\begin{aligned}
\P\left(
\bigl((L^n)^{\sigma_{n,\delta}}\bigr)^*_t
>\varepsilon
\right)
&\le
\frac{C}{\varepsilon}
\E\left(
[L^n]_{t\wedge\sigma_{n,\delta}}^{1/2}
\right)
\\
&\le
\frac{C}{\varepsilon}
\left(
\sqrt{\delta}
+
K\P(\sigma_{n,\delta}\le t)
\right).
\end{aligned}
\]
It follows that
\[
\limsup_{n\to\infty}
\P\bigl((L^n)^*_t>\varepsilon\bigr)
\le
\frac{C}{\varepsilon}\sqrt{\delta}.
\]
Letting \(\delta\downarrow0\) proves \(L^n\to0\) ucp.

We apply this argument after predictable localization by
\(1_{\{C^\circ\le m\}}\). More precisely, consider
\[
1_{\{C^\circ\le m\}}\cdot
\left(
\int_0^\cdot\int_{\{|x|\le a\}}
f_n(s,x)(\mu^n-\nu^n)(ds,dx)
-
\int_0^\cdot\int_{\{|x|\le a\}}
f(s,x)(\mu-\nu)(ds,dx)
\right).
\]
Its quadratic variation is bounded by the quadratic variation in
\eqref{eq:truncated-compensated-qv}. Each of its jumps is bounded in
absolute value by \(4ma\), because the compensator atoms have total
mass at most one. Hence the bounded-jump BDG argument gives ucp
convergence for every fixed \(m\).

On the event
$
        \left\{
        \sup_{0\le s\le t}C^\circ_s\le m
        \right\},
$
the localized and unlocalized processes agree on \([0,t]\).
Since \(C^\circ\) is locally bounded, letting \(m\to\infty\) yields
\begin{align}
&\sup_{0\le u\le t}
\left|
\int_0^u\int_{\{|x|\le a\}}
f_n(s,x)(\mu^n-\nu^n)(ds,dx)
-
\int_0^u\int_{\{|x|\le a\}}
f(s,x)(\mu-\nu)(ds,dx)
\right|
\longrightarrow0
\qquad\text{in probability}.
\label{eq:truncated-compensated-ucp}
\end{align}

We now consider the realized jumps above \(a\). Let \((n_j)\) be an
arbitrary subsequence. Passing to a further subsequence, we may assume
that
\[
        [X^{n_j}-X]_t\longrightarrow0
        \qquad\text{a.s.}
\]
The pathwise argument used in
Corollary~\ref{cor:large-jump-atomic-integrands-same-fn}, together with
the admissibility of \(a\), shows that, eventually along this further
subsequence, the jumps of \(X^{n_j}\) above \(a\) occur at exactly the
same times as the jumps of \(X\) above \(a\).

There are only finitely many such times on \([0,t]\). At each of them,
\({\rm (LU)}\), \({\rm (JC)}\), and
\[
        \sup_{s\le t}
        |\Delta(X^{n_j}-X)_s|
        \le
        [X^{n_j}-X]_t^{1/2}
        \longrightarrow0
\]
give
\[
        f_{n_j}(s,\Delta X^{n_j}_s)
        \longrightarrow
        f(s,\Delta X_s).
\]
Consequently,
\begin{align}
&\sup_{0\le u\le t}
\left|
\int_0^u\int_{\{|x|>a\}}
f_n(s,x)\,\mu^n(ds,dx)
-
\int_0^u\int_{\{|x|>a\}}
f(s,x)\,\mu(ds,dx)
\right|
\longrightarrow0
\qquad\text{in probability}.
\label{eq:realized-large-jump-ucp}
\end{align}
Here the passage from the further subsequence to the original
sequence follows from the subsequence criterion.

Suppose now that, for every \(t<\infty\),
\[
\sup_{0\le u\le t}
\left|
\int_0^u\int_{\{|x|>a\}}
f_n(s,x)\,\nu^n(ds,dx)
-
\int_0^u\int_{\{|x|>a\}}
f(s,x)\,\nu(ds,dx)
\right|
\longrightarrow0
\]
in probability. Combining this with
\eqref{eq:truncated-compensated-ucp},
\eqref{eq:realized-large-jump-ucp}, and
\eqref{eq:compensated-ucp-decomposition} gives
\[
        M^n\longrightarrow M
        \qquad\mathrm{ucp}.
\]

Conversely, suppose that \(M^n\to M\) ucp. Rearranging
\eqref{eq:compensated-ucp-decomposition} and again using
\eqref{eq:truncated-compensated-ucp} and
\eqref{eq:realized-large-jump-ucp}, we obtain
\[
\sup_{0\le u\le t}
\left|
\int_0^u\int_{\{|x|>a\}}
f_n(s,x)\,\nu^n(ds,dx)
-
\int_0^u\int_{\{|x|>a\}}
f(s,x)\,\nu(ds,dx)
\right|
\longrightarrow0
\]
in probability. This proves the converse and completes the proof.
\end{proof}

\section{Structural consequences of canonical \(p\)-th
power-variation convergence}
\label{sec:structure}

This section isolates two structural consequences of canonical
\(p\)-th power-variation convergence: threshold isolation for the
realised jumps and summable control of the predictable compensator
atoms on compact annuli. Neither result involves an integrand or any
of the growth, continuity or barycentric assumptions imposed later
on compensator integrals. Apart from the standing process and
finiteness assumptions, their only asymptotic input is

$$
        [X^n-X]^{(p)}_t
        \longrightarrow0.
$$

The lemmas are stated in the almost-sure form used after subsequence
extraction. Their proofs are given in Section~\ref{sec:proof}.

The first consequence is the Threshold Isolation Principle. For a
fixed \(r>0\), the threshold levels are chosen so that their closed
uncertainty bands lie in \([r/2,r]\), have pairwise disjoint interiors,
and can overlap only at their endpoints. A càdlàg path has only
finitely many jumps whose moduli are at least \(r/2\), and any fixed
jump magnitude can belong to at most two of the uncertainty bands.
Consequently, every path meets only finitely many such bands. This is
a pathwise genericity statement for the prescribed sequence of
thresholds, rather than a conclusion obtained by choosing a spatial
level outside a Lebesgue-null set. This
pathwise separation produces thresholds which are eventually avoided
by the limiting jump magnitudes, and sufficiently small perturbations
cannot move the approximating jumps across them. The resulting
classification is used to control the small-jump contributions and
supplies the spatial separation required by the later predictable-time
matching argument.

\begin{lemma}[\(p\)-th power Threshold Isolation Principle]
\label{lem:p-thres}

Fix \(p>1\) and \(t>0\). Let \(X\) and
\(\{X^n\}_{n\ge1}\) be càdlàg
processes such that \(X\), \(X^n\), and \(X^n-X\) possess finite
canonical \(p\)-th power variation on \([0,t]\). Assume that
\[
        [X^n-X]^{(p)}_t
        \longrightarrow0
        \qquad\text{in probability}.
\]

Fix \(r>0\), and for every \(k\ge1\) define
\[
\delta(k,r):=\frac{r}{2^{k+2}},
\qquad
L(r,k):=
r\left(1-\frac{3}{2^{k+2}}\right),
\]
and
\[
A_k(r)
:=
\left\{
\omega:
\bigl||\Delta X_s(\omega)|-L(r,k)\bigr|
>
\delta(k,r)
\text{ for every }s\le t
\right\}.
\]

Then
\[
1_{A_k(r)^c}\longrightarrow0
\qquad\text{almost surely as }k\to\infty.
\]
Consequently,
\[
\P(A_k(r)^c)\longrightarrow0
\qquad\text{as }k\to\infty,
\]
and, in particular,
\[
\P\left(\bigcup_{k\ge1}A_k(r)\right)=1.
\]

Moreover, for every fixed \(k\ge1\), define
\[
B_n^{(p)}(k,r)
:=
\left\{
\sup_{s\le t}
|\Delta(X^n-X)_s|
<
\delta(k,r)
\right\}.
\]
Then
\[
\P\bigl(B_n^{(p)}(k,r)\bigr)
\longrightarrow1
\qquad\text{as }n\to\infty.
\]

On the event
\[
A_k(r)\cap B_n^{(p)}(k,r),
\]
the classification of the jumps relative to the threshold \(L(r,k)\)
is preserved. More precisely,
\[
|\Delta X_s|
\le
L(r,k)-\delta(k,r)
\quad\Longrightarrow\quad
|\Delta X^n_s|<L(r,k),
\]
and
\[
|\Delta X_s|
\ge
L(r,k)+\delta(k,r)
\quad\Longrightarrow\quad
|\Delta X^n_s|>L(r,k),
\]
for every \(s\le t\).

In particular, on \(A_k(r)\cap B_n^{(p)}(k,r)\),
\[
1_{\{|\Delta X^n_s|\le L(r,k)\}}
\le
1_{\{|\Delta X_s|\le L(r,k)\}}
\qquad\text{for every }s\le t,
\]
and
\[
\begin{aligned}
&\sum_{s\le t}
|\Delta X^n_s|^p
1_{\{|\Delta X^n_s|\le L(r,k)\}}
\le
2^{p-1}[X^n-X]^{(p)}_t
+
2^{p-1}
\sum_{s\le t}
|\Delta X_s|^p
1_{\{|\Delta X_s|\le L(r,k)\}}.
\end{aligned}
\]
\end{lemma}

The above result does not only allow one to control the small jumps, but also to isolate the large jumps in a way that permits a direct comparison of the corresponding atomic contributions.
\begin{corollary}[Pathwise convergence of threshold-separated jump atoms]
\label{cor:large-jump-atomic-integrands-same-fn}
Fix \(t>0\). Let \(X\) and \((X^n)_{n\ge1}\) be càdlàg processes
on \([0,t]\), with jump measures \(\mu\) and \(\mu^n\),
respectively, such that
\[
[X^n-X]_t \to 0
\qquad \text{a.s.}.
\]

Let \(f_n\) and \(f\) be predictable functions on \(\Omega\times[0,t]\times\mathbb R\). Assume that there exists a set
\(\Omega_0\in\mathcal F\) with \(\mathbb P(\Omega_0)=1\) such that, for every
\(\omega\in\Omega_0\), every \(s\leq t\), and every compact set
\(K\subset\mathbb R\),
\[
        \sup_{x\in K}|f_n(s,x,\omega)-f(s,x,\omega)|\to0 .
\]

Fix \(r>0\), \(k\ge1\), \(a>0\), and let \(s\le t\). Assume moreover that
\[
        f(s,\cdot,\omega)
        \text{ is continuous at }\Delta X_s(\omega)
        \qquad\text{a.s.}
\]

Then
\[
1_{A_k(r)}
\left(
\int_{L(r,k)<|x|\le a} f_{n}(s,x)\,\mu^n(\{s\},dx)
-
\int_{L(r,k)<|x|\le a} f_{n}(s,x)\,\mu(\{s\},dx)
\right)
\longrightarrow 0
\]
almost surely on the event
\[
        \{(\Delta X)^*_t<a\}.
\]
\end{corollary}

Predictable jump graphs are connected to compensator atoms by the
following support implication. If \(T\) is a predictable stopping time, then, for every \(t>0\), the
compensator identity gives
\[
1_{\{T\le t\}}
1_{\{\nu(\{T\},\mathbb R)=0\}}
|\Delta X_T|
=
0
\qquad\text{almost surely}.
\]
Consequently, outside a null set,
\[
T\le t,\quad |\Delta X_T|>0
\quad\Longrightarrow\quad
\nu(\{T\},\mathbb R)>0.
\]
The analogous implication holds for \(X^n\) and \(\nu^n\). Hence the
matching of realised jumps on predictable graphs also identifies the
corresponding supports of the compensator atoms.
The proof first establishes a local alignment of the
threshold-separated large predictable jump times and then constructs a
common pairwise-disjoint predictable exhaustion. A pathwise counting
bound for the realised large-jump indicators, combined with conditional
compensation, yields summability of the corresponding supremal
compensator masses before the localization stops.
\begin{lemma}
[Compensator mass control under \(p\)-th power-variation convergence]
\label{complemma}

Fix \(p>1\) and \(t<\infty\). Let \(X\) and
\((X^n)_{n\ge1}\) be real-valued adapted càdlàg processes, defined
on the same filtered probability space, with jump measures
\(\mu\) and \(\mu_n\) and predictable compensators \(\nu\) and
\(\nu_n\), respectively.

Assume that \(X\), \(X^n\), and \(X^n-X\) possess finite canonical
\(p\)-th power variation on \([0,t]\), and that
\[
        [X^n-X]^{(p)}_t
        \longrightarrow0
        \qquad\text{almost surely}.
\]

Set
\[
        \mathcal A
        :=
        \left\{
        (\omega,s):
        \nu(\omega,\{s\},\mathbb R)>0
        \right\},
\]
and, for \(n\ge1\),
\[
        \mathcal A_n
        :=
        \left\{
        (\omega,s):
        \nu_n(\omega,\{s\},\mathbb R)>0
        \right\}.
\]

Fix \(r>0\), and let \(L(r,k)\) and \(\delta(k,r)\) be the
thresholds supplied by Lemma~\ref{lem:p-thres}. Then there exist stopping times
\[
        \theta_{q,k,\ell,n},
        \qquad
        q,k,\ell\in\mathbb N,\quad n\ge\ell,
\]
such that the following properties hold.

\begin{enumerate}

\item For every \(q,k,\ell\in\mathbb N\) and \(n\ge\ell\), on
\[
        (0,t]\cap[0,\theta_{q,k,\ell,n}),
\]
\[
\begin{aligned}
&\left\{
s\in\mathcal A_n:
|\Delta X^n_s|>L(r,k)
\right\}
=
\left\{
s\in\mathcal A:
|\Delta X_s|>L(r,k)
\right\},
\end{aligned}
\]
almost surely.

\item For every \(r>0\) and every \(\varepsilon>0\), there exists
\(k_0\in\mathbb N\) such that, for every \(k\ge k_0\), there exists
\(\ell_0=\ell_0(k,\varepsilon)\in\mathbb N\) such that, for every
\(\ell\ge\ell_0\), there exists
\(q_0=q_0(k,\ell,\varepsilon)\in\mathbb N\) such that, for every
\(q\ge q_0\),
\[
        \sup_{n\ge\ell}
        \P\left(
        \theta_{q,k,\ell,n}\le t
        \right)
        <
        \varepsilon.
\]

\item For every fixed \(q,k,\ell\),
\[
\sum_{0<s\le t}
\sup_{n\ge\ell}
\bigl(
1_{\{\cdot<\theta_{q,k,\ell,n}\}}
\bigr)_{s-}
\left(
\nu_n(\{s\},\{|x|>L(r,k)\})
+
\nu(\{s\},\{|x|>L(r,k)\})
\right)
<
\infty
\]
almost surely.

\end{enumerate}
\end{lemma}

\begin{corollary}[Tightness of annular compensator atom masses]
\label{cor:annular-compensator-atom-mass-tightness}

Fix \(p>1\) and \(t<\infty\). Let \(X\) and
\((X^n)_{n\ge1}\) be real-valued càdlàg processes with predictable
compensators \(\nu\) and \(\nu^n\), respectively. Assume that \(X\),
\(X^n\), and \(X^n-X\) possess finite canonical \(p\)-th power variation
on \([0,t]\), and that
\[
        [X^n-X]^{(p)}_t
        \longrightarrow0
        \qquad\text{in probability}.
\]
Then, for every compact annulus
\[
        E
        =
        \{x\in\mathbb R:b\le |x|\le a\},
        \qquad
        0<b<a<\infty,
\]
the family
\[
\left\{
        \sum_{0<s\le t}\nu^n(\{s\},E)
        :
        n\ge1
\right\}
\]
is bounded in probability.
\end{corollary}

\section{Weighted closure and \(p\)-th power-variation stability of nonlinear compensator operators}\label{sec:operators}

Let \(H\) be a separable Hilbert space and let \(t>0\) be fixed.  Let
\(\nu\) be a predictable compensator measure on \([0,t]\times\mathbb R\).
For \(u\le t\), write
\[
        \nu_u:=\nu|_{[0,u]\times\mathbb R},
        \qquad
        \nu_{u-}:=\nu|_{[0,u)\times\mathbb R},
\]
and
\[
        \lambda_s:=\nu(\{s\},\cdot).
\]
We also write \(\nu^c\) for the atom-free part of \(\nu\).
For \(0\le u\le t\), let \(\mathfrak M_u\) denote the class of
nonnegative Borel measures \(\xi\) on
\([0,u]\times\mathbb R\) satisfying
\[
\xi\bigl(
[0,u]\times\{x:b\le |x|\le a\}
\bigr)
<
\infty
\]
for every \(0<b<a<\infty\). Let \(\mathcal G_u\) denote the sigma-field
on \(\mathfrak M_u\) generated by the integration maps
\[
\xi\longmapsto\int g\,d\xi,
\]
where \(g\) ranges over the bounded Borel functions whose spatial
support is contained in a compact annulus.
%\begin{defin}[Trace-QV seminorm]
%Let \(Y\) be an \(H\)-valued cadlag process on \([0,t]\) with trace
%quadratic variation. Define
%\[
%        \|Y\|_{\operatorname{qv},t}
%        :=
%        \bigl([Y]^{\operatorname{tr}}_t\bigr)^{1/2}.
%\]
%This is a seminorm on the class of \(H\)-valued cadlag processes with
%trace quadratic variation.
%\end{defin}

\begin{defin}[Operator-generated compensator process]
A possibly random compensator operator is a family of maps
\[
B_u:
\Omega\times\mathfrak M_u
\longrightarrow H,
\qquad 0\le u\le t,
\]
such that \(B_u\) is measurable from
\[
\bigl(\Omega\times\mathfrak M_u,
      \mathcal F_u\otimes\mathcal G_u\bigr)
\]
into \(H\), equipped with its Borel sigma-field.

The operator-generated compensator process associated with \(B\) and
\(\nu\) is
\[
Y_u^{B,\nu}
:=
B_u\bigl(\omega,\nu_u(\omega)\bigr),
\qquad 0\le u\le t.
\]
When no ambiguity can arise, the dependence on \(\omega\) is
suppressed.

Whenever \(Y^{B,\nu}\) is càdlàg, its atomic response at time \(s\) is
defined by
\[
\begin{aligned}
D_sB(\nu_{s-},\lambda_s)
&:=
\Delta Y_s^{B,\nu}
\\
&=
B_s\left(
\omega,
\nu_{s-}
+
\delta_s\otimes\nu(\{s\},\cdot)
\right)
-
Y_{s-}^{B,\nu}.
\end{aligned}
\]
\end{defin}

\begin{defin}[Finite first-order response to continuous compensator mass]
Let \(B\) be a compensator operator evaluated along \(\nu\), and assume
that \(Y^{B,\nu}\) is \(H\)-valued and càdlàg and that
\[
\sum_{0<s\le t}
\left\lVert
D_sB\bigl(
        \nu_{s-},
        \nu(\{s\},\cdot)
\bigr)
\right\rVert_H
<
\infty
\qquad\text{almost surely}.
\]
Define
\[
J_u^{B,\nu}
:=
\sum_{0<s\le u}
D_sB\bigl(
        \nu_{s-},
        \nu(\{s\},\cdot)
\bigr),
\qquad 0\le u\le t,
\]
and
\[
C_u^{B,\nu}
:=
Y_u^{B,\nu}-J_u^{B,\nu}.
\]

We say that \(C^{B,\nu}\) has finite first-order response to the
continuous compensator mass if there exists a nonnegative predictable
function \(a^{B,\nu}\) such that
\[
\int_{(0,t]}\int_{\mathbb R}
a^{B,\nu}(s,x)\,\nu^c(ds,dx)
<
\infty
\qquad\text{almost surely},
\]
and, outside an evanescent set,
\[
\lVert
C_v^{B,\nu}-C_u^{B,\nu}
\rVert_H
\le
\int_{(u,v]}\int_{\mathbb R}
a^{B,\nu}(s,x)\,\nu^c(ds,dx)
\]
for every \(0\le u<v\le t\).
\end{defin}

\begin{defin}[Admissible compensator operator]
\label{defin:admissible-compensator-operator}

Let \(\eta\) be a predictable compensator measure on
\([0,t]\times\mathbb R\). A compensator operator \(B\) is called admissible along \(\eta\) on
\([0,t]\) if the associated process \(Y^{B,\eta}\) is \(H\)-valued and
càdlàg, and if its atomic responses satisfy
\[
\sum_{0<s\le t}
\left\lVert
D_sB\bigl(
        \eta_{s-},
        \eta(\{s\},\cdot)
\bigr)
\right\rVert_H
<
\infty
\qquad\text{almost surely},
\]
and, upon defining
\[
J_u^{B,\eta}
:=
\sum_{0<s\le u}
D_sB\bigl(
        \eta_{s-},
        \eta(\{s\},\cdot)
\bigr),
\qquad 0\le u\le t,
\]
and
\[
C_u^{B,\eta}
:=
Y_u^{B,\eta}-J_u^{B,\eta},
\]
there exists a nonnegative predictable function \(a^{B,\eta}\) such
that
\[
\int_{(0,t]}\int_{\mathbb R}
a^{B,\eta}(s,x)\,\eta^c(ds,dx)
<
\infty
\qquad\text{almost surely},
\]
and
\[
\left\lVert
C_v^{B,\eta}-C_u^{B,\eta}
\right\rVert_H
\le
\int_{(u,v]}\int_{\mathbb R}
a^{B,\eta}(s,x)\,\eta^c(ds,dx)
\]
for every \(0\le u<v\le t\), outside an evanescent set.
\end{defin}

\begin{prop}[Atomic representation of operator \(p\)-th power variation]
\label{prop:atomic-operator-p-variation}

Let \(p>1\), and let \(B\) be an admissible compensator operator along
\(\nu\) on \([0,t]\).

Then \(C^{B,\nu}\) is continuous and of finite variation, and
\[
[Y^{B,\nu}]^{(p)}_t
=
\sum_{0<s\le t}
\lVert
D_sB\bigl(\nu_{s-},\nu(\{s\},\cdot)\bigr)
\rVert_H^p.
\]
For \(p=2\),
\[
[Y^{B,\nu}]^{\operatorname{tr}}_t
=
\sum_{0<s\le t}
\lVert
D_sB\bigl(\nu_{s-},\nu(\{s\},\cdot)\bigr)
\rVert_H^2.
\]

More generally, let \(\widetilde B\) satisfy the same assumptions along
\(\nu\). Then
\[
\begin{aligned}
&[Y^{B,\nu}-Y^{\widetilde B,\nu}]^{(p)}_t
=
\sum_{0<s\le t}
\Bigl\lVert
D_sB\bigl(\nu_{s-},\nu(\{s\},\cdot)\bigr)
-
D_s\widetilde B
\bigl(\nu_{s-},\nu(\{s\},\cdot)\bigr)
\Bigr\rVert_H^p.
\end{aligned}
\]
\end{prop}

\begin{proof}
The first-order continuous-response estimate gives
\[
\operatorname{Var}(C^{B,\nu})_t
\le
\int_{(0,t]}\int_{\mathbb R}
a^{B,\nu}(s,x)\,\nu^c(ds,dx)
<
\infty.
\]
Since the controlling measure is atomless in time,
\(C^{B,\nu}\) is continuous. A continuous finite-variation process has
zero canonical \(p\)-th power variation for every \(p>1\). Since
\[
\Delta Y_s^{B,\nu}
=
D_sB\bigl(\nu_{s-},\nu(\{s\},\cdot)\bigr),
\]
it follows that
\[
[Y^{B,\nu}]^{(p)}_t
=
\sum_{0<s\le t}
\lVert
D_sB\bigl(\nu_{s-},\nu(\{s\},\cdot)\bigr)
\rVert_H^p.
\]
The trace-quadratic-variation identity is the case \(p=2\).

If \(\widetilde B\) satisfies the same assumptions, then
$
C^{B,\nu}-C^{\widetilde B,\nu}
$
is also continuous and of finite variation. Hence it has zero
canonical \(p\)-th power variation. Moreover,
\[
\begin{aligned}
\Delta\bigl(
Y^{B,\nu}-Y^{\widetilde B,\nu}
\bigr)_s
&=
D_sB\bigl(\nu_{s-},\nu(\{s\},\cdot)\bigr)
-
D_s\widetilde B
\bigl(\nu_{s-},\nu(\{s\},\cdot)\bigr).
\end{aligned}
\]
Summing the \(p\)-th powers of these jumps gives the asserted
difference formula.
\end{proof}

\begin{defin}[\(p\)-th power-variation equivalence along a compensator]
Let \(B\) and \(\widetilde B\) satisfy the hypotheses of
Proposition~\ref{prop:atomic-operator-p-variation}. Define the
\(L^0_+\)-valued pseudometric
\[
\lVert B-\widetilde B\rVert_{p;\nu,t}
:=
\bigl(
[Y^{B,\nu}-Y^{\widetilde B,\nu}]^{(p)}_t
\bigr)^{1/p}.
\]
We write
\[
B\sim_{p;\nu,t}\widetilde B
\]
if
\[
\lVert B-\widetilde B\rVert_{p;\nu,t}=0
\qquad\text{almost surely}.
\]
By Proposition~\ref{prop:atomic-operator-p-variation}, this is
equivalent to the existence of an event of probability one on which
\[
D_sB\bigl(\nu_{s-},\nu(\{s\},\cdot)\bigr)
=
D_s\widetilde B
\bigl(\nu_{s-},\nu(\{s\},\cdot)\bigr)
\]
for every \(s\in(0,t]\).
\end{defin}

For each pair \(B,\widetilde B\), Proposition~
\ref{prop:atomic-operator-p-variation} gives, for almost every
\(\omega\),
\[
\begin{aligned}
\lVert B-\widetilde B\rVert_{p;\nu,t}(\omega)
&=
\Biggl\lVert
\left(
D_sB\bigl(\nu_{s-},\nu(\{s\},\cdot)\bigr)
-
D_s\widetilde B
\bigl(\nu_{s-},\nu(\{s\},\cdot)\bigr)
\right)_{0<s\le t}
\Biggr\rVert_{\ell^p((0,t];H)}.
\end{aligned}
\]
Thus the quotient by \(\sim_{p;\nu,t}\) is identified pathwise and
isometrically with the image of the atomic-response map in
\(\ell^p((0,t];H)\). Admissibility places each individual response
family in \(\ell^1((0,t];H)\), and this image need not be closed in
\(\ell^p((0,t];H)\). No completeness of the admissible quotient is
asserted.

\begin{defin}[Integral compensator operator]
Let \(\eta\) be a predictable compensator measure on
\([0,t]\times\mathbb R\).

A compensator operator \(I\) is called an integral compensator
operator along \(\eta\) if there exists a real-valued predictable
function
\[
        g:\Omega\times[0,t]\times\mathbb R\longrightarrow\mathbb R
\]
such that
\[
        |g|*\eta\in\mathcal V_{\mathrm{loc}}
\]
and
\[
        Y_u^{I,\eta}
        :=
        I_u(\eta_u)
        =
        \int_{(0,u]}\int_{\mathbb R}
        g(s,x)\,\eta(ds,dx),
        \qquad 0\le u\le t.
\]

Its increment at time \(s\) is
\[
        \Delta Y_s^{I,\eta}
        =
        \int_{\mathbb R}
        g(s,x)\,\eta(\{s\},dx).
\]
\end{defin}

\begin{defin}[Admissible cylinder compensator operator]
\label{defin:admissible-cylinder-compensator-operator}

Let \(p>1\), and let \(\eta\) be a predictable compensator measure
on \([0,t]\times\mathbb R\).

A compensator operator \(P\) is called an admissible cylinder
compensator operator along \(\eta\) if there exist

\begin{enumerate}
\item
integral compensator operators
\[
        I^1,\ldots,I^m
\]
along \(\eta\);

\item
a map
\[
        \Phi:\mathbb R^m\longrightarrow H
\]
such that \(\Phi(0)=0\) and, for some \(L<\infty\),
\[
        \lVert\Phi(z)-\Phi(w)\rVert_H
        \le
        L
        \left(
        \sum_{j=1}^m|z_j-w_j|^p
        \right)^{1/p},
        \qquad z,w\in\mathbb R^m;
\]

\item
a compensator operator \(C\) such that
\[
        Y_u^{C,\eta}:=C_u(\eta_u)
\]
is continuous and has finite first-order response to the
continuous compensator mass;
\end{enumerate}

such that
\[
\begin{aligned}
Y_u^{P,\eta}
&:=
P_u(\eta_u)
=
Y_u^{C,\eta}
+
\sum_{0<s\le u}
\Phi\left(
        \Delta Y_s^{I^1,\eta},
        \ldots,
        \Delta Y_s^{I^m,\eta}
\right),
\qquad 0\le u\le t.
\end{aligned}
\]

The process \(Y^{C,\eta}\) is called the continuous residual of
\(Y^{P,\eta}\).
\end{defin}

The series above is well-defined and finite. Indeed, since \(\Phi(0)=0\),
\[
\begin{aligned}
\sum_{0<s\le t}
\left\lVert
\Phi\left(
        \Delta Y_s^{I^1,\eta},
        \ldots,
        \Delta Y_s^{I^m,\eta}
\right)
\right\rVert_H
&\le
L
\sum_{0<s\le t}
\left(
        \sum_{j=1}^m
        |\Delta Y_s^{I^j,\eta}|^p
\right)^{1/p}
\\
&\le
L
\sum_{j=1}^m
\sum_{0<s\le t}
|\Delta Y_s^{I^j,\eta}|
\\
&\le
L
\sum_{j=1}^m
\int_{(0,t]}\int_{\mathbb R}
|g_j(s,x)|\,\eta(ds,dx)
<\infty
\qquad\text{almost surely}.
\end{aligned}
\]
Hence the series converges absolutely and defines a càdlàg
finite-variation process.
Consequently, every admissible cylinder compensator operator is an
admissible compensator operator in the sense of
Definition~\ref{defin:admissible-compensator-operator}.

\begin{thm}
[\(p\)-th power-variation stability of cylinder compensator operators]
\label{thm:cylinder-total-p-variation-stability}

For \(j=1,\ldots,m\), let \(I_n^j\) and \(I^j\) be integral
compensator operators along \(\nu^n\) and \(\nu\), respectively.
Assume that the hypotheses of
Theorem~\ref{thm:CMC-pvar-compensator-stability}
hold for every pair \(I_n^j,I^j\). Thus, for every
\(j=1,\ldots,m\),
\[
\left[
Y^{I_n^j,\nu^n}-Y^{I^j,\nu}
\right]^{(p)}_t
\longrightarrow0
\qquad\text{in probability}.
\]

Let
\[
        \Phi:\mathbb R^m\longrightarrow H
\]
satisfy \(\Phi(0)=0\) and
\[
        \lVert\Phi(z)-\Phi(w)\rVert_H
        \le
        L
        \left(
        \sum_{j=1}^m|z_j-w_j|^p
        \right)^{1/p},
        \qquad z,w\in\mathbb R^m.
\]

Let \(P_n\) be an admissible cylinder compensator operator along
\(\nu^n\), generated by
\[
        I_n^1,\ldots,I_n^m
\]
and \(\Phi\), and let \(P\) be an admissible cylinder compensator
operator along \(\nu\), generated by
\[
        I^1,\ldots,I^m
\]
and the same map \(\Phi\).

Then
\[
        [Y^{P_n,\nu^n}-Y^{P,\nu}]^{(p)}_t
        \longrightarrow0
        \qquad\text{in probability}.
\]
\end{thm}

\begin{proof}
The continuous residuals of \(Y^{P_n,\nu^n}\) and \(Y^{P,\nu}\)
have finite first-order continuous response. Their difference is
therefore a continuous finite-variation process and hence has zero
canonical \(p\)-th power variation.

The atomic representation theorem consequently gives
\[
\begin{aligned}
[Y^{P_n,\nu^n}-Y^{P,\nu}]^{(p)}_t
=
\sum_{0<s\le t}
\lVert
\Phi\left(
        \Delta Y_s^{I_n^1,\nu^n},
        \ldots,
        \Delta Y_s^{I_n^m,\nu^n}
\right)
-
\Phi\left(
        \Delta Y_s^{I^1,\nu},
        \ldots,
        \Delta Y_s^{I^m,\nu}
\right)
\rVert_H^p.
\end{aligned}
\]
By the Lipschitz property of \(\Phi\),
\[
\begin{aligned}
[Y^{P_n,\nu^n}-Y^{P,\nu}]^{(p)}_t
&\le
L^p
\sum_{0<s\le t}
\sum_{j=1}^m
\left|
        \Delta Y_s^{I_n^j,\nu^n}
        -
        \Delta Y_s^{I^j,\nu}
\right|^p
\\
&=
L^p
\sum_{j=1}^m
\left[
        Y^{I_n^j,\nu^n}-Y^{I^j,\nu}
\right]^{(p)}_t.
\end{aligned}
\]
Every term in the finite sum converges to zero in probability by
Theorem~\ref{thm:CMC-pvar-compensator-stability}. Therefore
\[
        [Y^{P_n,\nu^n}-Y^{P,\nu}]^{(p)}_t
        \longrightarrow0
\]
in probability.
\end{proof}

\begin{thm}[\(p\)-th power-variation closure for nonlinear compensator operators]
\label{thm:nonlinear-total-p-variation-closure}
Let \(p>1\). Let \(X\) and \((X^n)_{n\ge1}\) be real-valued
càdlàg processes, defined on the same filtered probability space, with
jump measures \(\mu\) and \(\mu^n\) and predictable compensators
\(\nu\) and \(\nu^n\), respectively. Assume that \(X\), \(X^n\), and
\(X^n-X\) possess finite canonical \(p\)-th power variation on
\([0,t]\), and that
\[
        [X^n-X]^{(p)}_t
        \longrightarrow0
\]
in probability.

Let
\[
        E=\{x\in\mathbb R:b\le |x|\le a\},
        \qquad 0<b<a<\infty.
\]
Let \(B_n\) and \(B\) be admissible compensator operators evaluated
along \(\nu^n\) and \(\nu\), respectively.

For every \(m\ge1\), let \(P_n^{(m)}\) and \(P^{(m)}\) be admissible
cylinder compensator operators evaluated along \(\nu^n\) and \(\nu\),
respectively.

Suppose that there exists a deterministic sequence
\(\varepsilon_m\downarrow0\) such that, outside an evanescent set, for
every \(n,m\) and \(s\le t\),
\[
\begin{aligned}
&\lVert
D_sB_n
\bigl(
        \nu^n_{s-},
        \nu^n(\{s\},\cdot)
\bigr)
-
D_sP_n^{(m)}
\bigl(
        \nu^n_{s-},
        \nu^n(\{s\},\cdot)
\bigr)
\rVert_H
\le
\varepsilon_m
\nu^n(\{s\},E)^{1/p},
\end{aligned}
\]
and
\[
\begin{aligned}
&\lVert
D_sB
\bigl(
        \nu_{s-},
        \nu(\{s\},\cdot)
\bigr)
-
D_sP^{(m)}
\bigl(
        \nu_{s-},
        \nu(\{s\},\cdot)
\bigr)
\rVert_H
\le
\varepsilon_m
\nu(\{s\},E)^{1/p}.
\end{aligned}
\]

Assume finally that, for every fixed \(m\),
\[
\left[
Y^{P_n^{(m)},\nu^n}
-
Y^{P^{(m)},\nu}
\right]^{(p)}_t
\longrightarrow0
\]
in probability.

Then
\[
        [Y^{B_n,\nu^n}-Y^{B,\nu}]^{(p)}_t
        \longrightarrow0
\]
in probability.
\end{thm}

\begin{proof}
Every operator-generated process appearing below is of finite
variation by admissibility. For any difference \(Z\) of two such
processes, write
\[
        \lVert Z\rVert_{(p),t}
        :=
        \bigl([Z]^{(p)}_t\bigr)^{1/p}.
\]
Its continuous part has zero canonical \(p\)-th power variation, and
therefore
\[
        \lVert Z\rVert_{(p),t}
        =
        \left(
        \sum_{0<s\le t}
        \lVert\Delta Z_s\rVert_H^p
        \right)^{1/p}.
\]
Consequently, \(\lVert\cdot\rVert_{(p),t}\) satisfies the triangle
inequality on the class of differences used below.

We use the subsequence criterion for convergence in probability. Let
\((n_j)_{j\ge1}\) be an arbitrary subsequence. Since
\[
        [X^{n_j}-X]^{(p)}_t
        \longrightarrow0
\]
in probability, there exists a further subsequence
\((n_{j_h})_{h\ge1}\) such that
\[
        [X^{n_{j_h}}-X]^{(p)}_t
        \longrightarrow0
        \qquad\text{almost surely}.
\]
In the remainder of the proof, this further subsequence is relabelled
as the original sequence.

By admissibility, the continuous residual of
\[
        Y^{B_n,\nu^n}-Y^{P_n^{(m)},\nu^n}
\]
is continuous and of finite variation. It therefore has zero canonical
\(p\)-th power variation. Hence
\[
\begin{aligned}
\lVert
Y^{B_n,\nu^n}
-
Y^{P_n^{(m)},\nu^n}
\rVert_{(p),t}^p
&=
\sum_{0<s\le t}
\Bigl\lVert
D_sB_n
\bigl(
        \nu^n_{s-},
        \nu^n(\{s\},\cdot)
\bigr)
-
D_sP_n^{(m)}
\bigl(
        \nu^n_{s-},
        \nu^n(\{s\},\cdot)
\bigr)
\Bigr\rVert_H^p
\\
&\le
\varepsilon_m^p
\sum_{0<s\le t}\nu^n(\{s\},E).
\end{aligned}
\]
Similarly,
\[
\begin{aligned}
\lVert
Y^{P^{(m)},\nu}
-
Y^{B,\nu}
\rVert_{(p),t}^p
&=
\sum_{0<s\le t}
\Bigl\lVert
D_sP^{(m)}
\bigl(
        \nu_{s-},
        \nu(\{s\},\cdot)
\bigr)
-
D_sB
\bigl(
        \nu_{s-},
        \nu(\{s\},\cdot)
\bigr)
\Bigr\rVert_H^p
\\
&\le
\varepsilon_m^p
\sum_{0<s\le t}\nu(\{s\},E).
\end{aligned}
\]

Fix \(\zeta>0\), and choose \(r\in(0,b)\). By part~(2) of Lemma~\ref{complemma}, there exist \(k,\ell,q\in\mathbb N\) such that
\[
        \sup_{n\ge\ell}
        \P(\theta_{q,k,\ell,n}\le t)
        <
        \frac{\zeta}{2}.
\]
Since
$
        L(r,k)<r<b,
$
we have
$
        E
        \subset
        \{x\in\mathbb R:|x|>L(r,k)\}.
$

For every \(j\ge\ell\), define
\[
        \gamma_{q,k,\ell,j}(s)
        :=
        \bigl(
        1_{\{\cdot<\theta_{q,k,\ell,j}\}}
        \bigr)_{s-},
\]
and set
\[
\begin{aligned}
S_{q,k,\ell}
:=
\sum_{0<s\le t}
\sup_{j\ge\ell}
\gamma_{q,k,\ell,j}(s)
\left(
        \nu^j(\{s\},E)
        +
        \nu(\{s\},E)
\right).
\end{aligned}
\]
Because
$
        E
        \subset
        \{x\in\mathbb R:|x|>L(r,k)\},
$
part~(3) of Lemma~\ref{complemma} implies that
\[
        S_{q,k,\ell}<\infty
        \qquad\text{almost surely}.
\]
Hence there exists \(M<\infty\) such that
\[
        \P(S_{q,k,\ell}>M)
        <
        \frac{\zeta}{2}.
\]

For every \(n\ge\ell\), define
\[
        G_n
        :=
        \{\theta_{q,k,\ell,n}>t\}
        \cap
        \{S_{q,k,\ell}\le M\}.
\]
Then
\[
\begin{aligned}
\sup_{n\ge\ell}\P(G_n^c)
&\le
\sup_{n\ge\ell}
\P(\theta_{q,k,\ell,n}\le t)
+
\P(S_{q,k,\ell}>M)
<
\zeta.
\end{aligned}
\]

Fix \(n\ge\ell\). On \(G_n\), we have
\[
        \gamma_{q,k,\ell,n}(s)=1,
        \qquad 0<s\le t,
\]
and therefore
\[
\begin{aligned}
\sum_{0<s\le t}
\left(
        \nu^n(\{s\},E)
        +
        \nu(\{s\},E)
\right)
&=
\sum_{0<s\le t}
\gamma_{q,k,\ell,n}(s)
\left(
        \nu^n(\{s\},E)
        +
        \nu(\{s\},E)
\right)
\\
&\le
S_{q,k,\ell}
\le
M.
\end{aligned}
\]
Consequently, on \(G_n\),
\[
\lVert
Y^{B_n,\nu^n}
-
Y^{P_n^{(m)},\nu^n}
\rVert_{(p),t}
\le
\varepsilon_mM^{1/p}
\]
and
\[
\lVert
Y^{P^{(m)},\nu}
-
Y^{B,\nu}
\rVert_{(p),t}
\le
\varepsilon_mM^{1/p}.
\]
For every \(m\ge1\), the triangle inequality gives
\[
\begin{aligned}
\lVert
Y^{B_n,\nu^n}-Y^{B,\nu}
\rVert_{(p),t}
&\le
\lVert
Y^{B_n,\nu^n}
-
Y^{P_n^{(m)},\nu^n}
\rVert_{(p),t}
\\
&+
\lVert
Y^{P_n^{(m)},\nu^n}
-
Y^{P^{(m)},\nu}
\rVert_{(p),t}
+
\lVert
Y^{P^{(m)},\nu}
-
Y^{B,\nu}
\rVert_{(p),t}.
\end{aligned}
\]
Consequently, on \(G_n\), for every \(n\ge\ell\),
\[
\begin{aligned}
\lVert
Y^{B_n,\nu^n}-Y^{B,\nu}
\rVert_{(p),t}
&\le
2\varepsilon_mM^{1/p}
+
\lVert
Y^{P_n^{(m)},\nu^n}
-
Y^{P^{(m)},\nu}
\rVert_{(p),t}.
\end{aligned}
\]

Fix \(\delta>0\). Choose \(m\) sufficiently large that
\[
        2\varepsilon_mM^{1/p}
        <
        \frac{\delta}{2}.
\]
For this fixed \(m\) and every \(n\ge\ell\),
\[
\begin{aligned}
&\P\left(
\lVert
Y^{B_n,\nu^n}-Y^{B,\nu}
\rVert_{(p),t}
>
\delta
\right)
\le
\P(G_n^c)
+
\P\left(
\lVert
Y^{P_n^{(m)},\nu^n}
-
Y^{P^{(m)},\nu}
\rVert_{(p),t}
>
\frac{\delta}{2}
\right).
\end{aligned}
\]
For every fixed \(m\),
\[
\lVert
Y^{P_n^{(m)},\nu^n}
-
Y^{P^{(m)},\nu}
\rVert_{(p),t}
\longrightarrow0
\]
in probability. Therefore
\[
\limsup_{n\to\infty}
\P\left(
\lVert
Y^{B_n,\nu^n}-Y^{B,\nu}
\rVert_{(p),t}
>
\delta
\right)
\le
\zeta.
\]
Since \(\zeta>0\) was arbitrary, we have proved, along the relabelled
further subsequence, that
\[
\lVert
Y^{B_n,\nu^n}-Y^{B,\nu}
\rVert_{(p),t}
\longrightarrow0
\]
in probability. Equivalently,
\[
        [Y^{B_n,\nu^n}-Y^{B,\nu}]^{(p)}_t
        \longrightarrow0
\]
in probability along this further subsequence.

Thus every subsequence of
\[
        [Y^{B_n,\nu^n}-Y^{B,\nu}]^{(p)}_t
\]
contains a further subsequence which converges to zero in probability.
The subsequence criterion now yields
\[
        [Y^{B_n,\nu^n}-Y^{B,\nu}]^{(p)}_t
        \longrightarrow0
\]
in probability along the original sequence.
\end{proof}

The space \(K(E)\) is naturally homeomorphic to the space of
probability measures on \(E\cup\{\partial\}\) through the map
\[
        \lambda
        \longmapsto
        \lambda+
        \bigl(1-\lambda(E)\bigr)\delta_{\partial}.
\]
Under this identification, the unweighted uniform density of
polynomial cylinder functions follows from
\cite[Lemma~2.7]{CuchieroLarssonSvalutoFerro}, or directly from the
Stone--Weierstrass theorem.

Weighted Stone--Weierstrass theorems on possibly non-compact
infinite-dimensional spaces are developed in
\cite{CuchieroSchmockerTeichmann}. In that framework an admissible
weight is strictly positive and has compact sublevel sets. The scale
\[
        \lambda\longmapsto\lambda(E)^{1/p}
\]
used here vanishes at the zero measure; after removal of the zero
measure, its sublevel sets are not compact. The following result
therefore addresses a different weighted problem: it characterises
the closure of the measure-cylinder functions in the norm which
controls their approximation error relative to
\(\lambda(E)^{1/p}\) at the zero measure.

\begin{thm}[Characterisation of the weighted cylinder closure]
\label{thm:weighted-cylinder-density}

Let
\[
        E=\{x\in\mathbb R:b\le |x|\le a\},
        \qquad 0<b<a<\infty,
\]
and let
\[
        K(E)
        :=
        \bigl\{
        \lambda\in\mathcal M_+(E):
        \lambda(E)\le1
        \bigr\},
\]
equipped with the topology of weak convergence. Let \(H\) be a
separable Hilbert space and let \(p>1\).

Denote by \(\mathcal C_{\mathrm{cyl}}(E;H)\) the class of functions
\(P:K(E)\to H\) of the form
\[
P(\lambda)
=
\Phi
\left(
        \int_E h_1(x)\,\lambda(dx),
        \ldots,
        \int_E h_d(x)\,\lambda(dx)
\right),
\]
where
\[
        d\in\mathbb N,
        \qquad
        h_1,\ldots,h_d\in C(E),
\]
and
\[
        \Phi:\mathbb R^d\longrightarrow H
\]
is globally Lipschitz and satisfies \(\Phi(0)=0\).

For a function \(G:K(E)\to H\) satisfying \(G(0)=0\), set
\[
\lVert G\rVert_{p,*}
:=
\sup_{\substack{\lambda\in K(E)\\\lambda\ne0}}
\frac{\lVert G(\lambda)\rVert_H}
     {\lambda(E)^{1/p}}.
\]
Let
\[
\mathcal X_{p,*}(E;H)
:=
\left\{
G:K(E)\to H:
G(0)=0,\
\lVert G\rVert_{p,*}<\infty
\right\}.
\]
Equipped with \(\lVert\cdot\rVert_{p,*}\),
\(\mathcal X_{p,*}(E;H)\) is a normed vector space, and
\[
        \mathcal C_{\mathrm{cyl}}(E;H)
        \subset
        \mathcal X_{p,*}(E;H).
\]
The closure below is taken in this normed space. Then
\[
\begin{aligned}
&
\overline{\mathcal C_{\mathrm{cyl}}(E;H)}
^{\,\lVert\cdot\rVert_{p,*}}
=
\left\{
F\in C(K(E);H):
F(0)=0,\
\lim_{\delta\downarrow0}
\sup_{\substack{\lambda\in K(E)\\
                 0<\lambda(E)\le\delta}}
\frac{\lVert F(\lambda)\rVert_H}
     {\lambda(E)^{1/p}}
=0
\right\}.
\end{aligned}
\]

Moreover, for every function \(F\) in the right-hand side, the
approximating cylinder functions \(P_m\) may be chosen so that, for
every \(m\), there exists \(\delta_m>0\) such that
\[
        P_m(\lambda)=0
        \qquad\text{whenever}\qquad
        \lambda(E)\le\delta_m.
\]
\end{thm}

\begin{proof}
We first prove that every function in the set on the right-hand side
belongs to the weighted cylinder closure. Let
$
        F\in C(K(E);H)
$
satisfy \(F(0)=0\) and
\[
\lim_{\delta\downarrow0}
\sup_{\substack{\lambda\in K(E)\\
                 0<\lambda(E)\le\delta}}
\frac{\lVert F(\lambda)\rVert_H}
     {\lambda(E)^{1/p}}
=
0.
\]
Write
\[
        K:=K(E).
\]
Since \(E\) is compact, the set \(K\) is compact with respect to weak
convergence of finite measures.

For \(h\in C(E)\), define
\[
        L_h(\lambda)
        :=
        \int_E h(x)\,\lambda(dx),
        \qquad \lambda\in K.
\]
Let \(\mathfrak A\) be the unital algebra generated by the functions
\(L_h\), \(h\in C(E)\). Thus every element of \(\mathfrak A\) is a
polynomial in finitely many functions of the form \(L_h\).

The algebra \(\mathfrak A\) separates the points of \(K\). Indeed, if
\(\lambda,\widetilde\lambda\in K\) and
\[
        L_h(\lambda)=L_h(\widetilde\lambda)
        \qquad\text{for every }h\in C(E),
\]
then the uniqueness part of the Riesz representation theorem gives
$
        \lambda=\widetilde\lambda.
$
Consequently, the real Stone--Weierstrass theorem yields
\[
        \overline{\mathfrak A}^{\,\lVert\cdot\rVert_\infty}
        =
        C(K;\mathbb R).
\]

We first record the corresponding \(H\)-valued approximation
consequence. Let \(\gamma>0\). Since \(F(K)\) is compact in \(H\), there
exists a finite-dimensional subspace \(H_\gamma\subset H\) such that
\[
        \sup_{\lambda\in K}
        \operatorname{dist}
        \bigl(F(\lambda),H_\gamma\bigr)
        <
        \frac{\gamma}{2}.
\]
Let
$
        e_1,\ldots,e_N
$
be an orthonormal basis of \(H_\gamma\), and let
\[
        \Pi_\gamma:H\longrightarrow H_\gamma
\]
be the orthogonal projection. It follows that
$
        \sup_{\lambda\in K}
        \lVert F(\lambda)-\Pi_\gamma F(\lambda)\rVert_H
        <
        \frac{\gamma}{2}.
$

For \(j=1,\ldots,N\), the scalar function
\[
        \lambda
        \longmapsto
        \bigl\langle F(\lambda),e_j\bigr\rangle_H
\]
belongs to \(C(K;\mathbb R)\). By Stone--Weierstrass, there exists
\(q_j\in\mathfrak A\) such that
\[
        \sup_{\lambda\in K}
        \left|
        \bigl\langle F(\lambda),e_j\bigr\rangle_H
        -
        q_j(\lambda)
        \right|
        <
        \frac{\gamma}{2\sqrt{N}}.
\]
Define
\[
        Q(\lambda)
        :=
        \sum_{j=1}^N q_j(\lambda)e_j.
\]
Then
\[
\begin{aligned}
\lVert \Pi_\gamma F(\lambda)-Q(\lambda)\rVert_H^2
&=
\sum_{j=1}^N
\left|
\bigl\langle F(\lambda),e_j\bigr\rangle_H
-
q_j(\lambda)
\right|^2
<
\frac{\gamma^2}{4}.
\end{aligned}
\]
Therefore
$
        \sup_{\lambda\in K}
        \lVert F(\lambda)-Q(\lambda)\rVert_H
        <
        \gamma.
$ Since \(q_1,\ldots,q_N\) belong to \(\mathfrak A\), there exist
\(d\in\mathbb N\), functions
\[
        h_1,\ldots,h_d\in C(E),
\]
and a polynomial map
\[
        \Psi:\mathbb R^d\longrightarrow H_\gamma
\]
such that
\[
        Q(\lambda)
        =
        \Psi
        \left(
        \int_Eh_1(x)\,\lambda(dx),
        \ldots,
        \int_Eh_d(x)\,\lambda(dx)
        \right).
\]

We now incorporate the weight at the zero measure. Fix
\(\varepsilon>0\). By the assumed behaviour of \(F\) at zero, there
exists
$
        0<\delta<\frac12
$
such that
\[
\sup_{\substack{\lambda\in K\\
                 0<\lambda(E)\le2\delta}}
\frac{\lVert F(\lambda)\rVert_H}
     {\lambda(E)^{1/p}}
<
\frac{\varepsilon}{3}.
\]
Apply the preceding \(H\)-valued approximation result with
\[
        \gamma
        :=
        \frac{\varepsilon}{3}\delta^{1/p}.
\]
We then obtain a cylinder function \(Q\) satisfying
\[
        \sup_{\lambda\in K}
        \lVert F(\lambda)-Q(\lambda)\rVert_H
        <
        \frac{\varepsilon}{3}\delta^{1/p}.
\]

Choose a Lipschitz function
\[
        \chi_\delta:\mathbb R\longrightarrow[0,1]
\]
such that
\[
        \chi_\delta(r)=0
        \quad\text{for }r\le\delta
\]
and
\[
        \chi_\delta(r)=1
        \quad\text{for }r\ge2\delta.
\]
For example, one may take \(\chi_\delta\) to be affine on
\([\delta,2\delta]\).

Set
\[
        P(\lambda)
        :=
        \chi_\delta(\lambda(E))Q(\lambda).
\]
Since
$
        \lambda(E)=\int_E1\,\lambda(dx),
$
this is again a cylinder function. Moreover,
\[
        P(\lambda)=0
        \qquad\text{whenever}\qquad
        \lambda(E)\le\delta.
\]

We verify the weighted approximation estimate. If
\[
        0<\lambda(E)\le\delta,
\]
then \(P(\lambda)=0\), and hence
\[
\frac{\lVert F(\lambda)-P(\lambda)\rVert_H}
     {\lambda(E)^{1/p}}
=
\frac{\lVert F(\lambda)\rVert_H}
     {\lambda(E)^{1/p}}
<
\frac{\varepsilon}{3}.
\]

Suppose next that
\[
        \delta<\lambda(E)\le2\delta.
\]
Using
\[
\begin{aligned}
F(\lambda)-P(\lambda)
&=
\bigl(1-\chi_\delta(\lambda(E))\bigr)F(\lambda)
+
\chi_\delta(\lambda(E))
\bigl(F(\lambda)-Q(\lambda)\bigr),
\end{aligned}
\]
we obtain
\[
\begin{aligned}
\frac{\lVert F(\lambda)-P(\lambda)\rVert_H}
     {\lambda(E)^{1/p}}
&\le
\frac{\lVert F(\lambda)\rVert_H}
     {\lambda(E)^{1/p}}
+
\frac{\lVert F(\lambda)-Q(\lambda)\rVert_H}
     {\lambda(E)^{1/p}}
\\
&<
\frac{\varepsilon}{3}
+
\frac{
        \frac{\varepsilon}{3}\delta^{1/p}
     }{
        \lambda(E)^{1/p}
     }
\le
\frac{2\varepsilon}{3}.
\end{aligned}
\]

Finally, if
\[
        \lambda(E)>2\delta,
\]
then \(P(\lambda)=Q(\lambda)\), and therefore
\[
\begin{aligned}
\frac{\lVert F(\lambda)-P(\lambda)\rVert_H}
     {\lambda(E)^{1/p}}
&<
\frac{
        \frac{\varepsilon}{3}\delta^{1/p}
     }{
        (2\delta)^{1/p}
     }
<
\frac{\varepsilon}{3}.
\end{aligned}
\]
Combining the three cases gives
\[
\sup_{\substack{\lambda\in K\\ \lambda\ne0}}
\frac{\lVert F(\lambda)-P(\lambda)\rVert_H}
     {\lambda(E)^{1/p}}
<
\varepsilon.
\]

It remains only to verify that \(P\) can be represented using a
globally Lipschitz map. Include the constant function \(1\) among the
cylinder coordinates and define
\[
\Gamma(\lambda)
:=
\left(
        \lambda(E),
        \int_Eh_1(x)\,\lambda(dx),
        \ldots,
        \int_Eh_d(x)\,\lambda(dx)
\right).
\]
The set
\[
        \Gamma(K)\subset\mathbb R^{d+1}
\]
is compact. Choose a compactly supported Lipschitz function
\[
        \zeta:\mathbb R^{d+1}\longrightarrow[0,1]
\]
which is identically one on a neighbourhood of \(\Gamma(K)\). Define
\[
\begin{aligned}
\Phi(z_0,z_1,\ldots,z_d)
:=
\chi_\delta(z_0)\,
\zeta(z_0,z_1,\ldots,z_d)\,
\Psi(z_1,\ldots,z_d).
\end{aligned}
\]
Because \(\zeta\) has compact support and \(\Psi\) is polynomial, the
map
\[
        (z_0,z_1,\ldots,z_d)
        \longmapsto
        \zeta(z_0,z_1,\ldots,z_d)\Psi(z_1,\ldots,z_d)
\]
is bounded and globally Lipschitz. Since \(\chi_\delta\) is bounded and
globally Lipschitz, it follows that \(\Phi\) is globally Lipschitz.
Furthermore,
$
        \Phi(0)=0$
because \(\chi_\delta(0)=0\). For every \(\lambda\in K\), we have \(\zeta(\Gamma(\lambda))=1\), and
therefore
\[
\begin{aligned}
P(\lambda)
&=
\Phi
\left(
        \lambda(E),
        \int_Eh_1(x)\,\lambda(dx),
        \ldots,
        \int_Eh_d(x)\,\lambda(dx)
\right).
\end{aligned}
\]
Thus \(P\) has the required cylinder representation.

Applying the construction with \(\varepsilon=1/m\), \(m\ge1\), produces
a sequence \(P_m\in\mathcal C_{\mathrm{cyl}}(E;H)\) such that
\[
\sup_{\substack{\lambda\in K\\ \lambda\ne0}}
\frac{\lVert F(\lambda)-P_m(\lambda)\rVert_H}
     {\lambda(E)^{1/p}}
\longrightarrow0.
\]
Each \(P_m\) vanishes whenever
$
        \lambda(E)\le\delta_m
$
for the corresponding cutoff level \(\delta_m>0\). This proves that
the set on the right-hand side of the asserted identity is contained
in the weighted cylinder closure.

Conversely, suppose that
\[
        F\in
        \overline{\mathcal C_{\mathrm{cyl}}(E;H)}
        ^{\,\lVert\cdot\rVert_{p,*}}.
\]
There then exist \(P_m\in\mathcal C_{\mathrm{cyl}}(E;H)\) such that
\[
        \lVert F-P_m\rVert_{p,*}\longrightarrow0.
\]
Since \(\lambda(E)\le1\) on \(K\), we have
\[
\begin{aligned}
\sup_{\lambda\in K}
\lVert F(\lambda)-P_m(\lambda)\rVert_H
&\le
\lVert F-P_m\rVert_{p,*}
\longrightarrow0.
\end{aligned}
\]
Thus \(F\) is a uniform limit of continuous functions on \(K\), and
hence
\[
        F\in C(K;H).
\]
Moreover, \(F(0)=0\).

Fix \(m\). Write
\[
P_m(\lambda)
=
\Phi_m
\left(
        \int_E h_{m,1}(x)\,\lambda(dx),
        \ldots,
        \int_E h_{m,d_m}(x)\,\lambda(dx)
\right),
\]
and let \(L_m\) be a global Lipschitz constant for \(\Phi_m\).
Since \(\Phi_m(0)=0\),
\[
\begin{aligned}
\lVert P_m(\lambda)\rVert_H
&\le
L_m
\left(
\sum_{j=1}^{d_m}
\left|
\int_Eh_{m,j}(x)\,\lambda(dx)
\right|^2
\right)^{1/2}
\\
&\le
L_m
\left(
\sum_{j=1}^{d_m}
\lVert h_{m,j}\rVert_\infty^2
\right)^{1/2}
\lambda(E).
\end{aligned}
\]
Consequently,
\[
\lim_{\delta\downarrow0}
\sup_{\substack{\lambda\in K\\
                 0<\lambda(E)\le\delta}}
\frac{\lVert P_m(\lambda)\rVert_H}
     {\lambda(E)^{1/p}}
=
0.
\]

Let \(\varepsilon>0\). Choose \(m\) so large that
\[
        \lVert F-P_m\rVert_{p,*}
        <
        \frac{\varepsilon}{2}.
\]
For this fixed \(m\), choose \(\delta>0\) so small that
\[
\sup_{\substack{\lambda\in K\\
                 0<\lambda(E)\le\delta}}
\frac{\lVert P_m(\lambda)\rVert_H}
     {\lambda(E)^{1/p}}
<
\frac{\varepsilon}{2}.
\]
Then
\[
\begin{aligned}
&
\sup_{\substack{\lambda\in K\\
                 0<\lambda(E)\le\delta}}
\frac{\lVert F(\lambda)\rVert_H}
     {\lambda(E)^{1/p}}
\le
\lVert F-P_m\rVert_{p,*}
+
\sup_{\substack{\lambda\in K\\
                 0<\lambda(E)\le\delta}}
\frac{\lVert P_m(\lambda)\rVert_H}
     {\lambda(E)^{1/p}}
<
\varepsilon.
\end{aligned}
\]
Hence
\[
\lim_{\delta\downarrow0}
\sup_{\substack{\lambda\in K\\
                 0<\lambda(E)\le\delta}}
\frac{\lVert F(\lambda)\rVert_H}
     {\lambda(E)^{1/p}}
=
0.
\]
This proves the reverse inclusion and completes the proof.
\end{proof}

\begin{rema}[The critical weight]
The exponent \(1/p\) is determined by the available annular mass
control. More generally, an approximation error of the form
\[
        \lVert F(\lambda)-P(\lambda)\rVert_H
        \le
        \varepsilon\lambda(E)^\alpha
\]
gives
\[
\sum_{0<s\le t}
\left\lVert
F(\lambda_s)-P(\lambda_s)
\right\rVert_H^p
\le
\varepsilon^p
\sum_{0<s\le t}\lambda_s(E)^{\alpha p}.
\]
Since \(0\le\lambda_s(E)\le1\), the finiteness of
\[
        \sum_{0<s\le t}\lambda_s(E)
\]
implies finiteness of the right-hand side whenever
\[
        \alpha p\ge1.
\]
Conversely, if \(\alpha p<1\), choose
\[
        1<q\le\frac1{\alpha p}
\]
and set \(m_k=k^{-q}\). Then
\[
        \sum_{k\ge1}m_k<\infty,
        \qquad
        \sum_{k\ge1}m_k^{\alpha p}=\infty.
\]
Thus \(\alpha=1/p\) is the smallest admissible exponent and corresponds
to the largest approximation-error scale that can be absorbed using
only first-order annular mass control.

Every \(P\in\mathcal C_{\mathrm{cyl}}(E;H)\) already satisfies
\[
        \lVert P(\lambda)\rVert_H
        \le
        C_P\lambda(E)
\]
for some \(C_P<\infty\). Consequently,
\[
\sum_{0<s\le t}\lVert P(\lambda_s)\rVert_H
<
\infty
\]
whenever \(\sum_{0<s\le t}\lambda_s(E)<\infty\). The cutoff used in
the proof of Theorem~\ref{thm:weighted-cylinder-density} is therefore
not needed to make the cylinder response summable. Its purpose is to
obtain uniform control of the relative approximation error near the
zero measure.
\end{rema}

The preceding density theorem treats the constant-gauge geometry on a
compact annulus. We now extend the approximation mechanism to compact
mark sets equipped with a continuous, possibly degenerate control
gauge. This permits the response scale to reflect the available
compensator energy, including near the origin. The resulting stability
theorem requires no convergence of \(\nu^n\) to \(\nu\) in a topology
on predictable measures.

\begin{defin}[Gauge-controlled cylinder responses]
\label{defin:gauge-controlled-cylinder-responses}

Let \(E\subset\mathbb R\) be compact, let \(p>1\), and let
\[
        w:E\longrightarrow[0,\infty)
\]
be continuous. Set
\[
        Z_w:=\{x\in E:w(x)=0\}
\]
and
\[
        K(E)
        :=
        \left\{
        \lambda\in\mathcal M_+(E):
        \lambda(E)\le1
        \right\}.
\]
For \(\lambda\in K(E)\), define
\[
        q_w(\lambda)
        :=
        \int_E w(x)^p\,\lambda(dx).
\]

Two measures \(\lambda,\widetilde\lambda\in K(E)\) are called
\(w\)-equivalent if
\[
        \lambda|_{E\setminus Z_w}
        =
        \widetilde\lambda|_{E\setminus Z_w}.
\]
A function
\[
        F:K(E)\longrightarrow H
\]
is called \(w\)-invariant if
\[
        F(\lambda)=F(\widetilde\lambda)
\]
whenever \(\lambda\) and \(\widetilde\lambda\) are \(w\)-equivalent.

Let \(\mathcal C_{\mathrm{cyl}}(E,w;H)\) be the class of functions
\[
P(\lambda)
=
\Phi
\left(
        \int_E h_1(x)\,\lambda(dx),
        \ldots,
        \int_E h_d(x)\,\lambda(dx)
\right),
\]
where
\[
        h_1,\ldots,h_d\in C(E),
\]
each \(h_j\) satisfies
\[
        |h_j(x)|\le c_jw(x),
        \qquad x\in E,
\]
for some \(c_j<\infty\), and
\[
        \Phi:\mathbb R^d\longrightarrow H
\]
is globally Lipschitz with \(\Phi(0)=0\).

For every \(w\)-invariant function \(G:K(E)\to H\) satisfying
\[
        G(\lambda)=0
        \qquad\text{whenever}\qquad
        q_w(\lambda)=0,
\]
define
\[
        \lVert G\rVert_{p,w,*}
        :=
        \sup_{\substack{\lambda\in K(E)\\q_w(\lambda)>0}}
        \frac{\lVert G(\lambda)\rVert_H}
             {q_w(\lambda)^{1/p}}.
\]
\end{defin}

\begin{prop}[Weighted cylinder approximation at a degenerate gauge]
\label{prop:weighted-gauge-cylinder-approximation}

Let \(E\subset\mathbb R\) be compact, let \(p>1\), and let
\(w\in C(E;[0,\infty))\). Let
\[
        F:K(E)\longrightarrow H
\]
be continuous and \(w\)-invariant.

Assume that there exists
\[
        L\in\mathcal C_{\mathrm{cyl}}(E,w;H)
\]
such that, with
\[
        R:=F-L,
\]
we have
\[
        R(\lambda)=0
        \qquad\text{whenever}\qquad
        q_w(\lambda)=0
\]
and
\[
\lim_{\delta\downarrow0}
\sup_{\substack{\lambda\in K(E)\\
                 0<q_w(\lambda)\le\delta}}
\frac{\lVert R(\lambda)\rVert_H}
     {q_w(\lambda)^{1/p}}
=
0.
\]

Then there exists a sequence
\[
        P_m\in\mathcal C_{\mathrm{cyl}}(E,w;H),
        \qquad m\ge1,
\]
such that
\[
        \varepsilon_m
        :=
        \sup_{\substack{\lambda\in K(E)\\q_w(\lambda)>0}}
        \frac{\lVert F(\lambda)-P_m(\lambda)\rVert_H}
             {q_w(\lambda)^{1/p}}
        \longrightarrow0.
\]
Moreover, \(P_m\) may be chosen in the form
\[
        P_m=L+\widetilde P_m,
\]
where, for every \(m\), there exists \(\delta_m>0\) such that
\[
        \widetilde P_m(\lambda)=0
        \qquad\text{whenever}\qquad
        q_w(\lambda)\le\delta_m.
\]
\end{prop}

\begin{proof}
Let
\[
        \lambda\equiv_w\widetilde\lambda
\]
denote \(w\)-equivalence. Indeed, suppose that
\[
        \lambda_n\equiv_w\widetilde\lambda_n,
        \qquad n\ge1,
\]
and that
\[
        \lambda_n\longrightarrow\lambda,
        \qquad
        \widetilde\lambda_n\longrightarrow\widetilde\lambda
\]
weakly. For every \(h\in C(E)\) whose support is contained in
\(E\setminus Z_w\),
\[
        \int_Eh\,d\lambda_n
        =
        \int_Eh\,d\widetilde\lambda_n.
\]
Passing to the limit gives
\[
        \int_Eh\,d\lambda
        =
        \int_Eh\,d\widetilde\lambda.
\]
Such functions separate finite measures on the locally compact space
\(E\setminus Z_w\). Hence
\[
        \lambda|_{E\setminus Z_w}
        =
        \widetilde\lambda|_{E\setminus Z_w},
\]
so the relation \(\equiv_w\) is closed.

Every coordinate function dominated by \(w\) vanishes on \(Z_w\).
Consequently, \(L\) is \(w\)-invariant. Since \(F\) is
\(w\)-invariant, so is \(R=F-L\), and \(R\) induces a continuous
function on the quotient space \(\widehat K_w\).

The relation \(\equiv_w\) is closed on the
compact space \(K(E)\). Hence the quotient
\[
        \widehat K_w:=K(E)/{\equiv_w}
\]
is compact and Hausdorff.

Let \(\lambda,\widetilde\lambda\in K(E)\) be inequivalent. Their
restrictions to \(E\setminus Z_w\) are distinct. Hence there exists
a function \(h\in C(E)\), with support contained in
\(E\setminus Z_w\), such that
\[
        \int_Eh(x)\,\lambda(dx)
        \ne
        \int_Eh(x)\,\widetilde\lambda(dx).
\]
Since the support of \(h\) is compact and disjoint from \(Z_w\),
the function \(w\) has a strictly positive minimum on that support.
Consequently, for some \(c_h<\infty\),
\[
        |h(x)|\le c_hw(x),
        \qquad x\in E.
\]
Let \(\mathfrak A_w\) be the unital scalar algebra generated by the
maps
\[
        \lambda\longmapsto\int_Eh(x)\,\lambda(dx),
\]
where
\[
        h\in C(E),
        \qquad
        \operatorname{supp}(h)\subset E\setminus Z_w.
\]
Every such \(h\) satisfies
\[
        |h|\le c_hw
\]
for some \(c_h<\infty\). The preceding separation argument shows that
\(\mathfrak A_w\) separates the points of \(\widehat K_w\).
Consequently, the Stone--Weierstrass theorem gives
\[
        \overline{\mathfrak A_w}^{\,\lVert\cdot\rVert_\infty}
        =
        C(\widehat K_w;\mathbb R).
\] A finite-rank approximation of the
compact set \(R(K(E))\subset H\), followed by scalar
Stone--Weierstrass approximation of its coordinates, shows that
\(R\) can be approximated uniformly on \(K(E)\) by \(H\)-valued
cylinder functions generated by functions dominated by \(w\).

Fix \(\varepsilon>0\). By the assumed zero-tangent condition, choose
\(\delta>0\) such that
\[
\sup_{\substack{\lambda\in K(E)\\
                 0<q_w(\lambda)\le2\delta}}
\frac{\lVert R(\lambda)\rVert_H}
     {q_w(\lambda)^{1/p}}
<
\frac{\varepsilon}{3}.
\]
Choose a cylinder function \(Q\) such that
\[
        \sup_{\lambda\in K(E)}
        \lVert R(\lambda)-Q(\lambda)\rVert_H
        <
        \frac{\varepsilon}{3}\delta^{1/p}.
\]

Let
\[
        \chi_\delta:[0,\infty)\longrightarrow[0,1]
\]
be globally Lipschitz and satisfy
\[
        \chi_\delta(u)=0,
        \qquad 0\le u\le\delta,
\]
and
\[
        \chi_\delta(u)=1,
        \qquad u\ge2\delta.
\]
Since
\[
        q_w(\lambda)
        =
        \int_Ew(x)^p\,\lambda(dx)
\]
and
\[
        w(x)^p
        \le
        \lVert w\rVert_\infty^{p-1}w(x),
\]
the map \(q_w\) is an admissible cylinder coordinate. Define
\[
        \widetilde P(\lambda)
        :=
        \chi_\delta\bigl(q_w(\lambda)\bigr)Q(\lambda).
\]
After multiplying the finite-dimensional representation of \(Q\) by
a compactly supported Lipschitz cutoff which is identically one on
the compact coordinate image of \(K(E)\), the function
\(\widetilde P\) can be represented by a globally Lipschitz
finite-dimensional map. Hence
\[
        \widetilde P\in\mathcal C_{\mathrm{cyl}}(E,w;H).
\]

If
\[
        0<q_w(\lambda)\le\delta,
\]
then \(\widetilde P(\lambda)=0\), and therefore
\[
\frac{\lVert R(\lambda)-\widetilde P(\lambda)\rVert_H}
     {q_w(\lambda)^{1/p}}
<
\frac{\varepsilon}{3}.
\]

If
\[
        \delta<q_w(\lambda)\le2\delta,
\]
then
\[
\begin{aligned}
R(\lambda)-\widetilde P(\lambda)
&=
\left(
1-\chi_\delta\bigl(q_w(\lambda)\bigr)
\right)R(\lambda)
+
\chi_\delta\bigl(q_w(\lambda)\bigr)
\bigl(R(\lambda)-Q(\lambda)\bigr),
\end{aligned}
\]
and hence
\[
\begin{aligned}
\frac{\lVert R(\lambda)-\widetilde P(\lambda)\rVert_H}
     {q_w(\lambda)^{1/p}}
&\le
\frac{\lVert R(\lambda)\rVert_H}
     {q_w(\lambda)^{1/p}}
+
\frac{\lVert R(\lambda)-Q(\lambda)\rVert_H}
     {q_w(\lambda)^{1/p}}
\\
&<
\frac{\varepsilon}{3}
+
\frac{\varepsilon}{3}
=
\frac{2\varepsilon}{3}.
\end{aligned}
\]

If
\[
        q_w(\lambda)>2\delta,
\]
then \(\widetilde P(\lambda)=Q(\lambda)\), and
\[
\begin{aligned}
\frac{\lVert R(\lambda)-\widetilde P(\lambda)\rVert_H}
     {q_w(\lambda)^{1/p}}
&<
\frac{
        \frac{\varepsilon}{3}\delta^{1/p}
     }{
        (2\delta)^{1/p}
     }
<
\frac{\varepsilon}{3}.
\end{aligned}
\]
Thus
\[
\sup_{\substack{\lambda\in K(E)\\q_w(\lambda)>0}}
\frac{\lVert R(\lambda)-\widetilde P(\lambda)\rVert_H}
     {q_w(\lambda)^{1/p}}
<
\varepsilon.
\]

Since the sum of two cylinder functions is again a cylinder function,
\[
        P:=L+\widetilde P
\]
belongs to \(\mathcal C_{\mathrm{cyl}}(E,w;H)\), and
\[
        F-P=R-\widetilde P.
\]
Applying the construction with \(\varepsilon=1/m\) proves the result.
\end{proof}

\begin{lemma}[Tightness of bounded-mark predictable \(p\)-energies]
\label{lem:bounded-mark-predictable-p-energy}

Let \(p>1\), and let \(X\) and \((X^n)_{n\ge1}\) be real-valued
càdlàg processes with jump measures \(\mu\) and \(\mu^n\) and
predictable compensators \(\nu\) and \(\nu^n\), respectively. Assume
that \(X\), \(X^n\), and \(X^n-X\) possess finite canonical \(p\)-th
power variation on \([0,t]\), and that
\[
        [X^n-X]^{(p)}_t
        \longrightarrow0
        \qquad\text{in probability}.
\]

Then, for every \(0<a<\infty\),
\[
\int_{(0,t]}\int_{\{|x|\le a\}}
|x|^p\,\nu(ds,dx)
<
\infty
\qquad\text{almost surely},
\]
and
\[
\int_{(0,t]}\int_{\{|x|\le a\}}
|x|^p\,\nu^n(ds,dx)
<
\infty
\qquad\text{almost surely}
\]
for every \(n\ge1\). Moreover, the family
\[
\left\{
\int_{(0,t]}\int_{\{|x|\le a\}}
|x|^p\,\nu^n(ds,dx)
:
n\ge1
\right\}
\]
is bounded in probability.
\end{lemma}

\begin{proof}
Fix \(0<a<\infty\). For \(q>0\) and \(n\ge1\), define
\[
\sigma_{q,n}
:=
\inf\left\{
u>0:
[X^n]^{(p)}_u\ge q
\right\}
\]
and
\[
\gamma_{q,n}(s)
:=
\bigl(1_{\{\cdot<\sigma_{q,n}\}}\bigr)_{s-}
=
1_{\{s\le\sigma_{q,n}\}}.
\]
The process \(\gamma_{q,n}\) is predictable. By the compensation
identity,
\[
\begin{aligned}
&\mathbb E\left[
\int_{(0,t]}\int_{\{|x|\le a\}}
\gamma_{q,n}(s)|x|^p\,\nu^n(ds,dx)
\right]
=
\mathbb E\left[
\sum_{0<s\le t}
\gamma_{q,n}(s)
|\Delta X^n_s|^p
1_{\{|\Delta X^n_s|\le a\}}
\right].
\end{aligned}
\]
The terms with \(s<\sigma_{q,n}\) have total sum at most
$
        [X^n]^{(p)}_{\sigma_{q,n}-}
        \le q.
$
At \(s=\sigma_{q,n}\), there is at most one additional term, and the
truncation bounds it by \(a^p\). Consequently,
\[
\mathbb E\left[
\int_{(0,t]}\int_{\{|x|\le a\}}
\gamma_{q,n}(s)|x|^p\,\nu^n(ds,dx)
\right]
\le
q+a^p.
\]

Set
\[
I_{n,a}
:=
\int_{(0,t]}\int_{\{|x|\le a\}}
|x|^p\,\nu^n(ds,dx).
\]
On \(\{\sigma_{q,n}>t\}\), the stopped and unstopped integrals agree.
Hence, for every \(M>0\),
\[
\begin{aligned}
\P(I_{n,a}>M)
&\le
\P\left(
\int_{(0,t]}\int_{\{|x|\le a\}}
\gamma_{q,n}(s)|x|^p\,\nu^n(ds,dx)
>
M
\right)
+
\P(\sigma_{q,n}\le t)
\\
&\le
\frac{q+a^p}{M}
+
\P\left([X^n]^{(p)}_t\ge q\right).
\end{aligned}
\]

By Lemma~\ref{lem:canonical-p-power-basic},
\[
[X^n]^{(p)}_t
\le
2^{p-1}
\left(
[X^n-X]^{(p)}_t+[X]^{(p)}_t
\right).
\]
It follows that
$
        \{[X^n]^{(p)}_t:n\ge1\}
$
is bounded in probability. Given \(\zeta>0\), choose \(q\) so large
that
\[
\sup_{n\ge1}
\P\left([X^n]^{(p)}_t\ge q\right)
<
\frac{\zeta}{2},
\]
and then choose \(M\) so large that
\[
        \frac{q+a^p}{M}
        <
        \frac{\zeta}{2}.
\]
Thus
\[
        \sup_{n\ge1}\P(I_{n,a}>M)<\zeta,
\]
which proves the asserted boundedness in probability.

For a fixed \(n\), letting \(M\to\infty\) in the preceding estimate
gives
\[
\P(I_{n,a}=\infty)
\le
\P\left([X^n]^{(p)}_t\ge q\right).
\]
Letting \(q\to\infty\) proves that \(I_{n,a}<\infty\) almost surely.

Applying the same argument to \(X\) and \(\nu\) proves
\[
\int_{(0,t]}\int_{\{|x|\le a\}}
|x|^p\,\nu(ds,dx)
<
\infty
\qquad\text{almost surely}.
\]
\end{proof}

\begin{defin}[Locally linear compact control gauge]
\label{def:locally-linear-compact-gauge}

A pair \((E,w)\) is called a locally linear compact control gauge if
\(E\subset\mathbb R\) is compact,
\[
        w:E\longrightarrow[0,\infty)
\]
is continuous, and, whenever \(0\in E\),
\[
        w(0)=0
\]
and there exist \(c_w<\infty\) and \(r_w>0\) such that
\[
        w(x)\le c_w|x|,
        \qquad x\in E,\quad |x|\le r_w.
\]
\end{defin}

\begin{defin}[Stable global cylinder component]
\label{def:stable-global-cylinder-component}

Fix \(p>1\). Let \(X\) and \((X^n)_{n\ge1}\) be real-valued
càdlàg processes, defined on the same filtered probability space,
with predictable compensators \(\nu\) and \(\nu^n\), respectively.

A pair of \(H\)-valued atom-response families
\[
        \bigl((\Pi_n)_{n\ge1},\Pi\bigr)
\]
is called a stable global cylinder component relative to
\((X^n,X;\nu^n,\nu)\) on \([0,t]\) if there exist \(r\ge1\),

\begin{enumerate}

\item integral compensator operators
\[
        I_n^j
        \quad\text{along }\nu^n,
        \qquad
        I^j
        \quad\text{along }\nu,
        \qquad j=1,\ldots,r,
\]
such that the integrands generating every pair \(I_n^j,I^j\)
satisfy the hypotheses of
Theorem~\ref{thm:CMC-pvar-compensator-stability};

\item a map
\[
        \Phi:\mathbb R^r\longrightarrow H
\]
with \(\Phi(0)=0\) such that, for some \(L_\Phi<\infty\),
\[
\lVert\Phi(z)-\Phi(z')\rVert_H
\le
L_\Phi
\left(
        \sum_{j=1}^r|z_j-z'_j|^p
\right)^{1/p},
\qquad z,z'\in\mathbb R^r;
\]

\end{enumerate}

such that
\[
\begin{aligned}
\Pi_n(s)
&=
\Phi\left(
        \Delta Y_s^{I_n^1,\nu^n},
        \ldots,
        \Delta Y_s^{I_n^r,\nu^n}
\right),
\\
\Pi(s)
&=
\Phi\left(
        \Delta Y_s^{I^1,\nu},
        \ldots,
        \Delta Y_s^{I^r,\nu}
\right).
\end{aligned}
\]
\end{defin}

\begin{defin}[Compact-gauge residual perturbation]
\label{def:compact-gauge-residual-perturbation}

Let \((E,w)\) be a locally linear compact control gauge. A family
\[
        \bigl((R_n)_{n\ge1},R\bigr),
        \qquad
        R_n,R:K(E)\longrightarrow H,
\]
is called a compact-gauge residual perturbation if the following
conditions hold.

The function \(R\) is continuous and \(w\)-invariant,
\[
        R(\lambda)=0
        \qquad\text{whenever}\qquad
        q_w(\lambda)=0,
\]
and
\[
\lim_{\delta\downarrow0}
\sup_{\substack{\lambda\in K(E)\\
                 0<q_w(\lambda)\le\delta}}
\frac{\lVert R(\lambda)\rVert_H}
     {q_w(\lambda)^{1/p}}
=
0.
\]

Every \(R_n\) is \(w\)-invariant and satisfies
\[
        R_n(\lambda)=R(\lambda)
        \qquad\text{whenever}\qquad
        q_w(\lambda)=0.
\]
Moreover, with
\[
        \delta_n
        :=
        \sup_{\substack{\lambda\in K(E)\\q_w(\lambda)>0}}
        \frac{
        \lVert R_n(\lambda)-R(\lambda)\rVert_H
        }{
        q_w(\lambda)^{1/p}
        },
\]
we have
\[
        \delta_n<\infty,
        \qquad
        \delta_n\longrightarrow0.
\]
\end{defin}

\begin{thm}
[Nonlinear \(p\)-th power-variation stability with a global cylinder
component and a compact-gauge residual]
\label{thm:compact-gauge-nonlinear-pvar-stability}

Let \(p>1\) and \(t<\infty\). Let \(X\) and
\((X^n)_{n\ge1}\) be real-valued càdlàg processes, defined on the same
filtered probability space, with jump measures \(\mu,\mu^n\) and
predictable compensators \(\nu,\nu^n\), respectively. Assume that
\(X\), \(X^n\), and \(X^n-X\) possess finite canonical \(p\)-th power
variation on \([0,t]\), and that
\[
        [X^n-X]^{(p)}_t
        \longrightarrow0
        \qquad\text{in probability}.
\]

Let \((E,w)\) be a locally linear compact control gauge and assume
that
\[
\P\left(
        \exists s\le t:
        \Delta X_s\in\partial E\setminus\{0\}
\right)
=
0.
\]
For
\[
        \eta\in\{\nu,\nu^1,\nu^2,\ldots\},
\]
write
\[
        \lambda_s^\eta
        :=
        \eta(\{s\},\cdot)|_E.
\]

Let
\[
        \bigl((\Pi_n)_{n\ge1},\Pi\bigr)
\]
be a stable global cylinder component relative to
\((X^n,X;\nu^n,\nu)\) on \([0,t]\), and let
\[
        \bigl((R_n)_{n\ge1},R\bigr)
\]
be a compact-gauge residual perturbation on \(K(E)\).

For every \(n\ge1\), let \(B_n\) be an admissible compensator
operator along \(\nu^n\) satisfying
\[
Y_u^{B_n,\nu^n}
=
C_u^{B_n,\nu^n}
+
\sum_{0<s\le u}
\left[
        \Pi_n(s)+R_n(\lambda_s^{\nu^n})
\right],
\qquad 0\le u\le t.
\]
Let \(B\) be an admissible compensator operator along \(\nu\)
satisfying
\[
Y_u^{B,\nu}
=
C_u^{B,\nu}
+
\sum_{0<s\le u}
\left[
        \Pi(s)+R(\lambda_s^\nu)
\right],
\qquad 0\le u\le t.
\]

Then
\[
        [Y^{B_n,\nu^n}-Y^{B,\nu}]^{(p)}_t
        \longrightarrow0
        \qquad\text{in probability}.
\]
\end{thm}

\begin{proof}
Apply
Proposition~\ref{prop:weighted-gauge-cylinder-approximation}
to \(R\), with the cylinder term there chosen to be identically zero.
There exist
\[
        S_m\in\mathcal C_{\mathrm{cyl}}(E,w;H),
        \qquad m\ge1,
\]
such that
\[
        \varepsilon_m
        :=
        \sup_{\substack{\lambda\in K(E)\\q_w(\lambda)>0}}
        \frac{
        \lVert R(\lambda)-S_m(\lambda)\rVert_H
        }{
        q_w(\lambda)^{1/p}
        }
        \longrightarrow0.
\]
Moreover, for every \(m\), there exists \(d_m>0\) such that
\[
        S_m(\lambda)=0
        \qquad\text{whenever}\qquad
        q_w(\lambda)\le d_m.
\]

We first establish the required control of the accumulated gauge
energies. The assumptions on \(w\), together with compactness of
\(E\), imply that there exists \(c_{E,w}<\infty\) such that
\[
        w(x)\le c_{E,w}|x|,
        \qquad x\in E.
\]
Indeed, if \(0\in E\), one may take
\[
c_{E,w}
:=
\max\left\{
        c_w,
        \frac{\lVert w\rVert_\infty}{r_w}
\right\},
\]
whereas, if \(0\notin E\), one may take
\[
c_{E,w}
:=
\frac{\lVert w\rVert_\infty}
     {\operatorname{dist}(E,0)}.
\]

Set
\[
        a_E:=1+\sup_{x\in E}|x|.
\]
For every
\[
        \eta\in\{\nu,\nu^1,\nu^2,\ldots\},
\]
we have
\[
\begin{aligned}
\int_{(0,t]}\int_E
w(x)^p\,\eta(ds,dx)
&\le
c_{E,w}^p
\int_{(0,t]}\int_{\{|x|\le a_E\}}
|x|^p\,\eta(ds,dx).
\end{aligned}
\]
Lemma~\ref{lem:bounded-mark-predictable-p-energy} therefore gives
\[
\int_{(0,t]}\int_E
w(x)^p\,\eta(ds,dx)
<
\infty
\qquad\text{almost surely}
\]
for every such \(\eta\), and shows that
\[
\left\{
\int_{(0,t]}\int_E
w(x)^p\,\nu^n(ds,dx)
:
n\ge1
\right\}
\]
is bounded in probability.

Define
\[
        \mathcal Q_{n,t}^w
        :=
        \sum_{0<s\le t}
        q_w(\lambda_s^{\nu^n})
\]
and
\[
        \mathcal Q_t^w
        :=
        \sum_{0<s\le t}
        q_w(\lambda_s^\nu).
\]
Since
\[
\mathcal Q_{n,t}^w
\le
\int_{(0,t]}\int_E
w(x)^p\,\nu^n(ds,dx),
\]
the family
$
        \{\mathcal Q_{n,t}^w:n\ge1\}
$
is bounded in probability. Similarly,
\[
        \mathcal Q_t^w<\infty
        \qquad\text{almost surely}.
\]

We next verify that the additional coordinates occurring in \(S_m\)
can be combined with the global cylinder coordinates. The construction
in
Proposition~\ref{prop:weighted-gauge-cylinder-approximation}
shows that these additional coordinates may be chosen to be either
\[
        \lambda\longmapsto
        \int_Eh(x)\,\lambda(dx),
\]
where \(h\in C(E)\) has compact support in \(E\setminus Z_w\), or
\[
        \lambda\longmapsto
        q_w(\lambda)
        =
        \int_Ew(x)^p\,\lambda(dx).
\]

If \(h\) has compact support in \(E\setminus Z_w\), then its support
is separated from zero. Indeed, if \(0\in E\), then \(0\in Z_w\),
whereas, if \(0\notin E\), compactness of \(E\) gives
$
        \operatorname{dist}(E,0)>0.
$
Consequently, for every nonzero such \(h\),
\[
        d_h
        :=
        \operatorname{dist}(\operatorname{supp}(h),0)
        >
        0,
\]
and
\[
        2|h(x)|1_E(x)
        \le
        \frac{2\lVert h\rVert_\infty}{d_h}|x|,
        \qquad x\in\mathbb R.
\]
Thus the pair of identical integrands
\[
        g_n(x)=g(x)=h(x)1_E(x)
\]
satisfies \({\rm (LG_0)}\), with a deterministic locally bounded
coefficient. Its finite-variation integrability follows from the
local finiteness of compensator measures away from zero.
Since the same function is used along \(\nu^n\) and \(\nu\),
condition \({\rm (LU)}\) holds identically. Its only possible
discontinuities lie in \(\partial E\setminus\{0\}\), so
\({\rm (JC)}\) follows from the boundary assumption.

For the coordinate \(q_w\), take
\[
        g_n(x)=g(x)=w(x)^p1_E(x).
\]
If \(0\in E\), then, for \(x\in E\) with \(|x|\le r_w\),
\[
        2w(x)^p
        \le
        2c_w^p|x|^p
        \le
        2c_w^pr_w^{p-1}|x|.
\]
For \(x\in E\) with \(|x|>r_w\),
\[
        2w(x)^p
        \le
        \frac{2\lVert w\rVert_\infty^p}{r_w}|x|.
\]
Hence
\[
        2w(x)^p1_E(x)
        \le
        C_w|x|,
        \qquad x\in\mathbb R,
\]
where
\[
        C_w
        :=
        2\max\left\{
        c_w^pr_w^{p-1},
        \frac{\lVert w\rVert_\infty^p}{r_w}
        \right\}.
\]
If \(0\notin E\), then
\[
        2w(x)^p1_E(x)
        \le
        \frac{2\lVert w\rVert_\infty^p}
             {\operatorname{dist}(E,0)}
        |x|,
        \qquad x\in\mathbb R.
\]
Thus \({\rm (LG_0)}\) holds in both cases.

The finite-variation integrability of \(w^p1_E\) follows from
Lemma~\ref{lem:bounded-mark-predictable-p-energy} and
\[
        w(x)\le c_{E,w}|x|,
        \qquad x\in E.
\]
Condition \({\rm (LU)}\) again holds identically. The extension
\(w^p1_E\) is continuous outside \(\partial E\), and it is continuous
at zero whenever \(0\in\partial E\), because \(w(0)=0\). Hence its
only possible discontinuities lie in
\(\partial E\setminus\{0\}\), and \({\rm (JC)}\) follows from the
boundary assumption.

Finally, let \(g\) denote either \(h1_E\) or \(w^p1_E\). For every
compensator atom, every \(a>0\), and every
\[
        \eta\in\{\nu,\nu^1,\nu^2,\ldots\},
\]
we have
\[
\begin{aligned}
\left|
\int_{\{|x|>a\}}
g(x)\,\eta(\{s\},dx)
\right|^p
&\le
\lVert g\rVert_\infty^p
\eta(\{s\},\{|x|>a\})^p
\\
&\le
\lVert g\rVert_\infty^p
\eta(\{s\},\{|x|>a\}),
\end{aligned}
\]
because every compensator atom is a subprobability measure. The same
estimate applies to the mixed integral in which the integrand used
along \(\nu^n\) is integrated against \(\nu\). Therefore
\({\rm (BE}_p{\rm )}\) holds for every additional coordinate.

It follows that every coordinate used to represent \(S_m\), when
extended by zero outside \(E\), satisfies the hypotheses of
Theorem~\ref{thm:CMC-pvar-compensator-stability}. The same is true,
by assumption, for the coordinate pairs
\[
        I_n^j,I^j,
        \qquad j=1,\ldots,r.
\]

Fix \(m\). Write
\[
S_m(\lambda)
=
\Psi_m\left(
        \int_Eh_{m,1}(x)\,\lambda(dx),
        \ldots,
        \int_Eh_{m,d_m}(x)\,\lambda(dx)
\right),
\]
where \(\Psi_m(0)=0\) and \(\Psi_m\) is globally Lipschitz. By
equivalence of norms on \(\mathbb R^{d_m}\), it is globally Lipschitz
with respect to the \(\ell^p\)-norm as well.

Define
\[
\Theta_m:
\mathbb R^{r+d_m}
\longrightarrow H
\]
by
\[
        \Theta_m(z,y):=\Phi(z)+\Psi_m(y).
\]
The map \(\Theta_m\) is globally Lipschitz with respect to the
\(\ell^p\)-norm and satisfies
$
        \Theta_m(0,0)=0.
$

Let \(Q_n^{(m)}\) be the cylinder compensator operator along
\(\nu^n\) with identically zero continuous residual and atomic
response
\[
        \Pi_n(s)+S_m(\lambda_s^{\nu^n})
\]
at time \(s\). Thus
\[
Y_u^{Q_n^{(m)},\nu^n}
=
\sum_{0<s\le u}
\left[
        \Pi_n(s)+S_m(\lambda_s^{\nu^n})
\right].
\]
Let \(Q^{(m)}\) be the corresponding cylinder compensator operator
along \(\nu\), also with identically zero continuous residual, and
with
\[
Y_u^{Q^{(m)},\nu}
=
\sum_{0<s\le u}
\left[
        \Pi(s)+S_m(\lambda_s^\nu)
\right].
\]

The global cylinder responses are absolutely summable. Indeed,
\[
\begin{aligned}
\sum_{0<s\le t}\lVert\Pi_n(s)\rVert_H
&\le
L_\Phi
\sum_{0<s\le t}
\left(
        \sum_{j=1}^r
        \left|
        \Delta Y_s^{I_n^j,\nu^n}
        \right|^p
\right)^{1/p}
\\
&\le
L_\Phi
\sum_{j=1}^r
\sum_{0<s\le t}
\left|
        \Delta Y_s^{I_n^j,\nu^n}
\right|
<
\infty
\end{aligned}
\]
almost surely, and the same estimate holds for \(\Pi\). The preceding
coordinate-integrability estimates give the corresponding absolute
summability for the responses generated by \(S_m\). Hence
\(Q_n^{(m)}\) and \(Q^{(m)}\) are well defined admissible cylinder
compensator operators.

Theorem~\ref{thm:cylinder-total-p-variation-stability} now gives
\[
\left[
Y^{Q_n^{(m)},\nu^n}
-
Y^{Q^{(m)},\nu}
\right]^{(p)}_t
\longrightarrow0
\qquad\text{in probability}
\]
for every fixed \(m\).

Both \(B_n\) and \(Q_n^{(m)}\) are admissible along \(\nu^n\).
Their continuous residuals have finite variation, and their
difference is continuous. Therefore
Proposition~\ref{prop:atomic-operator-p-variation} gives
\[
\begin{aligned}
&
\left[
Y^{B_n,\nu^n}
-
Y^{Q_n^{(m)},\nu^n}
\right]^{(p)}_t
=
\sum_{0<s\le t}
\left\lVert
R_n(\lambda_s^{\nu^n})
-
S_m(\lambda_s^{\nu^n})
\right\rVert_H^p.
\end{aligned}
\]
For every \(\lambda\in K(E)\) with \(q_w(\lambda)>0\),
\[
\begin{aligned}
\lVert R_n(\lambda)-S_m(\lambda)\rVert_H
&\le
\lVert R_n(\lambda)-R(\lambda)\rVert_H
+
\lVert R(\lambda)-S_m(\lambda)\rVert_H
\\
&\le
(\delta_n+\varepsilon_m)q_w(\lambda)^{1/p}.
\end{aligned}
\]
The same difference vanishes when \(q_w(\lambda)=0\). Consequently,
\[
\left(
\left[
Y^{B_n,\nu^n}
-
Y^{Q_n^{(m)},\nu^n}
\right]^{(p)}_t
\right)^{1/p}
\le
(\delta_n+\varepsilon_m)
(\mathcal Q_{n,t}^w)^{1/p}.
\]
Similarly,
\[
\left(
\left[
Y^{B,\nu}
-
Y^{Q^{(m)},\nu}
\right]^{(p)}_t
\right)^{1/p}
\le
\varepsilon_m
(\mathcal Q_t^w)^{1/p}.
\]

The \(\ell^p\)-triangle inequality therefore yields
\[
\begin{aligned}
&
\left(
[Y^{B_n,\nu^n}-Y^{B,\nu}]^{(p)}_t
\right)^{1/p}
\le
(\delta_n+\varepsilon_m)
(\mathcal Q_{n,t}^w)^{1/p}
+
\left(
\left[
Y^{Q_n^{(m)},\nu^n}
-
Y^{Q^{(m)},\nu}
\right]^{(p)}_t
\right)^{1/p}
+
\varepsilon_m
(\mathcal Q_t^w)^{1/p}.
\end{aligned}
\]

Fix \(\varepsilon>0\) and \(\zeta>0\). Since
$
        \{\mathcal Q_{n,t}^w:n\ge1\}
$
is bounded in probability and
\[
        \mathcal Q_t^w<\infty
        \qquad\text{almost surely},
\]
there exists \(M<\infty\) such that
\[
        \sup_{n\ge1}
        \P(\mathcal Q_{n,t}^w>M)
        <
        \zeta
\]
and
\[
        \P(\mathcal Q_t^w>M)
        <
        \zeta.
\]
Choose \(m\) sufficiently large that
\[
        \varepsilon_mM^{1/p}
        <
        \frac{\varepsilon}{6}.
\]
For this fixed \(m\), choose \(N\) sufficiently large that
\[
        \delta_nM^{1/p}
        <
        \frac{\varepsilon}{6},
        \qquad n\ge N.
\]
Then, on
\[
\{\mathcal Q_{n,t}^w\le M\}
\cap
\{\mathcal Q_t^w\le M\},
\]
the first and third terms in the preceding triangle inequality are
bounded by \(\varepsilon/3\) and \(\varepsilon/6\), respectively.
Hence, for \(n\ge N\),
\[
\begin{aligned}
&\P\left(
\left(
[Y^{B_n,\nu^n}-Y^{B,\nu}]^{(p)}_t
\right)^{1/p}
>
\varepsilon
\right)
\le
\P(\mathcal Q_{n,t}^w>M)
+
\P(\mathcal Q_t^w>M)
+
\P\left(
\left(
\left[
Y^{Q_n^{(m)},\nu^n}
-
Y^{Q^{(m)},\nu}
\right]^{(p)}_t
\right)^{1/p}
>
\frac{\varepsilon}{2}
\right).
\end{aligned}
\]
Taking the upper limit as \(n\to\infty\) gives
\[
\limsup_{n\to\infty}
\P\left(
\left(
[Y^{B_n,\nu^n}-Y^{B,\nu}]^{(p)}_t
\right)^{1/p}
>
\varepsilon
\right)
\le
2\zeta.
\]
Since \(\zeta>0\) was arbitrary,
\[
        [Y^{B_n,\nu^n}-Y^{B,\nu}]^{(p)}_t
        \longrightarrow0
        \qquad\text{in probability}.
\]
\end{proof}

\begin{rema}
Two specialisations of the theorem are immediate.

First, take \(H=\mathbb R\), \(r=1\),
\[
        \Phi(z)=z,
        \qquad
        R_n=R=0,
\]
and let \(I_n^1\) and \(I^1\) be generated by \(f_n\) and \(f\),
respectively. Choose
\[
        E=\{0\},
        \qquad
        w\equiv0,
\]
and take
\[
        B_n=I_n^1,
        \qquad
        B=I^1.
\]
The conclusion is precisely
\[
        [f_n*\nu^n-f*\nu]^{(p)}_t
        \longrightarrow0
        \qquad\text{in probability}.
\]
Thus Theorem~\ref{thm:CMC-pvar-compensator-stability} is the
zero-residual, one-coordinate special case.

Second, suppose that
\[
        F=L+R,
        \qquad
        L\in\mathcal C_{\mathrm{cyl}}(E,w;H),
\]
and set
\[
        R_n:=F_n-L.
\]
Represent \(L\) by its finitely many integral coordinates and use
these coordinates for the global cylinder component. Then the theorem
gives the compact-gauge nonlinear stability result with atom
responses \(F_n\) and \(F\).
\end{rema}

On a compact annulus the control gauge may be chosen constant. In that
case the gauge energy of a compensator atom is simply its mass on the
annulus. Moreover, the gauge has no zero set on the mark space, and
the local condition at the origin is vacuous. The general stability
theorem therefore yields the following special case.

\begin{corollary}
[Nonlinear \(p\)-th power-variation stability on a compact annulus]
\label{cor:nonlinear-pvar-stability-compact-annulus}

Let \(p>1\) and \(t<\infty\). Let \(X\) and
\((X^n)_{n\ge1}\) be real-valued càdlàg processes, defined on the same
filtered probability space, with jump measures \(\mu\) and \(\mu^n\)
and predictable compensators \(\nu\) and \(\nu^n\), respectively.
Assume that \(X\), \(X^n\), and \(X^n-X\) possess finite canonical
\(p\)-th power variation on \([0,t]\), and that
\[
        [X^n-X]^{(p)}_t
        \longrightarrow0
        \qquad\text{in probability}.
\]

Let
\[
        E=\{x\in\mathbb R:b\le |x|\le a\},
        \qquad 0<b<a<\infty,
\]
and assume that
\[
\P\left(
        \exists s\le t:
        |\Delta X_s|\in\{b,a\}
\right)
=
0.
\]
Set
\[
K(E)
:=
\left\{
        \lambda\in\mathcal M_+(E):
        \lambda(E)\le1
\right\}.
\]

Let \(F:K(E)\to H\) be continuous, and let
\(F_n:K(E)\to H\), \(n\ge1\), satisfy
\[
        F_n(0)=F(0)=0.
\]
Assume that
\[
\lim_{\delta\downarrow0}
\sup_{\substack{\lambda\in K(E)\\
                 0<\lambda(E)\le\delta}}
\frac{\lVert F(\lambda)\rVert_H}
     {\lambda(E)^{1/p}}
=
0
\]
and, with
\[
\delta_n
:=
\sup_{\substack{\lambda\in K(E)\\\lambda\ne0}}
\frac{
        \lVert F_n(\lambda)-F(\lambda)\rVert_H
}{
        \lambda(E)^{1/p}
},
\]
that
\[
        \delta_n<\infty,
        \qquad
        \delta_n\longrightarrow0.
\]

For
\[
        \eta\in\{\nu,\nu^1,\nu^2,\ldots\},
\]
set
\[
        \lambda_s^\eta
        :=
        \eta(\{s\},\cdot)|_E.
\]
Assume that, for every such \(\eta\), there exists an admissible
compensator operator \(B^{F,\eta}\) along \(\eta\) whose atomic
responses satisfy
\[
D_sB^{F,\eta}
\bigl(
        \eta_{s-},
        \eta(\{s\},\cdot)
\bigr)
=
F(\lambda_s^\eta),
\qquad 0<s\le t.
\]
For every \(n\ge1\), assume that there exists an admissible compensator
operator \(B^{F_n,\nu^n}\) along \(\nu^n\) whose atomic responses
satisfy
\[
D_sB^{F_n,\nu^n}
\bigl(
        \nu^n_{s-},
        \nu^n(\{s\},\cdot)
\bigr)
=
F_n(\lambda_s^{\nu^n}),
\qquad 0<s\le t.
\]

Then
\[
\left[
Y^{B^{F_n,\nu^n},\nu^n}
-
Y^{B^{F,\nu},\nu}
\right]^{(p)}_t
\longrightarrow0
\qquad\text{in probability}.
\]
\end{corollary}

\begin{proof}
Apply
Theorem~\ref{thm:compact-gauge-nonlinear-pvar-stability}
with
\[
        w(x):=1,
        \qquad x\in E.
\]
Since \(0\notin E\), the local condition on the gauge at the origin is
vacuous, and
\[
        Z_w=\varnothing.
\]
Hence the \(w\)-invariance requirements are automatic. Moreover, for
every \(\lambda\in K(E)\),
\[
        q_w(\lambda)
        =
        \int_Ew(x)^p\,\lambda(dx)
        =
        \lambda(E).
\]

Use a stable global cylinder component that is identically zero.
More precisely, take
\(r=1\), let \(I_n^1\) and \(I^1\) be generated by the identically
zero integrand, and take
\[
        \Phi\equiv0.
\]
Then
\[
        \Pi_n(s)=\Pi(s)=0,
        \qquad 0<s\le t.
\]

Set
\[
        R_n:=F_n,
        \qquad
        R:=F.
\]
The weighted zero-tangent condition for \(R\) is exactly
\[
\lim_{\delta\downarrow0}
\sup_{\substack{\lambda\in K(E)\\
                 0<\lambda(E)\le\delta}}
\frac{\lVert F(\lambda)\rVert_H}
     {\lambda(E)^{1/p}}
=
0.
\]
Since \(q_w(\lambda)=0\) if and only if \(\lambda=0\), the required
agreement on the zero-energy set follows from
\[
        F_n(0)=F(0)=0.
\]
Finally,
\[
\sup_{\substack{\lambda\in K(E)\\q_w(\lambda)>0}}
\frac{\lVert R_n(\lambda)-R(\lambda)\rVert_H}
     {q_w(\lambda)^{1/p}}
=
\delta_n
\longrightarrow0.
\]

Take
\[
        B_n:=B^{F_n,\nu^n},
        \qquad
        B:=B^{F,\nu}.
\]
Their atomic responses have the form required in
Theorem~\ref{thm:compact-gauge-nonlinear-pvar-stability}. The theorem
therefore gives
\[
\left[
Y^{B^{F_n,\nu^n},\nu^n}
-
Y^{B^{F,\nu},\nu}
\right]^{(p)}_t
\longrightarrow0
\]
in probability.
\end{proof}
\begin{corollary}[Nonlinear stability on a compact neighbourhood of
the origin]
\label{cor:origin-neighbourhood-nonlinear-pvar-stability}

In the setting of
Theorem~\ref{thm:compact-gauge-nonlinear-pvar-stability}, let
\[
        E=[-a,a],
        \qquad
        w(x)=|x|,
\]
and assume that
\[
\P\left(
        \exists s\le t:
        |\Delta X_s|=a
\right)
=
0.
\]
For
\[
        \eta\in\{\nu,\nu^1,\nu^2,\ldots\},
\]
write
\[
        \lambda_s^\eta
        :=
        \eta(\{s\},\cdot)|_E.
\]

Let
\[
        \bigl((F_n)_{n\ge1},F\bigr)
\]
be a compact-gauge residual perturbation on \(K(E)\). For every
\(n\ge1\), let \(B_n\) be an admissible compensator operator along
\(\nu^n\) satisfying
\[
D_sB_n
\bigl(
        \nu^n_{s-},
        \nu^n(\{s\},\cdot)
\bigr)
=
F_n(\lambda_s^{\nu^n}),
\qquad 0<s\le t.
\]
Let \(B\) be an admissible compensator operator along \(\nu\)
satisfying
\[
D_sB
\bigl(
        \nu_{s-},
        \nu(\{s\},\cdot)
\bigr)
=
F(\lambda_s^\nu),
\qquad 0<s\le t.
\]

Then
\[
        [Y^{B_n,\nu^n}-Y^{B,\nu}]^{(p)}_t
        \longrightarrow0
        \qquad\text{in probability}.
\]
\end{corollary}

\begin{proof}
For \(E=[-a,a]\) and \(w(x)=|x|\), the pair \((E,w)\) is a locally
linear compact control gauge and
\[
        q_w(\lambda)
        =
        \int_{[-a,a]}|x|^p\,\lambda(dx).
\]
Moreover,
\[
        \partial E\setminus\{0\}
        =
        \{-a,a\},
\]
so the boundary condition in
Theorem~\ref{thm:compact-gauge-nonlinear-pvar-stability} is precisely
the assumed condition.

Use a stable global cylinder component that is identically zero.
More precisely, take
\(r=1\), let \(I_n^1\) and \(I^1\) be generated by the identically
zero integrand, and take
\[
        \Phi\equiv0.
\]
Then
\[
        \Pi_n(s)=\Pi(s)=0,
        \qquad 0<s\le t.
\]
Apply Theorem~\ref{thm:compact-gauge-nonlinear-pvar-stability} with
\[
        R_n:=F_n,
        \qquad
        R:=F.
\]
The asserted convergence follows.
\end{proof}

\subsection{Gauge-controlled nonlinear responses}
\label{subsec:nonlinear-compensator-examples}

Fix \(p>1\), let \(E\subset\mathbb R\) be compact, and let
\[
        w:E\longrightarrow[0,\infty)
\]
be continuous. For \(\lambda\in K(E)\), define the finite measure
\[
        \vartheta_\lambda(dx)
        :=
        w(x)^p\,\lambda(dx).
\]
Its total mass is
\[
        \vartheta_\lambda(E)
        =
        q_w(\lambda).
\]
The map
\[
        \lambda\longmapsto\vartheta_\lambda
\]
is continuous for the weak topologies on finite measures. Moreover, it
depends only on the restriction of \(\lambda\) to \(E\setminus Z_w\).
Consequently, every continuous functional of \(\vartheta_\lambda\) is
\(w\)-invariant.

For every \(h\in C(E)\),
\[
\int_Eh(x)\,\vartheta_\lambda(dx)
=
\int_Eh(x)w(x)^p\,\lambda(dx).
\]
Since
\[
        |h(x)|w(x)^p
        \le
        \lVert h\rVert_\infty
        \lVert w\rVert_\infty^{p-1}w(x),
\]
this is an admissible gauge-controlled cylinder coordinate.

\subsubsection{Gauge-weighted covariance and cumulants}

Let
\[
        \varphi:E\longrightarrow\mathbb R^d
\]
be continuous and bounded, and let
\[
        H:=\mathbb R_{\mathrm{sym}}^{d\times d}
\]
be equipped with the Frobenius norm. Define
\[
\begin{aligned}
R_{\operatorname{cov}}^w(\lambda)
&:=
\int_E
\varphi(x)\varphi(x)^\top\,\vartheta_\lambda(dx)
-
\left(
\int_E\varphi(x)\,\vartheta_\lambda(dx)
\right)
\left(
\int_E\varphi(x)\,\vartheta_\lambda(dx)
\right)^\top.
\end{aligned}
\]
This response is generated by the finitely many coordinates
\[
\int_E
\varphi_i(x)w(x)^p\,\lambda(dx)
\]
and
\[
\int_E
\varphi_i(x)\varphi_j(x)w(x)^p\,\lambda(dx),
\qquad 1\le i,j\le d.
\]
It is therefore an exact polynomial gauge-controlled cylinder
response. The polynomial map may be multiplied by a globally
Lipschitz cutoff which is identically one on the compact coordinate
image of \(K(E)\).

Furthermore,
\[
\begin{aligned}
\lVert R_{\operatorname{cov}}^w(\lambda)\rVert_H
&\le
\lVert\varphi\rVert_\infty^2
\left(
q_w(\lambda)+q_w(\lambda)^2
\right)
\\
&\le
\lVert\varphi\rVert_\infty^2
\left(
1+\lVert w\rVert_\infty^p
\right)
q_w(\lambda).
\end{aligned}
\]
Hence, whenever
\[
        \sum_{0<s\le t}q_w(\lambda_s^\eta)<\infty,
\]
the series
\[
        \sum_{0<s\le t}
        R_{\operatorname{cov}}^w(\lambda_s^\eta)
\]
converges absolutely.

The same construction applies to higher cumulant responses. For
example, when \(\varphi:E\to\mathbb R\), define
\[
\begin{aligned}
R_3^w(\lambda)
&:=
\int_E\varphi(x)^3\,\vartheta_\lambda(dx)
-
3
\left(
\int_E\varphi(x)\,\vartheta_\lambda(dx)
\right)
\left(
\int_E\varphi(x)^2\,\vartheta_\lambda(dx)
\right)
+
2
\left(
\int_E\varphi(x)\,\vartheta_\lambda(dx)
\right)^3.
\end{aligned}
\]
This is again an exact polynomial gauge-controlled cylinder response.
More generally, each fixed-order cumulant polynomial generated by
\[
        \int_E\varphi(x)^j\,\vartheta_\lambda(dx),
        \qquad j=1,\ldots,r,
\]
defines such a response. Since
\[
        q_w(\lambda)\le\lVert w\rVert_\infty^p,
\]
every fixed-order cumulant response is bounded by a constant multiple
of \(q_w(\lambda)\).

\subsubsection{Gauge-weighted kernel discrepancies}

Let \(\rho\in\mathcal P(E)\), and let
\[
        k:E\times E\longrightarrow\mathbb R
\]
be a bounded continuous positive definite kernel. Define
\[
\begin{aligned}
R_{k,\rho}^w(\lambda)
:=
\iint_{E\times E}
k(x,y)
\bigl(
\vartheta_\lambda-q_w(\lambda)\rho
\bigr)(dx)
\bigl(
\vartheta_\lambda-q_w(\lambda)\rho
\bigr)(dy).
\end{aligned}
\]
If \(q_w(\lambda)>0\), then
\[
R_{k,\rho}^w(\lambda)
=
q_w(\lambda)^2
\operatorname{MMD}_k^2
\left(
\frac{\vartheta_\lambda}{q_w(\lambda)},
\rho
\right),
\]
and the response is defined to be zero when \(q_w(\lambda)=0\).
It satisfies
\[
        0
        \le
        R_{k,\rho}^w(\lambda)
        \le
        4\lVert k\rVert_\infty q_w(\lambda)^2.
\]
Consequently,
\[
\frac{
        |R_{k,\rho}^w(\lambda)|
}{
        q_w(\lambda)^{1/p}
}
\le
4\lVert k\rVert_\infty
q_w(\lambda)^{2-1/p}
\longrightarrow0
\]
as \(q_w(\lambda)\downarrow0\).

If
\[
        k_N(x,y)
        =
        \sum_{j=1}^{d_N}c_{N,j}h_{N,j}(x)h_{N,j}(y),
\]
then
\[
\begin{aligned}
R_{k_N,\rho}^w(\lambda)
&=
\sum_{j=1}^{d_N}c_{N,j}
\Bigg(
\int_Eh_{N,j}(x)w(x)^p\,\lambda(dx)
-
q_w(\lambda)
\int_Eh_{N,j}(x)\,\rho(dx)
\Bigg)^2.
\end{aligned}
\]
The attainable coordinate vectors form a compact subset of a
finite-dimensional space. Hence the polynomial coordinate map may be
multiplied by a compactly supported Lipschitz cutoff which equals one
on a neighbourhood of this set. Thus finite-rank kernel discrepancies
are exact gauge-controlled cylinder responses in the sense of
Definition~\ref{defin:gauge-controlled-cylinder-responses}.

For general \(k\), there exist continuous finite-rank positive
definite kernels
\[
        k_N(x,y)
        =
        \sum_{j=1}^{d_N}
        c_{N,j}h_{N,j}(x)h_{N,j}(y),
        \qquad c_{N,j}\ge0,
\]
such that
\[
        \lVert k_N-k\rVert_\infty\longrightarrow0.
\]
For example, they may be obtained by finite-dimensional approximation
of the canonical feature map of \(k\).
Then
\[
\left|
R_{k_N,\rho}^w(\lambda)
-
R_{k,\rho}^w(\lambda)
\right|
\le
4\lVert k_N-k\rVert_\infty q_w(\lambda)^2,
\]
and therefore
\[
\begin{aligned}
&
\sup_{\substack{\lambda\in K(E)\\q_w(\lambda)>0}}
\frac{
\left|
R_{k_N,\rho}^w(\lambda)
-
R_{k,\rho}^w(\lambda)
\right|
}{
q_w(\lambda)^{1/p}
}
\le
4\lVert k_N-k\rVert_\infty
\lVert w\rVert_\infty^{2p-1}
\longrightarrow0.
\end{aligned}
\]
Hence the general kernel response belongs to the weighted closure of
the gauge-controlled cylinder responses.

Moreover,
\[
\begin{aligned}
\sum_{0<s\le t}
R_{k,\rho}^w(\lambda_s^\eta)
&\le
4\lVert k\rVert_\infty
\sum_{0<s\le t}q_w(\lambda_s^\eta)^2
\\
&\le
4\lVert k\rVert_\infty
\lVert w\rVert_\infty^p
\sum_{0<s\le t}q_w(\lambda_s^\eta),
\end{aligned}
\]
so the associated pure-jump process is well defined whenever the
gauge energies are summable.

\subsubsection{Gauge-weighted regularised transport responses}

Let \(\rho\in\mathcal P(E)\) and fix \(\varepsilon>0\). Define
\[
R_{\varepsilon,\rho}^w(\lambda)
:=
q_w(\lambda)^2
W_2^2
\left(
\frac{
        \vartheta_\lambda+\varepsilon\rho
}{
        q_w(\lambda)+\varepsilon
},
\rho
\right).
\]
This response is continuous and \(w\)-invariant. Since \(E\) is
compact,
\[
        0
        \le
        R_{\varepsilon,\rho}^w(\lambda)
        \le
        \operatorname{diam}(E)^2q_w(\lambda)^2.
\]
It follows that
\[
\frac{
R_{\varepsilon,\rho}^w(\lambda)
}{
q_w(\lambda)^{1/p}
}
\le
\operatorname{diam}(E)^2
q_w(\lambda)^{2-1/p}
\longrightarrow0
\]
as \(q_w(\lambda)\downarrow0\). Proposition~
\ref{prop:weighted-gauge-cylinder-approximation}, applied with
\(L=0\), therefore places this response in the weighted closure of
the gauge-controlled cylinder class.

The corresponding atom series is absolutely convergent because
\[
\begin{aligned}
\sum_{0<s\le t}
R_{\varepsilon,\rho}^w(\lambda_s^\eta)
&\le
\operatorname{diam}(E)^2
\sum_{0<s\le t}
q_w(\lambda_s^\eta)^2
\\
&\le
\operatorname{diam}(E)^2
\lVert w\rVert_\infty^p
\sum_{0<s\le t}
q_w(\lambda_s^\eta).
\end{aligned}
\]

Consequently, under the process, boundary, coordinate-integrability
and admissibility assumptions of
Theorem~\ref{thm:compact-gauge-nonlinear-pvar-stability}, each of the
preceding responses, with zero continuous residual, satisfies
\[
\left[
Y^{F,\nu^n}
-
Y^{F,\nu}
\right]^{(p)}_t
\longrightarrow0
\]
in probability. For the covariance and cumulant responses, their
finitely many integral coordinates form the stable global cylinder
component and the residual is zero. For the general kernel and
regularised transport responses, the global cylinder component may be
taken to be zero and the entire response is used as the compact-gauge
residual.

\subsubsection{A global covariance component with a local transport
residual}

The preceding examples use either an exact cylinder response or a
compact-gauge residual. The two components may also be combined in a
single operator. We give an example in which the cylinder component
depends on the entire compensator atom, whereas the residual is a
nonlinear functional of its full restriction to a compact mark set,
rather than of finitely many prescribed moments.

Let
\[
        G=(g_1,\ldots,g_d):
        \mathbb R\longrightarrow\mathbb R^d
\]
be continuous and bounded, with \(G(0)=0\), and suppose that, for some
\(c_G<\infty\),
\[
        \lVert G(x)\rVert_{\mathbb R^d}
        \le
        c_G\bigl(|x|^p\wedge1\bigr),
        \qquad x\in\mathbb R.
\]
For a finite measure \(\Lambda\) on \(\mathbb R\) satisfying
\[
        \Lambda(\mathbb R)\le1,
\]
define
\[
\begin{aligned}
\Gamma_G(\Lambda)
&:=
\int_{\mathbb R}
G(x)G(x)^\top\,\Lambda(dx)
-
\left(
\int_{\mathbb R}G(x)\,\Lambda(dx)
\right)
\left(
\int_{\mathbb R}G(x)\,\Lambda(dx)
\right)^\top.
\end{aligned}
\]
Thus \(\Gamma_G(\Lambda)\) is the covariance response of the bounded
feature map \(G\), with the missing mass of the subprobability measure
placed at the zero feature.

The response \(\Gamma_G\) is generated by the finitely many global
coordinates
\[
        \int_{\mathbb R}g_i(x)\,\Lambda(dx),
        \qquad 1\le i\le d,
\]
and
\[
        \int_{\mathbb R}g_i(x)g_j(x)\,\Lambda(dx),
        \qquad 1\le i,j\le d.
\]
The set of attainable coordinate vectors is bounded, and its closure
is therefore compact in the finite-dimensional coordinate space. The
covariance polynomial is Lipschitz on a neighbourhood of this compact
set. Multiplying it by a compactly supported globally Lipschitz cutoff
which equals one on that neighbourhood gives a globally Lipschitz
cylinder representation without changing \(\Gamma_G\) on compensator
atoms.

Every coordinate integrand \(g_i\) and \(g_ig_j\) is continuous,
bounded and dominated by a constant multiple of
\[
        |x|^p\wedge1.
\]
Its finite-variation compensator integrability therefore follows from
Lemma~\ref{lem:bounded-mark-predictable-p-energy} on bounded mark sets
and from local finiteness of the compensator away from the origin.
The local growth and jump-continuity conditions of
Theorem~\ref{thm:CMC-pvar-compensator-stability} are immediate. Since
the same coordinate integrands are used along \(\nu^n\) and \(\nu\),
condition \({\rm (LU)}\) holds identically. Moreover, if \(h\) denotes
any of these coordinate integrands and
\(\Lambda\) is any atom of \(\nu^n\) or \(\nu\), then
\[
\left|
\int_{\{|x|>a\}}h(x)\,\Lambda(dx)
\right|^p
\le
\lVert h\rVert_\infty^p
\Lambda(\{|x|>a\})^p
\le
\lVert h\rVert_\infty^p
\Lambda(\{|x|>a\}),
\]
because \(\Lambda(\mathbb R)\le1\). This estimate applies both to the
integrals against \(\nu^n\) and \(\nu\), including the mixed integral
in which the integrand used along \(\nu^n\) is integrated against
\(\nu\). Hence \({\rm (BE}_p{\rm )}\) holds with a deterministic constant.
Thus every pair of corresponding coordinate integrals satisfies the
hypotheses of
Theorem~\ref{thm:CMC-pvar-compensator-stability}.

Now let \((E,w)\) be a locally linear compact control gauge, let
\(\rho\in\mathcal P(E)\), and fix \(\varepsilon>0\). Take
\[
        H
        :=
        \mathbb R_{\mathrm{sym}}^{d\times d}
        \oplus\mathbb R,
\]
equipped with its product Hilbert norm. For a predictable compensator
measure \(\eta\), write
\[
        \Lambda_s^\eta
        :=
        \eta(\{s\},\cdot)
\]
for its full atom and
\[
        \lambda_s^\eta
        :=
        \Lambda_s^\eta|_E.
\]
Define the hybrid atomic response by
\[
\mathcal R_{G,\varepsilon,\rho}^w
\bigl(\Lambda_s^\eta\bigr)
:=
\left(
        \Gamma_G(\Lambda_s^\eta),
        R_{\varepsilon,\rho}^w(\lambda_s^\eta)
\right).
\]
The covariance term supplies the stable global cylinder component. The
transport term is a compact-gauge residual: as shown above,
\(R_{\varepsilon,\rho}^w\) is continuous and \(w\)-invariant, vanishes
on the zero-energy class, and satisfies
\[
\frac{
R_{\varepsilon,\rho}^w(\lambda)
}{
q_w(\lambda)^{1/p}
}
\longrightarrow0
\qquad\text{as}\qquad
q_w(\lambda)\downarrow0.
\]

Define
\[
\begin{aligned}
\Pi_n(s)
&:=
\left(
        \Gamma_G(\Lambda_s^{\nu^n}),0
\right),
\\
\Pi(s)
&:=
\left(
        \Gamma_G(\Lambda_s^\nu),0
\right).
\end{aligned}
\]
By the preceding coordinate representation,
$
        \bigl((\Pi_n)_{n\ge1},\Pi\bigr)
$
is a stable global cylinder component.

Define the compact-gauge residuals by
\[
R_n(\lambda)
=
R(\lambda)
:=
\left(
        0,
        R_{\varepsilon,\rho}^w(\lambda)
\right),
\qquad \lambda\in K(E).
\]
The preceding properties of \(R_{\varepsilon,\rho}^w\) show that
$
        \bigl((R_n)_{n\ge1},R\bigr)
$
is a compact-gauge residual perturbation. In this case its perturbation
distance is identically zero.

For
\[
        \eta\in\{\nu,\nu^1,\nu^2,\ldots\},
\]
define
\[
Y_u^{G,\varepsilon,\rho,\eta}
:=
\sum_{0<s\le u}
\left(
        \Gamma_G(\Lambda_s^\eta),
        R_{\varepsilon,\rho}^w(\lambda_s^\eta)
\right),
\qquad 0\le u\le t.
\]
This process is well defined and has finite variation. Indeed, since
every compensator atom has mass at most one,
\[
\begin{aligned}
\lVert\Gamma_G(\Lambda)\rVert_{\mathrm F}
&\le
\int_{\mathbb R}
\lVert G(x)\rVert_{\mathbb R^d}^2\,\Lambda(dx)
+
\left\lVert
\int_{\mathbb R}G(x)\,\Lambda(dx)
\right\rVert_{\mathbb R^d}^2
\\
&\le
2
\int_{\mathbb R}
\lVert G(x)\rVert_{\mathbb R^d}^2\,\Lambda(dx).
\end{aligned}
\]
Furthermore,
\[
        \lVert G(x)\rVert_{\mathbb R^d}^2
        \le
        c_G^2\bigl(|x|^p\wedge1\bigr).
\]
Lemma~\ref{lem:bounded-mark-predictable-p-energy}, together with local
finiteness of the compensator away from the origin, therefore gives
\[
\sum_{0<s\le t}
\lVert\Gamma_G(\Lambda_s^\eta)\rVert_{\mathrm F}
<
\infty
\qquad\text{almost surely}.
\]
For the transport component,
\[
\begin{aligned}
\sum_{0<s\le t}
R_{\varepsilon,\rho}^w(\lambda_s^\eta)
&\le
\operatorname{diam}(E)^2
\sum_{0<s\le t}
q_w(\lambda_s^\eta)^2
\\
&\le
\operatorname{diam}(E)^2
\lVert w\rVert_\infty^p
\sum_{0<s\le t}
q_w(\lambda_s^\eta)
<
\infty
\end{aligned}
\]
almost surely. Consequently, the preceding series defines an
admissible compensator operator with identically zero continuous
residual.

Assume the process and boundary hypotheses of
Theorem~\ref{thm:compact-gauge-nonlinear-pvar-stability}. Applying that
theorem to the component \((\Pi_n,\Pi)\) and the residual perturbation
\(((R_n)_{n\ge1},R)\) gives
\[
\left[
Y^{G,\varepsilon,\rho,\nu^n}
-
Y^{G,\varepsilon,\rho,\nu}
\right]^{(p)}_t
\longrightarrow0
\qquad\text{in probability}.
\]

The direct-sum structure keeps the two output components separate.
The point of the example is therefore not an interaction between them,
but that the framework admits, within a single operator, a global
cylinder component built from the full compensator atom on
\(\mathbb R\) and a genuinely nonlinear residual depending on its
entire restriction to \(E\). Nonlinear modulation of an atomic
response by an accumulated compensator state is treated below in
Section~\ref{subsec:history-dependent-compensator-responses}.
\subsection{Nonlinear responses on compact annuli}
\label{subsec:annular-nonlinear-compensator-examples}

The constant gauge \(w\equiv1\) reduces
\(\vartheta_\lambda\) to \(\lambda\) and \(q_w(\lambda)\) to
\(\lambda(E)\). The following examples retain the unweighted forms
of familiar compensator characteristics and discrepancies.

Fix \(p>1\) and a compact annulus
\[
        E
        :=
        \{x\in\mathbb R:b\le |x|\le a\},
        \qquad
        0<b<a<\infty.
\]
For
\[
        \eta\in\{\nu,\nu^1,\nu^2,\ldots\},
\]
write
\[
        \Lambda_s^\eta
        :=
        \eta(\{s\},\cdot)|_E,
        \qquad
        m(\lambda)
        :=
        \lambda(E).
\]
Up to evanescence, every compensator atom is a subprobability measure,
so that
\[
        0\le m(\lambda_s^\eta)\le1.
\]
Moreover,
\[
\sum_{0<s\le t}m(\lambda_s^\eta)
\le
\eta((0,t]\times E)
<
\infty.
\]

Suppose that
\[
        R:K(E)\longrightarrow H
\]
satisfies
\[
        \lVert R(\lambda)\rVert_H
        \le
        K_Rm(\lambda)
\]
for some \(K_R<\infty\). If \(C^\eta\) is an \(H\)-valued continuous
finite-variation process,
then
\[
        Y_u^{R,\eta}
        :=
        C_u^\eta
        +
        \sum_{0<s\le u}R(\lambda_s^\eta)
\]
is a classically defined \(H\)-valued finite-variation process. Indeed,
\[
\sum_{0<s\le t}
\lVert R(\lambda_s^\eta)\rVert_H
\le
K_R\sum_{0<s\le t}m(\lambda_s^\eta)
<
\infty.
\]
Since the continuous finite-variation part has zero canonical
\(p\)-th power variation,
\[
        [Y^{R,\eta}]^{(p)}_t
        =
        \sum_{0<s\le t}
        \lVert R(\lambda_s^\eta)\rVert_H^p.
\]

\subsubsection{Predictable covariance characteristics}

Let
\[
        \varphi:E\longrightarrow\mathbb R^d
\]
be continuous and bounded, and let
\[
        H:=\mathbb R_{\mathrm{sym}}^{d\times d}
\]
be equipped with the Frobenius norm. Define
\[
\begin{aligned}
Y_u^{\operatorname{cov},\eta}
&:=
\int_{(0,u]}\int_E
\varphi(x)\varphi(x)^\top\,\eta(ds,dx)
-
\sum_{0<s\le u}
\left(
        \int_E\varphi(x)\,\lambda_s^\eta(dx)
\right)
\left(
        \int_E\varphi(x)\,\lambda_s^\eta(dx)
\right)^\top.
\end{aligned}
\]
Equivalently,
\[
        Y_u^{\operatorname{cov},\eta}
        =
        C_u^{\operatorname{cov},\eta}
        +
        \sum_{0<s\le u}
        R_{\operatorname{cov}}(\lambda_s^\eta),
\]
where
\[
        C_u^{\operatorname{cov},\eta}
        :=
        \int_{(0,u]}\int_E
        \varphi(x)\varphi(x)^\top\,\eta^c(ds,dx)
\]
and
\[
\begin{aligned}
R_{\operatorname{cov}}(\lambda)
&:=
\int_E
\varphi(x)\varphi(x)^\top\,\lambda(dx)
-
\left(
        \int_E\varphi(x)\,\lambda(dx)
\right)
\left(
        \int_E\varphi(x)\,\lambda(dx)
\right)^\top.
\end{aligned}
\]
The process \(C^{\operatorname{cov},\eta}\) is continuous and of finite
variation.

The response \(R_{\operatorname{cov}}\) is a polynomial cylinder
function generated by the coordinates
\[
        \int_E\varphi_i(x)\,\lambda(dx),
        \qquad
        1\le i\le d,
\]
and
\[
        \int_E\varphi_i(x)\varphi_j(x)\,\lambda(dx),
        \qquad
        1\le i,j\le d.
\]
Furthermore,
\[
\begin{aligned}
\lVert R_{\operatorname{cov}}(\lambda)\rVert_H
&\le
\int_E
\lVert\varphi(x)\rVert_{\mathbb R^d}^2\,\lambda(dx)
+
\left\lVert
\int_E\varphi(x)\,\lambda(dx)
\right\rVert_{\mathbb R^d}^2
\\
&\le
\lVert\varphi\rVert_\infty^2
\bigl(m(\lambda)+m(\lambda)^2\bigr)
\\
&\le
2\lVert\varphi\rVert_\infty^2m(\lambda).
\end{aligned}
\]
Consequently,
\[
\frac{
\lVert R_{\operatorname{cov}}(\lambda)\rVert_H^p
}{
m(\lambda)
}
\le
2^p\lVert\varphi\rVert_\infty^{2p}
m(\lambda)^{p-1}
\longrightarrow0
\]
as \(m(\lambda)\downarrow0\).

Although the polynomial map generating
\(R_{\operatorname{cov}}\) is not globally Lipschitz on its entire
Euclidean domain, the attainable moment vectors form a compact set.
Multiplying the polynomial by a Lipschitz cutoff which is identically
one on a neighbourhood of that compact set produces a globally
Lipschitz cylinder representation without changing the response on
\(K(E)\).

If \(\eta\) is the compensator of a jump measure \(\mu\) and
\[
        M
        :=
        \varphi 1_E*(\mu-\eta)
\]
is locally square integrable, then
\(Y^{\operatorname{cov},\eta}\) is the matrix-valued predictable
covariance characteristic of \(M\).

\subsubsection{Conditional cumulant characteristics}

Let
\[
        \varphi:E\longrightarrow\mathbb R
\]
be continuous and bounded. Define the third-order response by
\[
\begin{aligned}
R_3(\lambda)
&:=
\int_E\varphi(x)^3\,\lambda(dx)
-
3
\left(
        \int_E\varphi(x)\,\lambda(dx)
\right)
\left(
        \int_E\varphi(x)^2\,\lambda(dx)
\right)
+
2
\left(
        \int_E\varphi(x)\,\lambda(dx)
\right)^3.
\end{aligned}
\]
The corresponding total compensator operator is
\[
        Y_u^{(3),\eta}
        :=
        \int_{(0,u]}\int_E
        \varphi(x)^3\,\eta^c(ds,dx)
        +
        \sum_{0<s\le u}R_3(\lambda_s^\eta).
\]
Its continuous part is continuous and of finite variation.

The response \(R_3\) is a polynomial cylinder function generated by
\[
        \int_E\varphi(x)\,\lambda(dx),
        \qquad
        \int_E\varphi(x)^2\,\lambda(dx),
        \qquad
        \int_E\varphi(x)^3\,\lambda(dx).
\]
Since \(m(\lambda)\le1\), every monomial occurring in
\(R_3(\lambda)\) is bounded by a constant multiple of \(m(\lambda)\).
Thus there exists a constant \(C_\varphi<\infty\) such that
\[
        |R_3(\lambda)|
        \le
        C_\varphi m(\lambda).
\]
In particular, the jump series defining \(Y^{(3),\eta}\) converges
absolutely and
\[
        [Y^{(3),\eta}]^{(p)}_t
        =
        \sum_{0<s\le t}
        |R_3(\lambda_s^\eta)|^p.
\]

More generally, for every integer \(r\ge2\), the \(r\)-th conditional
cumulant response is obtained by applying the usual cumulant polynomial
to the moment coordinates
\[
        \int_E\varphi(x)^j\,\lambda(dx),
        \qquad
        j=1,\ldots,r.
\]
Every such response is an exact polynomial cylinder function and is
bounded by a constant multiple of \(m(\lambda)\). The attainable moment vectors form a compact subset of the
finite-dimensional coordinate space. Multiplying the corresponding
cumulant polynomial by a compactly supported Lipschitz cutoff which
equals one on a neighbourhood of this set therefore gives a globally
Lipschitz cylinder representation without changing the response on
\(K(E)\).

\subsubsection{Kernel discrepancies and maximum mean discrepancy}

Let
\[
        \rho\in\mathcal P(E)
\]
and let
\[
        k:E\times E\longrightarrow\mathbb R
\]
be a bounded continuous positive definite kernel. Define
\[
\begin{aligned}
R_{k,\rho}(\lambda)
:=
\int_E\int_E
k(x,y)\,
\bigl(\lambda-m(\lambda)\rho\bigr)(dx)\,
\bigl(\lambda-m(\lambda)\rho\bigr)(dy).
\end{aligned}
\]
For \(m(\lambda)>0\),
\[
        R_{k,\rho}(\lambda)
        =
        m(\lambda)^2
        \operatorname{MMD}_k^2
        \left(
        \frac{\lambda}{m(\lambda)},
        \rho
        \right),
\]
and \(R_{k,\rho}(0)=0\).

Since
\[
\bigl\lVert
\lambda-m(\lambda)\rho
\bigr\rVert_{\mathrm{TV}}
\le
2m(\lambda),
\]
we have
\[
        0
        \le
        R_{k,\rho}(\lambda)
        \le
        4\lVert k\rVert_\infty m(\lambda)^2.
\]
Consequently,
\[
        Y_u^{k,\rho,\eta}
        :=
        \sum_{0<s\le u}
        R_{k,\rho}(\lambda_s^\eta)
\]
is a classically defined increasing finite-variation process.

Suppose first that \(k\) has finite rank and can be written as
\[
        k(x,y)
        =
        \sum_{j=1}^N
        c_jh_j(x)h_j(y),
        \qquad
        c_j\ge0,
        \qquad
        h_j\in C(E).
\]
Then
\[
\begin{aligned}
R_{k,\rho}(\lambda)
&=
\sum_{j=1}^N
c_j
\left(
        \int_Eh_j(x)\,\lambda(dx)
        -
        m(\lambda)\int_Eh_j(x)\,\rho(dx)
\right)^2.
\end{aligned}
\]
Hence \(R_{k,\rho}\) is an exact polynomial cylinder function generated
by the mass coordinate
\[
        m(\lambda)=\int_E1\,\lambda(dx)
\]
and the moment coordinates
\[
        \int_Eh_j(x)\,\lambda(dx),
        \qquad
        j=1,\ldots,N.
\]
The attainable coordinate vectors form a compact set. Hence the
quadratic coordinate map may be multiplied by a compactly supported
Lipschitz cutoff which equals one on a neighbourhood of this set.
This gives a globally Lipschitz cylinder representation of
\(R_{k,\rho}\) on \(K(E)\).
For a general bounded continuous positive definite kernel \(k\), there
exist continuous finite-rank positive definite kernels \(k_N\) such
that
\[
        \lVert k_N-k\rVert_\infty
        \longrightarrow0.
\]
For example, such kernels may be obtained by finite-dimensional
projection of the canonical feature map of \(k\). For every
\(\lambda\in K(E)\),
\[
\begin{aligned}
|R_{k_N,\rho}(\lambda)-R_{k,\rho}(\lambda)|
&\le
4\lVert k_N-k\rVert_\infty m(\lambda)^2.
\end{aligned}
\]
It follows that
\[
\begin{aligned}
\sup_{\substack{\lambda\in K(E)\\\lambda\ne0}}
\frac{
|R_{k_N,\rho}(\lambda)-R_{k,\rho}(\lambda)|
}{
m(\lambda)^{1/p}
}
&\le
4\lVert k_N-k\rVert_\infty
\sup_{0<m\le1}m^{2-1/p}
\\
&\le
4\lVert k_N-k\rVert_\infty
\longrightarrow0.
\end{aligned}
\]
Thus \(R_{k,\rho}\) belongs to the weighted cylinder closure associated
with the weight \(\lambda(E)^{1/p}\).

\subsubsection{Intermediate-order Wasserstein responses}

Let
$
        \rho\in\mathcal P(E)
$
and choose
$
        \frac1p<\alpha<1.
$
Define
\[
F_{\alpha,\rho}(\lambda)
:=
\begin{cases}
m(\lambda)^\alpha
W_2^2\left(
        \dfrac{\lambda}{m(\lambda)},
        \rho
\right),
& m(\lambda)>0,\\[2mm]
0,
& m(\lambda)=0.
\end{cases}
\]
The map \(F_{\alpha,\rho}\) is continuous on \(K(E)\). Indeed, it is
continuous away from the zero measure, while
\[
        0
        \le
        F_{\alpha,\rho}(\lambda)
        \le
        \operatorname{diam}(E)^2m(\lambda)^\alpha
\]
implies continuity at the zero measure. Furthermore,
\[
\frac{
F_{\alpha,\rho}(\lambda)
}{
m(\lambda)^{1/p}
}
\le
\operatorname{diam}(E)^2
m(\lambda)^{\alpha-1/p}
\longrightarrow0
\]
as \(m(\lambda)\downarrow0\). Hence
\(F_{\alpha,\rho}\) belongs to the weighted cylinder closure described
in Theorem~\ref{thm:weighted-cylinder-density}.

In general, \(F_{\alpha,\rho}\) is not
\(O(m(\lambda))\) at the zero measure. To see this, choose
\(\sigma\in\mathcal P(E)\) such that \(\sigma\ne\rho\), and put
\[
        \lambda_r:=r\sigma,
        \qquad 0<r\le1.
\]
Then
\[
\frac{
F_{\alpha,\rho}(\lambda_r)
}{
m(\lambda_r)
}
=
r^{\alpha-1}W_2^2(\sigma,\rho)
\longrightarrow\infty
\]
as \(r\downarrow0\).

The associated jump series is nevertheless absolutely convergent
whenever
\[
        \sum_{0<s\le t}
        m(\lambda_s^\eta)^\alpha
        <
        \infty,
\]
because
\[
\begin{aligned}
\sum_{0<s\le t}
F_{\alpha,\rho}(\lambda_s^\eta)
&\le
\operatorname{diam}(E)^2
\sum_{0<s\le t}
m(\lambda_s^\eta)^\alpha
<
\infty.
\end{aligned}
\]
This additional summability does not follow from first-order annular
mass finiteness. Indeed, for every \(\alpha<1\), one may choose
\[
        1<q\le\frac1\alpha
\]
and set \(m_k=k^{-q}\). Then
\[
        \sum_{k\ge1}m_k<\infty,
        \qquad
        \sum_{k\ge1}m_k^\alpha=\infty.
\]
Thus the absolute summability required by admissibility in
Corollary~\ref{cor:nonlinear-pvar-stability-compact-annulus}
does not follow from the weighted zero-tangent condition together
with finiteness of the first-order annular mass.

\subsubsection{Regularised Wasserstein responses}

Let
$
        \rho\in\mathcal P(E)
$
and fix \(\varepsilon>0\). Define
\[
R_{\varepsilon,\rho}(\lambda)
:=
m(\lambda)^2
W_2^2
\left(
        \frac{\lambda+\varepsilon\rho}
             {m(\lambda)+\varepsilon},
        \rho
\right).
\]
Since \(E\) is compact, \(W_2\) metrises weak convergence on
\(\mathcal P(E)\). The maps
\[
        \lambda\longmapsto m(\lambda)
\]
and
\[
        \lambda
        \longmapsto
        \frac{\lambda+\varepsilon\rho}
             {m(\lambda)+\varepsilon}
\]
are continuous for weak convergence. Hence
\[
        R_{\varepsilon,\rho}:K(E)\longrightarrow\mathbb R
\]
is continuous.

Moreover,
\[
        0
        \le
        R_{\varepsilon,\rho}(\lambda)
        \le
        \operatorname{diam}(E)^2m(\lambda)^2.
\]
Therefore
\[
\frac{
R_{\varepsilon,\rho}(\lambda)
}{
m(\lambda)^{1/p}
}
\le
\operatorname{diam}(E)^2
m(\lambda)^{2-1/p}
\longrightarrow0
\]
as \(m(\lambda)\downarrow0\). Thus
\(R_{\varepsilon,\rho}\) satisfies the hypothesis of
Theorem~\ref{thm:weighted-cylinder-density}.

Consequently, there exist globally Lipschitz cylinder functions
\(P_N\) such that
\[
\sup_{\substack{\lambda\in K(E)\\\lambda\ne0}}
\frac{
|R_{\varepsilon,\rho}(\lambda)-P_N(\lambda)|
}{
m(\lambda)^{1/p}
}
\longrightarrow0.
\]
The corresponding compensator process
\[
        Y_u^{\varepsilon,\rho,\eta}
        :=
        \sum_{0<s\le u}
        R_{\varepsilon,\rho}(\lambda_s^\eta)
\]
is classically well defined, since
\[
\begin{aligned}
\sum_{0<s\le t}
R_{\varepsilon,\rho}(\lambda_s^\eta)
&\le
\operatorname{diam}(E)^2
\sum_{0<s\le t}m(\lambda_s^\eta)^2
\\
&\le
\operatorname{diam}(E)^2
\sum_{0<s\le t}m(\lambda_s^\eta)
<
\infty.
\end{aligned}
\]
Let \(X\) and \((X^n)_{n\ge1}\) have predictable compensators
\(\nu\) and \(\nu^n\), respectively. Assume that \(X\), \(X^n\), and
\(X^n-X\) possess finite canonical \(p\)-th power variation on
\([0,t]\), that
\[
        [X^n-X]^{(p)}_t
        \longrightarrow0
        \qquad\text{in probability},
\]
and that the boundary levels \(b\) and \(a\) are regular for \(X\).
Corollary~\ref{cor:nonlinear-pvar-stability-compact-annulus}, applied with
\[
        F_n=F=R_{\varepsilon,\rho},
        \qquad n\ge1,
\]
then gives
\[
\left[
        Y^{\varepsilon,\rho,\nu^n}
        -
        Y^{\varepsilon,\rho,\nu}
\right]^{(p)}_t
\longrightarrow0
\qquad\text{in probability}.
\]

In general, \(R_{\varepsilon,\rho}\) is not a cylinder function of
finitely many moment coordinates. This provides an example for which
the weighted cylinder closure, rather than merely the finite cylinder
class, is essential.

\subsubsection{History-dependent compensator responses}
\label{subsec:history-dependent-compensator-responses}

Let \(H_0\) and \(H\) be separable Hilbert spaces. For a predictable
compensator measure \(\eta\), write
\[
        \lambda_s^\eta
        :=
        \eta(\{s\},\cdot)|_E,
        \qquad
        m(\lambda_s^\eta)
        :=
        \lambda_s^\eta(E).
\]

\begin{thm}
[Stability of history-dependent compensator responses]
\label{thm:history-dependent-response-stability}

Let \(p>1\). For every \(n\ge1\), let
\[
R_n:
\Omega\times[0,t]\times K(E)
\longrightarrow H_0,
\]
and let
\[
R:
\Omega\times[0,t]\times K(E)
\longrightarrow H_0
\]
be predictable response maps. Assume that
\[
\sum_{0<s\le t}
\left\lVert
R_n\bigl(s,\lambda_s^{\nu^n}\bigr)
-
R\bigl(s,\lambda_s^\nu\bigr)
\right\rVert_{H_0}^p
\longrightarrow0
\]
in probability.

Assume that
\[
\sum_{0<s\le t}m(\lambda_s^\nu)<\infty
\qquad\text{almost surely},
\]
that, for every \(n\ge1\),
\[
\sum_{0<s\le t}m(\lambda_s^{\nu^n})<\infty
\qquad\text{almost surely},
\]
and that the family
\[
\left\{
\sum_{0<s\le t}m(\lambda_s^{\nu^n})
:n\ge1
\right\}
\]
is bounded in probability.

Let \(Z^n\) and \(Z\) be \(\mathbb R^d\)-valued càdlàg processes such
that
\[
\sup_{0\le u\le t}
\lVert Z_u^n-Z_u\rVert_{\mathbb R^d}
\longrightarrow0
\]
in probability.

Let
\[
\Phi:
\Omega\times[0,t]\times\mathbb R^d\times H_0
\longrightarrow H
\]
be predictable in \((\omega,s)\) and Borel measurable in its remaining
variables. Suppose that, for every \(K<\infty\), there exist constants
\(L_K,C_K<\infty\) such that the following conditions hold outside an
evanescent set.

For every \(z\in\mathbb R^d\) with \(\lVert z\rVert_{\mathbb R^d}\le K\)
and every \(u,v\in H_0\),
\[
\lVert
\Phi_s(z,u)-\Phi_s(z,v)
\rVert_H
\le
L_K\lVert u-v\rVert_{H_0}.
\]

For every \(n\ge1\), every \(\lambda\in K(E)\), and every
\(z,z'\in\mathbb R^d\) satisfying
\[
\lVert z\rVert_{\mathbb R^d}
\vee
\lVert z'\rVert_{\mathbb R^d}
\le K,
\]
assume that
\[
\begin{aligned}
&\lVert
\Phi_s\bigl(z,R_n(s,\lambda)\bigr)
-
\Phi_s\bigl(z',R_n(s,\lambda)\bigr)
\rVert_H
\le
L_K
\lVert z-z'\rVert_{\mathbb R^d}
m(\lambda)^{1/p}.
\end{aligned}
\]

Finally, assume that, for every \(n\ge1\), every
\(\lambda\in K(E)\), and every \(z\in\mathbb R^d\) satisfying
\(\lVert z\rVert_{\mathbb R^d}\le K\),
\[
\begin{aligned}
&\lVert
\Phi_s\bigl(z,R_n(s,\lambda)\bigr)
\rVert_H
+
\lVert
\Phi_s\bigl(z,R(s,\lambda)\bigr)
\rVert_H
\le
C_Km(\lambda).
\end{aligned}
\]

Let \(C^{\nu^n}\) and \(C^\nu\) be \(H\)-valued continuous
finite-variation processes, and define
\[
\begin{aligned}
Y_u^n
&:=
C_u^{\nu^n}
+
\sum_{0<s\le u}
\Phi_s\left(
Z_{s-}^n,
R_n\bigl(s,\lambda_s^{\nu^n}\bigr)
\right),
\\
Y_u
&:=
C_u^\nu
+
\sum_{0<s\le u}
\Phi_s\left(
Z_{s-},
R\bigl(s,\lambda_s^\nu\bigr)
\right),
\qquad 0\le u\le t.
\end{aligned}
\]

Then \(Y^n\) and \(Y\) are well-defined càdlàg finite-variation
processes and
\[
        \bigl[Y^n-Y\bigr]^{(p)}_t
        \longrightarrow0
\]
in probability.
\end{thm}

\begin{proof}
Since \(Z^n\) and \(Z\) are càdlàg, their paths are bounded on
\([0,t]\). The first-order response bound and the assumed finiteness
of the annular atomic masses therefore give
\[
\begin{aligned}
&\sum_{0<s\le t}
\left\lVert
\Phi_s\left(
Z_{s-}^n,
R_n\bigl(s,\lambda_s^{\nu^n}\bigr)
\right)
\right\rVert_H
<\infty,
\\
&\sum_{0<s\le t}
\left\lVert
\Phi_s\left(
Z_{s-},
R\bigl(s,\lambda_s^\nu\bigr)
\right)
\right\rVert_H
<\infty
\end{aligned}
\]
almost surely. Hence both jump series converge absolutely.

The difference \(C^{\nu^n}-C^\nu\) is continuous and of finite
variation, and therefore has zero canonical \(p\)-th power variation.
Consequently,
\[
\begin{aligned}
[Y^n-Y]^{(p)}_t
&=
\sum_{0<s\le t}
\Bigl\lVert
\Phi_s\left(
Z_{s-}^n,
R_n\bigl(s,\lambda_s^{\nu^n}\bigr)
\right)
-
\Phi_s\left(
Z_{s-},
R\bigl(s,\lambda_s^\nu\bigr)
\right)
\Bigr\rVert_H^p.
\end{aligned}
\]

At every \(s\le t\),
\[
\begin{aligned}
\Phi_s\left(
Z_{s-}^n,
R_n\bigl(s,\lambda_s^{\nu^n}\bigr)
\right)
-
\Phi_s\left(
Z_{s-},
R\bigl(s,\lambda_s^\nu\bigr)
\right)
&=
\Phi_s\left(
Z_{s-}^n,
R_n\bigl(s,\lambda_s^{\nu^n}\bigr)
\right)
-
\Phi_s\left(
Z_{s-},
R_n\bigl(s,\lambda_s^{\nu^n}\bigr)
\right)
\\
&+
\Phi_s\left(
Z_{s-},
R_n\bigl(s,\lambda_s^{\nu^n}\bigr)
\right)
-
\Phi_s\left(
Z_{s-},
R\bigl(s,\lambda_s^\nu\bigr)
\right).
\end{aligned}
\]

Fix \(K,M<\infty\), and consider the event
\[
\begin{aligned}
\Gamma_{n,K,M}
:=
\Biggl\{
&\sup_{0\le u\le t}\lVert Z_u^n\rVert_{\mathbb R^d}
\vee
\sup_{0\le u\le t}\lVert Z_u\rVert_{\mathbb R^d}
\le K,
\sum_{0<s\le t}m(\lambda_s^{\nu^n})
\le M
\Biggr\}.
\end{aligned}
\]

On \(\Gamma_{n,K,M}\), the weighted Lipschitz estimate in the state
variable gives
\[
\begin{aligned}
\sum_{0<s\le t}
\Bigl\lVert
\Phi_s\left(
Z_{s-}^n,
R_n\bigl(s,\lambda_s^{\nu^n}\bigr)
\right)
-
\Phi_s\left(
Z_{s-},
R_n\bigl(s,\lambda_s^{\nu^n}\bigr)
\right)
\Bigr\rVert_H^p
&\le
L_K^p
\sup_{0\le u\le t}
\lVert Z_u^n-Z_u\rVert_{\mathbb R^d}^p
\sum_{0<s\le t}m(\lambda_s^{\nu^n})
\\
&\le
L_K^pM
\sup_{0\le u\le t}
\lVert Z_u^n-Z_u\rVert_{\mathbb R^d}^p.
\end{aligned}
\]
This converges to zero in probability.

The Lipschitz estimate in the response variable gives, on the same
event,
\[
\begin{aligned}
&\sum_{0<s\le t}
\Bigl\lVert
\Phi_s\left(
Z_{s-},
R_n\bigl(s,\lambda_s^{\nu^n}\bigr)
\right)
-
\Phi_s\left(
Z_{s-},
R\bigl(s,\lambda_s^\nu\bigr)
\right)
\Bigr\rVert_H^p
\le
L_K^p
\sum_{0<s\le t}
\left\lVert
R_n\bigl(s,\lambda_s^{\nu^n}\bigr)
-
R\bigl(s,\lambda_s^\nu\bigr)
\right\rVert_{H_0}^p,
\end{aligned}
\]
which also converges to zero in probability.

The convergence
\[
\sup_{0\le u\le t}
\lVert Z_u^n-Z_u\rVert_{\mathbb R^d}
\longrightarrow0
\]
in probability and the almost sure boundedness of the càdlàg path
\(Z\) imply that the family
\[
\left\{
\sup_{0\le u\le t}\lVert Z_u^n\rVert_{\mathbb R^d}
\vee
\sup_{0\le u\le t}\lVert Z_u\rVert_{\mathbb R^d}
:n\ge1
\right\}
\]
is bounded in probability.

Fix \(\varepsilon>0\). We may therefore choose \(K<\infty\) such that
\[
\sup_{n\ge1}
\P\left(
\sup_{0\le u\le t}\lVert Z_u^n\rVert_{\mathbb R^d}
\vee
\sup_{0\le u\le t}\lVert Z_u\rVert_{\mathbb R^d}
>
K
\right)
<
\frac{\varepsilon}{2}.
\]
By the assumed boundedness in probability of the annular atomic
masses, we may then choose \(M<\infty\) such that
\[
\sup_{n\ge1}
\P\left(
\sum_{0<s\le t}m(\lambda_s^{\nu^n})>M
\right)
<
\frac{\varepsilon}{2}.
\]
Consequently,
$
        \sup_{n\ge1}\P(\Gamma_{n,K,M}^c)<\varepsilon.
$

For these fixed \(K\) and \(M\), the two preceding estimates and
\[
        (a+b)^p\le2^{p-1}(a^p+b^p),
        \qquad a,b\ge0,
\]
show that, for every \(\delta>0\),
\[
\lim_{n\to\infty}
\P\left(
\bigl[Y^n-Y\bigr]^{(p)}_t>\delta,
\ \Gamma_{n,K,M}
\right)
=
0.
\]
It follows that
\[
\limsup_{n\to\infty}
\P\left(
\bigl[Y^n-Y\bigr]^{(p)}_t>\delta
\right)
\le
\varepsilon.
\]
Since \(\varepsilon>0\) was arbitrary,
\[
        \bigl[Y^n-Y\bigr]^{(p)}_t
        \longrightarrow0
\]
in probability.
\end{proof}

The factorised state modulation
\[
        \Phi_s(z,u)=\Psi(z)u,
\]
where
\[
        \Psi:\mathbb R^d\longrightarrow\mathcal L(H_0,H)
\]
is bounded and globally Lipschitz, is covered by
Theorem~\ref{thm:history-dependent-response-stability}. Indeed, if
\[
\lVert R_n(s,\lambda)\rVert_{H_0}
+
\lVert R(s,\lambda)\rVert_{H_0}
\le
C_Rm(\lambda),
\]
then
\[
\begin{aligned}
\lVert
\Psi(z)R_n(s,\lambda)
-
\Psi(z')R_n(s,\lambda)
\rVert_H
&\le
L_\Psi C_R
\lVert z-z'\rVert_{\mathbb R^d}
m(\lambda)
\\
&\le
L_\Psi C_R
\lVert z-z'\rVert_{\mathbb R^d}
m(\lambda)^{1/p},
\end{aligned}
\]
because \(m(\lambda)\le1\). Thus the state-modulated covariance,
cumulant, kernel and regularised transport responses satisfy the
weighted state estimate whenever their unmodulated responses satisfy
the stated first-order bound.
\subsubsection{Cumulative covariance-type operators}

The covariance response considered above depends only on the current
compensator atom. A different covariance-type functional is obtained
by forming the covariance of the compensator mass accumulated up to
the current time.

Let
\[
        \varphi:E\longrightarrow\mathbb R^d
\]
be continuous and bounded. For a predictable compensator measure
\(\eta\), define
\[
Z_u^\eta
:=
\int_{(0,u]}\int_E
\varphi(x)\,\eta(ds,dx)
\]
and
\[
Q_u^\eta
:=
\int_{(0,u]}\int_E
\varphi(x)\varphi(x)^\top\,\eta(ds,dx).
\]
The cumulative covariance-type process is
\[
        \mathcal C_u^\eta
        :=
        Q_u^\eta-Z_u^\eta(Z_u^\eta)^\top.
\]

At every \(s\le t\),
\[
\begin{aligned}
\Delta\mathcal C_s^\eta
&=
\Delta Q_s^\eta
-
Z_{s-}^\eta(\Delta Z_s^\eta)^\top
-
\Delta Z_s^\eta(Z_{s-}^\eta)^\top
-
\Delta Z_s^\eta(\Delta Z_s^\eta)^\top,
\end{aligned}
\]
where
\[
        \Delta Z_s^\eta
        =
        \int_E\varphi(x)\,\eta(\{s\},dx)
\]
and
\[
        \Delta Q_s^\eta
        =
        \int_E
        \varphi(x)\varphi(x)^\top\,\eta(\{s\},dx).
\]
Thus the response at time \(s\) depends both on the current
compensator atom and on the preceding state \(Z_{s-}^\eta\).

\begin{prop}[Stability of cumulative covariance-type operators]
\label{prop:cumulative-covariance-stability}

Let \(X\) and \((X^n)_{n\ge1}\) have predictable compensators
\(\nu\) and \(\nu^n\), respectively. Assume that \(X\), \(X^n\), and
\(X^n-X\) possess finite canonical \(p\)-th power variation on
\([0,t]\), and that
\[
        [X^n-X]^{(p)}_t\longrightarrow0
\]
in probability. Assume that the boundary levels of \(E\) are regular
for \(X\).

Suppose that the scalar integrands
\[
        \varphi_i1_E,
        \qquad 1\le i\le d,
\]
and
\[
        \varphi_i\varphi_j1_E,
        \qquad 1\le i,j\le d,
\]
satisfy the hypotheses of
Theorem~\ref{thm:CMC-pvar-compensator-stability}. Assume in addition
that the integrands \(\varphi_i1_E\) satisfy the hypotheses of
Theorem~\ref{thm:CMC-p-variation-compensator-ucp}.

Then
\[
        [\mathcal C^{\nu^n}-\mathcal C^\nu]^{(p)}_t
        \longrightarrow0
\]
in probability.
\end{prop}

\begin{proof}
Applying Theorem~\ref{thm:CMC-pvar-compensator-stability} componentwise
gives
\[
        [Z^{\nu^n}-Z^\nu]^{(p)}_t
        \longrightarrow0
\]
and
\[
        [Q^{\nu^n}-Q^\nu]^{(p)}_t
        \longrightarrow0
\]
in probability. Theorem~\ref{thm:CMC-p-variation-compensator-ucp}
also gives
\[
        \sup_{u\le t}
        \lVert Z_u^{\nu^n}-Z_u^\nu\rVert_{\mathbb R^d}
        \longrightarrow0
\]
in probability.

Subtracting the two atomic covariance increments and using
\[
        \lVert xy^\top\rVert_{\mathrm F}
        =
        \lVert x\rVert_{\mathbb R^d}
        \lVert y\rVert_{\mathbb R^d},
\]
we obtain
\[
\begin{aligned}
\lVert
\Delta\mathcal C_s^{\nu^n}
-
\Delta\mathcal C_s^\nu
\rVert_{\mathrm F}
&\le
\lVert
\Delta Q_s^{\nu^n}-\Delta Q_s^\nu
\rVert_{\mathrm F}
+
2
\lVert
Z_{s-}^{\nu^n}-Z_{s-}^\nu
\rVert_{\mathbb R^d}
\lVert\Delta Z_s^{\nu^n}\rVert_{\mathbb R^d}
\\
&+
2
\lVert Z_{s-}^\nu\rVert_{\mathbb R^d}
\lVert
\Delta Z_s^{\nu^n}-\Delta Z_s^\nu
\rVert_{\mathbb R^d}
\\
&+
\left(
\lVert\Delta Z_s^{\nu^n}\rVert_{\mathbb R^d}
+
\lVert\Delta Z_s^\nu\rVert_{\mathbb R^d}
\right)
\lVert
\Delta Z_s^{\nu^n}-\Delta Z_s^\nu
\rVert_{\mathbb R^d}.
\end{aligned}
\]

Since every compensator atom is a subprobability measure,
\[
        \lVert\Delta Z_s^{\nu^n}\rVert_{\mathbb R^d}
        \le
        \lVert\varphi\rVert_\infty
        \nu^n(\{s\},E)
        \le
        \lVert\varphi\rVert_\infty,
\]
and similarly for \(\nu\). Moreover,
\[
\sum_{0<s\le t}
\lVert\Delta Z_s^{\nu^n}\rVert_{\mathbb R^d}^p
\le
\lVert\varphi\rVert_\infty^p
\sum_{0<s\le t}\nu^n(\{s\},E).
\]
By
Corollary~\ref{cor:annular-compensator-atom-mass-tightness},
the family
\[
\left\{
        \sum_{0<s\le t}\nu^n(\{s\},E)
        :
        n\ge1
\right\}
\]
is bounded in probability. Consequently,
\[
\left\{
        \sum_{0<s\le t}
        \lVert\Delta Z_s^{\nu^n}\rVert_{\mathbb R^d}^p
        :
        n\ge1
\right\}
\]
is bounded in probability.
It follows that
\[
\begin{aligned}
&\sum_{0<s\le t}
\lVert
Z_{s-}^{\nu^n}-Z_{s-}^\nu
\rVert_{\mathbb R^d}^p
\lVert\Delta Z_s^{\nu^n}\rVert_{\mathbb R^d}^p
\le
\sup_{u\le t}
\lVert Z_u^{\nu^n}-Z_u^\nu\rVert_{\mathbb R^d}^p
\sum_{0<s\le t}
\lVert\Delta Z_s^{\nu^n}\rVert_{\mathbb R^d}^p
\longrightarrow0
\end{aligned}
\]
in probability.

Since \(Z^\nu\) has finite variation,
$
        \sup_{u\le t}
        \lVert Z_u^\nu\rVert_{\mathbb R^d}
        <
        \infty
$
almost surely. Therefore
\[
\begin{aligned}
&\sum_{0<s\le t}
\lVert Z_{s-}^\nu\rVert_{\mathbb R^d}^p
\lVert
\Delta Z_s^{\nu^n}-\Delta Z_s^\nu
\rVert_{\mathbb R^d}^p
\le
\sup_{u\le t}
\lVert Z_u^\nu\rVert_{\mathbb R^d}^p
[Z^{\nu^n}-Z^\nu]^{(p)}_t
\longrightarrow0
\end{aligned}
\]
in probability. The final term is bounded by
$
(2\lVert\varphi\rVert_\infty)^p
[Z^{\nu^n}-Z^\nu]^{(p)}_t
$
and hence also converges to zero in probability.

Using
\[
        (a_1+a_2+a_3+a_4)^p
        \le
        4^{p-1}
        \sum_{j=1}^4a_j^p,
\]
and summing over \(s\le t\), we conclude that
\[
\sum_{0<s\le t}
\lVert
\Delta\mathcal C_s^{\nu^n}
-
\Delta\mathcal C_s^\nu
\rVert_{\mathrm F}^p
\longrightarrow0
\]
in probability.

The continuous part of
\(\mathcal C^{\nu^n}-\mathcal C^\nu\) is continuous and of finite
variation. It has zero canonical \(p\)-th power variation. Consequently,
\[
        [\mathcal C^{\nu^n}-\mathcal C^\nu]^{(p)}_t
        \longrightarrow0
\]
in probability.
\end{proof}

\subsubsection{History-modulated regularised Wasserstein responses}

Let
\[
        g_1,\ldots,g_r:
        \Omega\times[0,t]\times E\longrightarrow\mathbb R
\]
be predictable functions, and define
\[
U_u^\eta
:=
\left(
\int_{(0,u]}\int_Eg_1(s,x)\,\eta(ds,dx),
\ldots,
\int_{(0,u]}\int_Eg_r(s,x)\,\eta(ds,dx)
\right).
\]
Let
\[
        \psi:\mathbb R^r\longrightarrow\mathbb R
\]
be bounded and globally Lipschitz. Using the regularised Wasserstein
response \(R_{\varepsilon,\rho}\) defined above, set
\[
\mathcal W_u^{\varepsilon,\rho,\eta}
:=
\sum_{0<s\le u}
\psi(U_{s-}^\eta)
R_{\varepsilon,\rho}(\lambda_s^\eta).
\]
This process is well defined because
\[
\begin{aligned}
\sum_{0<s\le t}
\left|
\psi(U_{s-}^\eta)
R_{\varepsilon,\rho}(\lambda_s^\eta)
\right|
&\le
\lVert\psi\rVert_\infty
\operatorname{diam}(E)^2
\sum_{0<s\le t}m(\lambda_s^\eta)^2
\\
&\le
\lVert\psi\rVert_\infty
\operatorname{diam}(E)^2
\sum_{0<s\le t}m(\lambda_s^\eta)
<
\infty.
\end{aligned}
\]

Suppose that the functions \(g_1,\ldots,g_r\), extended by zero
outside \(E\), satisfy the hypotheses of
Theorem~\ref{thm:CMC-p-variation-compensator-ucp}, and suppose that the hypotheses of
Corollary~\ref{cor:nonlinear-pvar-stability-compact-annulus}
hold with
\[
        F_n=F=R_{\varepsilon,\rho},
        \qquad n\ge1.
\]
Then
\[
U^{\nu^n}\longrightarrow U^\nu
\qquad\text{ucp},
\]
and
\[
\left[
\sum_{0<s\le\cdot}
R_{\varepsilon,\rho}(\lambda_s^{\nu^n})
-
\sum_{0<s\le\cdot}
R_{\varepsilon,\rho}(\lambda_s^\nu)
\right]^{(p)}_t
\longrightarrow0
\]
in probability.

In Theorem~\ref{thm:history-dependent-response-stability}, take
\[
R_n(s,\lambda)
=
R(s,\lambda)
=
R_{\varepsilon,\rho}(\lambda),
\qquad
Z^n=U^{\nu^n},
\qquad
Z=U^\nu,
\]
and let both continuous finite-variation terms be identically zero.
Set
\[
        \Phi(z,c):=\psi(z)c.
\]
For every \(z\in\mathbb R^r\) and \(c,c'\in\mathbb R\),
\[
\left|
\Phi(z,c)-\Phi(z,c')
\right|
\le
\lVert\psi\rVert_\infty|c-c'|.
\]
Thus \(\Phi\) is uniformly Lipschitz in the response variable. Since
\[
0\le
R_{\varepsilon,\rho}(\lambda)
\le
\operatorname{diam}(E)^2m(\lambda)^2,
\]
we have
\[
\begin{aligned}
\left|
\Phi\bigl(z,R_{\varepsilon,\rho}(\lambda)\bigr)
-
\Phi\bigl(z',R_{\varepsilon,\rho}(\lambda)\bigr)
\right|
&\le
\operatorname{Lip}(\psi)
\operatorname{diam}(E)^2
|z-z'|m(\lambda)^2
\\
&\le
\operatorname{Lip}(\psi)
\operatorname{diam}(E)^2
|z-z'|m(\lambda)^{1/p},
\end{aligned}
\]
because \(0\le m(\lambda)\le1\). Moreover,
\[
\begin{aligned}
\left|
\Phi\bigl(z,R_n(s,\lambda)\bigr)
\right|
+
\left|
\Phi\bigl(z,R(s,\lambda)\bigr)
\right|
&=
2\left|
\Phi\bigl(z,R_{\varepsilon,\rho}(\lambda)\bigr)
\right|
\\
&\le
2\lVert\psi\rVert_\infty
\operatorname{diam}(E)^2m(\lambda)^2
\\
&\le
2\lVert\psi\rVert_\infty
\operatorname{diam}(E)^2m(\lambda).
\end{aligned}
\]
By
Corollary~\ref{cor:annular-compensator-atom-mass-tightness},
\[
\left\{
        \sum_{0<s\le t}m(\lambda_s^{\nu^n})
        :
        n\ge1
\right\}
\]
is bounded in probability.
Theorem~\ref{thm:history-dependent-response-stability} therefore
yields
\[
\left[
\mathcal W^{\varepsilon,\rho,\nu^n}
-
\mathcal W^{\varepsilon,\rho,\nu}
\right]^{(p)}_t
\longrightarrow0
\]
in probability.
The covariance and cumulant operators describe the conditional
second- and higher-order structure of the jump marks at predictable
times. Kernel and Wasserstein responses instead quantify how the
conditional distribution of a jump mark differs from a prescribed
reference law. The accompanying powers of the atom mass ensure that
the response vanishes sufficiently rapidly for small compensator
atoms and that the resulting jump series is summable. State-modulated
responses additionally allow the effect of the current compensator
atom to depend on information accumulated from the preceding
compensator history.

\section{Proofs of the structural results and of
Theorem~\ref{thm:CMC-pvar-compensator-stability}}
\label{sec:proof}

\subsection{Proofs of the structural results}

\begin{proof}[Proof of Lemma \ref{lem:p-thres}]
Fix \(r>0\). For every \(k\ge1\), set
\[
\mathcal I_k(r)
:=
[L(r,k)-\delta(k,r),\,L(r,k)+\delta(k,r)].
\]
Since
\[
L(r,k)-\delta(k,r)
=
r\left(1-\frac1{2^k}\right)
\]
and
\[
L(r,k)+\delta(k,r)
=
r\left(1-\frac1{2^{k+1}}\right),
\]
we have
\[
\mathcal I_k(r)
=
\left[
r\left(1-\frac1{2^k}\right),
r\left(1-\frac1{2^{k+1}}\right)
\right].
\]
In particular,
\[
\mathcal I_k(r)\subseteq\left[\frac r2,r\right]
\qquad\text{for every }k\ge1.
\]

By the definition of \(A_k(r)\),
\[
A_k(r)^c
=
\left\{
\omega:
\exists s\le t
\text{ such that }
|\Delta X_s(\omega)|\in \mathcal I_k(r)
\right\}.
\]

The interiors of the intervals \(\mathcal I_k(r)\) are pairwise disjoint, and
two distinct intervals can intersect only at a common endpoint.
Consequently, every point of \([r/2,r]\) belongs to at most two of the
intervals \(\mathcal I_k(r)\).

Fix a càdlàg path \(X(\omega)\). The set
\[
\mathcal T_r(\omega)
:=
\left\{
s\le t:
|\Delta X_s(\omega)|\ge\frac r2
\right\}
\]
is finite. If \(A_k(r)^c\) occurs, then some
\(s\in\mathcal T_r(\omega)\) satisfies
$
|\Delta X_s(\omega)|\in \mathcal I_k(r).
$
Since every fixed jump magnitude can belong to at most two of the
intervals \(\mathcal I_k(r)\), it follows that
\[
\#\{k\ge1:\omega\in A_k(r)^c\}
\le
2\,\#\mathcal T_r(\omega)
<\infty.
\]
Therefore
\[
1_{A_k(r)^c}(\omega)\longrightarrow0
\qquad\text{as }k\to\infty
\]
for every càdlàg path. Hence
\[
1_{A_k(r)^c}\longrightarrow0
\qquad\text{almost surely}.
\]
Dominated convergence gives
\[
\P(A_k(r)^c)
=
\E\bigl[1_{A_k(r)^c}\bigr]
\longrightarrow0.
\]
In particular,
\[
\P\left(\bigcup_{k\ge1}A_k(r)\right)=1.
\]

Now fix \(k\ge1\). Since
\[
\sup_{s\le t}
|\Delta(X^n-X)_s|^p
\le
\sum_{s\le t}
|\Delta(X^n-X)_s|^p
\le
[X^n-X]^{(p)}_t,
\]
we obtain
\[
\sup_{s\le t}
|\Delta(X^n-X)_s|
\le
\left(
[X^n-X]^{(p)}_t
\right)^{1/p}.
\]
The assumption
\[
[X^n-X]^{(p)}_t\longrightarrow0
\qquad\text{in probability}
\]
therefore implies
\[
\sup_{s\le t}
|\Delta(X^n-X)_s|
\longrightarrow0
\qquad\text{in probability}.
\]
Consequently,
\[
\P\bigl(B_n^{(p)}(k,r)\bigr)
=
\P\left(
\sup_{s\le t}
|\Delta(X^n-X)_s|
<
\delta(k,r)
\right)
\longrightarrow1.
\]

We now work on the event
$
A_k(r)\cap B_n^{(p)}(k,r).
$
Let \(s\le t\).

If
\[
|\Delta X_s|
\le
L(r,k)-\delta(k,r),
\]
then
\[
\begin{aligned}
|\Delta X^n_s|
&\le
|\Delta X_s|
+
|\Delta(X^n-X)_s|
\\
&<
L(r,k)-\delta(k,r)+\delta(k,r)
=
L(r,k).
\end{aligned}
\]

If
\[
|\Delta X_s|
\ge
L(r,k)+\delta(k,r),
\]
then
\[
\begin{aligned}
|\Delta X^n_s|
&\ge
|\Delta X_s|
-
|\Delta(X^n-X)_s|
\\
&>
L(r,k)+\delta(k,r)-\delta(k,r)
=
L(r,k).
\end{aligned}
\]

On \(A_k(r)\), no jump magnitude of \(X\) belongs to the closed
threshold band
$
[L(r,k)-\delta(k,r),\,L(r,k)+\delta(k,r)].
$
It follows that, on \(A_k(r)\cap B_n^{(p)}(k,r)\), the classification
of the jumps relative to \(L(r,k)\) is preserved.

In particular,
\[
|\Delta X^n_s|\le L(r,k)
\quad\Longrightarrow\quad
|\Delta X_s|\le L(r,k),
\]
and hence
\[
1_{\{|\Delta X^n_s|\le L(r,k)\}}
\le
1_{\{|\Delta X_s|\le L(r,k)\}}
\qquad\text{for every }s\le t.
\]

Finally, using
$
\Delta X^n_s
=
\Delta(X^n-X)_s+\Delta X_s
$
and
\[
|u+v|^p
\le
2^{p-1}\bigl(|u|^p+|v|^p\bigr),
\qquad p\ge1,
\]
we obtain, on \(A_k(r)\cap B_n^{(p)}(k,r)\),
\begin{align*}
\sum_{s\le t}
|\Delta X^n_s|^p
1_{\{|\Delta X^n_s|\le L(r,k)\}}
&=
\sum_{s\le t}
|\Delta(X^n-X)_s+\Delta X_s|^p
1_{\{|\Delta X^n_s|\le L(r,k)\}}
\\
&\le
2^{p-1}
\sum_{s\le t}
|\Delta(X^n-X)_s|^p
+
2^{p-1}
\sum_{s\le t}
|\Delta X_s|^p
1_{\{|\Delta X^n_s|\le L(r,k)\}}
\\
&\le
2^{p-1}[X^n-X]^{(p)}_t
+
2^{p-1}
\sum_{s\le t}
|\Delta X_s|^p
1_{\{|\Delta X_s|\le L(r,k)\}}.
\end{align*}
This proves the result.
\end{proof}

\begin{proof}[Proof of Corollary \ref{cor:large-jump-atomic-integrands-same-fn}]
Let \(\Omega_1\) be a full-measure event such that, for every
\(\omega\in\Omega_1\),
\[
        [X^n-X]_t(\omega)\to0,
\]
and such that, for the fixed \(s\) appearing in the statement,
\[
        f(s,\cdot,\omega)
        \text{ is continuous at }\Delta X_s(\omega).
\]
We work on the full-measure event \(\Omega_0\cap\Omega_1\), where \(\Omega_0\)
is the event from the assumption.

Fix
\[
        \omega \in A_k(r)\cap \Omega_0\cap\Omega_1 \cap \{(\Delta X)^*_t<a\}.
\]

Since
\[
        \sup_{u\le t}|\Delta(X^{n}-X)_u|
        \le [X^{n}-X]_t^{1/2},
\]
and
$
        [X^{n}-X]_t(\omega)\to0,
$
there exists \(N_0(\omega)\in\mathbb N\) such that, for all
\(n\ge N_0(\omega)\),
\[
        \sup_{u\le t}
        |\Delta(X^{n}-X)_u|(\omega)
        <\delta(k,r).
\]
Hence
$
        \omega\in B_n^{(2)}(k,r)
$
for all sufficiently large \(n\). Therefore, for all sufficiently large
\(n\),
$
        \omega\in
        A_k(r)\cap B_n^{(2)}(k,r).
$

Since \(\omega\in A_k(r)\), either
$
        |\Delta X_s(\omega)|
        \le
        L(r,k)-\delta(k,r),
$
or
$
        |\Delta X_s(\omega)|
        \ge
        L(r,k)+\delta(k,r).
$

\medskip

\noindent

\textit{Case 1:}
\[
        |\Delta X_s(\omega)|
        \le
        L(r,k)-\delta(k,r).
\]

Since
$
|\Delta X^{n}_s(\omega)-\Delta X_s(\omega)|
<
\delta(k,r)
$
for all sufficiently large \(n\), we obtain
$
|\Delta X^{n}_s(\omega)|
<
L(r,k)
$
for all sufficiently large \(n\).

Thus neither \(\mu(\omega,\{s\},dx)\) nor
\(\mu^n(\omega,\{s\},dx)\) charges
$
        \{x:L(r,k)<|x|\le a\}.
$
Hence
\[
\int_{L(r,k)<|x|\le a}
f_{n}(s,x,\omega)\,
\mu^n(\omega,\{s\},dx)
-
\int_{L(r,k)<|x|\le a}
f_{n}(s,x,\omega)\,
\mu(\omega,\{s\},dx)
=0
\]
for all sufficiently large \(n\).
\medskip

\noindent
\textit{Case 2:}
\[
        |\Delta X_s(\omega)|
        \ge
        L(r,k)+\delta(k,r).
\]

Since
\[
|\Delta(X^{n}-X)_s(\omega)|
\le
[X^{n}-X]_t(\omega)^{1/2}
\to0,
\]
we have
$
\Delta X^{n}_s(\omega)\to\Delta X_s(\omega)
$.
Together with
$
|\Delta X_s(\omega)|\ge L(r,k)+\delta(k,r),
$
this implies
$
|\Delta X^{n}_s(\omega)|>L(r,k)
$
for all sufficiently large \(n\).

Since
$
        (\Delta X)^*_t(\omega)<a,
$
we have
$
        |\Delta X_s(\omega)|<a.
$
Together with
$
        \Delta X^{n}_s(\omega)\to\Delta X_s(\omega),
$
this implies that, for all sufficiently large \(n\),
$
        |\Delta X^{n}_s(\omega)|<a.
$
Therefore, for all sufficiently large \(n\),
\[
        1_{\{L(r,k)<|\Delta X^{n}_s(\omega)|\le a\}}
        =
        1_{\{L(r,k)<|\Delta X_s(\omega)|\le a\}}
        =1.
\]

Since
$
        \mu(\omega,\{s\},dx)
        =
        \delta_{\Delta X_s(\omega)}(dx),
$
and
$
        \mu^n(\omega,\{s\},dx)
        =
        \delta_{\Delta X^{n}_s(\omega)}(dx),
$
we get, for all sufficiently large \(n\),
\[
\begin{aligned}
\int_{L(r,k)<|x|\le a}
f_{n}(s,x,\omega)\,
\mu^n(\omega,\{s\},dx)
 &-
\int_{L(r,k)<|x|\le a}
f_{n}(s,x,\omega)\,
\mu(\omega,\{s\},dx)
\\
&=
f_{n}(s,\Delta X^{n}_s(\omega),\omega)
-
f_{n}(s,\Delta X_s(\omega),\omega).
\end{aligned}
\]

We show that the last expression converges to zero. Write
\[
x_n:=\Delta X^{n}_s(\omega),
\qquad
x:=\Delta X_s(\omega).
\]
Then \(x_n\to x\). Hence there exists a compact set \(K\subset\mathbb R\)
such that
\[
        x_n\in K
        \quad\text{for all sufficiently large }n,
        \qquad
        x\in K.
\]
For all sufficiently large \(n\),
\[
\begin{aligned}
|f_{n}(s,x_n,\omega)-f_{n}(s,x,\omega)|
&\le
|f_{n}(s,x_n,\omega)-f(s,x_n,\omega)|
+
|f(s,x_n,\omega)-f(s,x,\omega)|
+
|f(s,x,\omega)-f_{n}(s,x,\omega)|
\\
&\le
2\sup_{y\in K}
|f_{n}(s,y,\omega)-f(s,y,\omega)|
+
|f(s,x_n,\omega)-f(s,x,\omega)|.
\end{aligned}
\]
The first term tends to zero by local uniform convergence of
\(f_{n}(s,\cdot,\omega)\) to \(f(s,\cdot,\omega)\) on \(K\). The second
term tends to zero because \(f(s,\cdot,\omega)\) is continuous at
\(x=\Delta X_s(\omega)\). Therefore
\[
        f_{n}(s,\Delta X^{n}_s(\omega),\omega)
        -
        f_{n}(s,\Delta X_s(\omega),\omega)
        \to0.
\]

Consequently,
\[
\int_{L(r,k)<|x|\le a}
f_{n}(s,x,\omega)\,
\mu^n(\omega,\{s\},dx)
-
\int_{L(r,k)<|x|\le a}
f_{n}(s,x,\omega)\,
\mu(\omega,\{s\},dx)
\to0.
\]

Thus the asserted convergence holds for every
\[
        \omega \in A_k(r)\cap \Omega_0\cap\Omega_1 \cap \{(\Delta X)^*_t<a\}.
\]
Outside \(A_k(r)\), the prefactor \(1_{A_k(r)}\) makes the assertion
trivial. Since
$
        \mathbb P(\Omega_0\cap\Omega_1)=1,
$
the result follows.
\end{proof}

\begin{proof}[Proof of Lemma \ref{complemma}]
We first show that large predictable jump times coincide on the threshold-separated event.
%[Exact jump matching of large purely predictable jumps in probability]
\begin{claim}
\label{claim:exact-jump-matching-prob}
For each $n\ge1$, let $\{U_m\}_{m\ge1}$ and $\{U_m^n\}_{m\ge1}$ be exhausting sequences of predictable
stopping times for $\mathcal{A}$ and $\mathcal{A}_n$, respectively.

Fix $r>0$ and $k\ge1$, and let $A_k(r)$ and $B_n^{(p)}(k,r)$ be as in
Lemma~\ref{lem:p-thres}. Then for each $n\in\N$ there exists a set $\Omega_{n}$ with
\[
\mathbb P(\Omega_{n})=1
\]
such that for every
\[
\omega\in (A_k(r)\cap B_n^{(p)}(k,r))\cap \Omega_{n},
\]
one has
\[
\{U_m^n(\omega): |\Delta X^n_{U_m^n}(\omega)|>L(r,k)\}
=
\{U_m(\omega): |\Delta X_{U_m}(\omega)|>L(r,k)\}.
\]
\end{claim}

\begin{rema}
By an exhausting sequence we mean that for almost every $\omega$,
\[
\{s\le t:\nu_n(\omega,\{s\},\mathbb R)>0\}
=
\{U_m^n(\omega):m\ge1,\ U_m^n(\omega)\le t\},
\]
and similarly for $\{U_m\}$ and $\mathcal{A}$.
\end{rema}

\begin{proof}[Proof of the claim]
Fix $r>0$ and $k\ge1$. For each $m\ge1$, since $U_m$ is predictable, the map
\[
(s,x)\mapsto 1_{\{s=U_m\}}\,1_{\{\nu_n(\{s\},\mathbb R)=0\}}\,|x|
\]
is predictable. Hence the compensator identity for $X^n$ yields
\begin{align*}
0
=
\int_{\mathbb R}
|x|\,1_{\{\nu_n(\{U_m\},\mathbb R)=0\}}\,\nu_n(\{U_m\},dx)
&=
\E\!\left[
\int_{\mathbb R}
|x|\,1_{\{\nu_n(\{U_m\},\mathbb R)=0\}}\,\mu_n(\{U_m\},dx)
\Bigm| \mathcal F_{U_m-}
\right]
\\
&=
\E\!\left[
1_{\{\nu_n(\{U_m\},\mathbb R)=0\}}\,|\Delta X^n_{U_m}|
\Bigm| \mathcal F_{U_m-}
\right].
\end{align*}
Since the conditional expectation is nonnegative, it follows that
\[
1_{\{\nu_n(\{U_m\},\mathbb R)=0\}}\,|\Delta X^n_{U_m}|=0
\qquad \text{a.s.}
\]
Let $\Omega^{(1)}_{m,n}$ be a full-measure event on which the implication
\[
|\Delta X^n_{U_m}|>0
\quad\Longrightarrow\quad
\nu_n(\{U_m\},\mathbb R)>0
\]
holds.

Similarly, for each $m\ge1$, since $U_m^n$ is predictable, the map
\[
(s,x)\mapsto 1_{\{s=U_m^n\}}\,1_{\{\nu(\{s\},\mathbb R)=0\}}\,|x|
\]
is predictable. Hence the compensator identity for $X$ yields
\begin{align*}
0
=
\int_{\mathbb R}
|x|\,1_{\{\nu(\{U_m^n\},\mathbb R)=0\}}\,\nu(\{U_m^n\},dx)
&=
\E\!\left[
\int_{\mathbb R}
|x|\,1_{\{\nu(\{U_m^n\},\mathbb R)=0\}}\,\mu(\{U_m^n\},dx)
\Bigm| \mathcal F_{U_m^n-}
\right]
\\
&=
\E\!\left[
1_{\{\nu(\{U_m^n\},\mathbb R)=0\}}\,|\Delta X_{U_m^n}|
\Bigm| \mathcal F_{U_m^n-}
\right].
\end{align*}
Hence
\[
1_{\{\nu(\{U_m^n\},\mathbb R)=0\}}\,|\Delta X_{U_m^n}|=0
\qquad \text{a.s.}
\]
Let $\Omega^{(2)}_{m,n}$ be a full-measure event on which the implication
\[
|\Delta X_{U_m^n}|>0
\quad\Longrightarrow\quad
\nu(\{U_m^n\},\mathbb R)>0
\]
holds.

Now define
\[
\Omega_{n}
:=
\bigcap_{m=1}^\infty \Omega^{(1)}_{m,n}
\cap
\bigcap_{m=1}^\infty \Omega^{(2)}_{m,n}.
\]
Since this is a countable intersection of full-measure events,
$
\mathbb P(\Omega_{n})=1.
$
We now prove that for every
\[
\omega\in (A_k(r)\cap B_n^{(p)}(k,r))\cap \Omega_{n},
\]
one has
\[
\{U_m(\omega): |\Delta X_{U_m}(\omega)|>L(r,k)\}
\subseteq
\{U_m^n(\omega): |\Delta X^n_{U_m^n}(\omega)|>L(r,k)\}.
\]

Fix such an $\omega$, and let $m\ge1$ satisfy
$
|\Delta X_{U_m}(\omega)|>L(r,k).
$
By Lemma~\ref{lem:p-thres}, the classification relative to $L(r,k)$ is preserved
on $A_k(r)\cap B_n^{(p)}(k,r)$, hence
$
|\Delta X^n_{U_m}(\omega)|>L(r,k).
$
In particular,
$
|\Delta X^n_{U_m}(\omega)|>0.
$
Since $\omega\in\Omega_{n}\subseteq \Omega^{(1)}_{m,n}$, it follows that
$
\nu_n(\omega,\{U_m(\omega)\},\mathbb R)>0.
$
Because $\{U_j^n\}_{j\ge1}$ exhausts the predictable jump times of $\mathcal{A}_n$, there exists
$j=j(\omega,m,n)$ such that
$
U_m(\omega)=U_j^n(\omega).
$
Moreover,
\[
|\Delta X^n_{U_j^n}(\omega)|
=
|\Delta X^n_{U_m}(\omega)|
>
L(r,k).
\]
Thus
\[
U_m(\omega)\in \{U_j^n(\omega): |\Delta X^n_{U_j^n}(\omega)|>L(r,k)\}.
\]
Since $m$ was arbitrary, the first inclusion follows.

We next prove that for every
\[
\omega\in (A_k(r)\cap B_n^{(p)}(k,r))\cap \Omega_{n},
\]
one has
\[
\{U_m^n(\omega): |\Delta X^n_{U_m^n}(\omega)|>L(r,k)\}
\subseteq
\{U_m(\omega): |\Delta X_{U_m}(\omega)|>L(r,k)\}.
\]

Fix such an $\omega$, and let $m\ge1$ satisfy
$
|\Delta X^n_{U_m^n}(\omega)|>L(r,k).
$
Again by Lemma~\ref{lem:p-thres},
$
|\Delta X_{U_m^n}(\omega)|>L(r,k),
$
hence in particular
$
|\Delta X_{U_m^n}(\omega)|>0.
$
Since $\omega\in\Omega_{n}\subseteq \Omega^{(2)}_{m,n}$, it follows that
$
\nu(\omega,\{U_m^n(\omega)\},\mathbb R)>0.
$
Because $\{U_j\}_{j\ge1}$ exhausts the predictable jump times of $\mathcal{A}$, there exists
$j=j(\omega,m,n)$ such that
$
U_m^n(\omega)=U_j(\omega).
$
Moreover,
\[
|\Delta X_{U_j}(\omega)|
=
|\Delta X_{U_m^n}(\omega)|
>
L(r,k).
\]
Thus
\[
U_m^n(\omega)\in \{U_j(\omega): |\Delta X_{U_j}(\omega)|>L(r,k)\}.
\]
Since $m$ was arbitrary, the reverse inclusion follows.

Combining the two inclusions, we conclude that for every
\[
\omega\in (A_k(r)\cap B_n^{(p)}(k,r))\cap \Omega_{n},
\]
\[
\{U_m^n(\omega): |\Delta X^n_{U_m^n}(\omega)|>L(r,k)\}
=
\{U_m(\omega): |\Delta X_{U_m}(\omega)|>L(r,k)\}.
\]
This proves the claim.
\end{proof}
Fix \(r>0\), \(k\ge1\), and \(q,\ell\ge1\). Define
\[
\begin{aligned}
&N_u
:=
\sum_{s\le u}1_{s\in\mathcal A}1_{|\Delta X_s|>L(r,k)},
\qquad
N_{u,n}
:=
\sum_{s\le u}1_{s\in\mathcal A_n}1_{|\Delta X^n_s|>L(r,k)},\\[0.5em]
&\tau_q^X
:=
\inf\{u\ge0:N_u\ge q\},
\qquad
\tau_{q,n}
:=
\inf\{u\ge0:N_{u,n}\ge q\}\wedge \tau_q^X,\\[0.5em]
&\sigma_k
:=
\inf\Bigl\{
s\in[0,t]:
\bigl||\Delta X_s|-L(r,k)\bigr|
\le\delta(k,r)
\Bigr\},
\\
&\eta_{n,k}
:=
\inf\Bigl\{
s\in[0,t]:
(\Delta(X^n-X))^*_s
\ge \delta(k,r)
\Bigr\},\\[0.5em]
&\eta_k^{(\ell)}
:=
\inf_{n\ge\ell}\eta_{n,k},
\qquad
\zeta_{\ell,k}
:=
\eta_k^{(\ell)}\wedge\sigma_k,
\qquad
\theta_{q,k,\ell,n}
:=
\zeta_{\ell,k}\wedge\tau_{q,n}.
\end{aligned}
\]
Since $N$ and $N_{\cdot,n}$ are adapted increasing càdlàg counting
processes, the hitting times
\[
\tau^X_q:=\inf\{u\ge0:N_u\ge q\},
        \qquad
\inf\{u\ge0:N_{u,n}\ge q\}
\]
are stopping times. Indeed, for every $u\ge0$,
\[
\{\tau^X_q\le u\}=\{N_u\ge q\}\in\mathcal F_u,
\]
and similarly for $N_{\cdot,n}$. Hence
\[
\tau_{q,n}:=\inf\{u\ge0:N_{u,n}\ge q\}\wedge \tau^X_q
\]
is a stopping time.
Every time counted by \(N_t\) contributes more than
\(L(r,k)^p\) to the \(p\)-th jump energy of \(X\), and similarly for
\(N_{t,n}\) and \(X^n\). Hence, by
Lemma~\ref{lem:canonical-p-power-basic},
\[
\begin{aligned}
\P(\tau_{q,n}\le t)
&\le
\P(N_t\ge q)+\P(N_{t,n}\ge q)
\\
&\le
\P\left(
[X]^{(p)}_t\ge qL(r,k)^p
\right)
+
\P\left(
[X^n]^{(p)}_t\ge qL(r,k)^p
\right).
\end{aligned}
\]

The family
\[
        \left\{[X^n]^{(p)}_t:n\ge1\right\}
\]
is bounded in probability. Indeed,
Lemma~\ref{lem:canonical-p-power-basic} gives
\[
[X^n]^{(p)}_t
\le
2^{p-1}
\left(
[X^n-X]^{(p)}_t+[X]^{(p)}_t
\right).
\]
Since
\[
        [X^n-X]^{(p)}_t\longrightarrow0
        \qquad\text{almost surely},
\]
the family
\(
\{[X^n-X]^{(p)}_t:n\ge1\}
\)
is bounded in probability. Therefore
\[
\lim_{q\to\infty}
\sup_n
\P\left(
[X^n]^{(p)}_t\ge qL(r,k)^p
\right)
=
0.
\]
Since \([X]^{(p)}_t<\infty\) almost surely, also
\[
\lim_{q\to\infty}
\P\left(
[X]^{(p)}_t\ge qL(r,k)^p
\right)
=
0.
\]
Consequently,
\[
\lim_{q\to\infty}
\sup_n
\P(\tau_{q,n}\le t)
=
0.
\]

Moreover,
\[
\{\theta_{q,k,\ell,n}\le t\}
\subseteq
\{\sigma_k\le t\}
\cup
\{\eta_k^{(\ell)}\le t\}
\cup
\{\tau_{q,n}\le t\}.
\]
By the definition of \(\sigma_k\),
\[
        \{\sigma_k\le t\}=A_k(r)^c.
\]
Hence Lemma~\ref{lem:p-thres} gives
\[
        \mathbb P(\sigma_k\le t)
        =
        \mathbb P(A_k(r)^c)
        \longrightarrow0
        \qquad\text{as }k\to\infty.
\]
For fixed \(k\),
\[
\sup_{s\le t}
|\Delta(X^n-X)_s|
\le
\left(
[X^n-X]^{(p)}_t
\right)^{1/p}
\longrightarrow0
\qquad\text{almost surely}.
\]
Thus, almost surely, \(\eta_{n,k}=\infty\) for every sufficiently
large \(n\). Consequently,
\[
        \eta_k^{(\ell)}\uparrow\infty
        \qquad\text{almost surely as }\ell\to\infty,
\]
and hence
\[
        \mathbb P(\eta_k^{(\ell)}\le t)\to0
        \qquad\text{as } \ell\to\infty .
\]
Combining this with
\[
        \lim_{q\to\infty}\sup_n
        \mathbb P(\tau_{q,n}\le t)=0,
\]
we obtain the asserted localization property: for every \(\varepsilon>0\)
there exists \(k_0\) such that for all \(k\ge k_0\) there exists
\(\ell_0=\ell_0(k,\varepsilon)\) such that for all \(\ell\ge\ell_0\)
there exists \(q_0=q_0(k,\ell,\varepsilon)\) such that for all \(q\ge q_0\),
\[
        \sup_{n\ge\ell}
        \mathbb P(\theta_{q,k,\ell,n}\le t)
        <\varepsilon .
\]
This proves assertion~(2).
We shall use two standard facts from the theory of predictable sets.
First, every thin predictable set admits an exhaustion by predictable
stopping times. Second, if \(B\) and \(C\) are predictable sets and
\(B\) is thin, then \(B\setminus C\) is again a thin predictable set.

We shall also use the following standard fact. If \(U\) is a predictable
stopping time and \(\tau\) is a stopping time, then
\(\{U\le \tau\}\in\mathcal F_{U-}\) (see e.g. \cite{JacodShiryaev}).
Hence the restricted time
\[
U1_{\{U\le \tau\}}+\infty 1_{\{U>\tau\}}
\]
is predictable. Consequently, restricting a predictable graph to
\(\{s\le\tau\}\) preserves predictability.
For fixed \(\ell\), let \(\{T_{\ell,\ell,m}\}_{m\in\mathbb N}\) be a
predictable exhaustion of the thin predictable set
\[
\mathcal{A}_\ell\cap
\left\{
s:
\int_{|x|>L(r,k)}\nu_\ell(\{s\},dx)>0
\right\}
\cap \{s\le \tau_{q,\ell}\}.
\]
Inductively, having chosen the families
\(\{T_{\ell,j,m}\}_{m\in\mathbb N}\), \(\ell\le j\le n\), let
\(\{T_{\ell,n+1,m}\}_{m\in\mathbb N}\) be a predictable exhaustion of
the thin predictable set
\[
\mathcal{A}_{n+1}\cap
\left\{
s:
\int_{|x|>L(r,k)}\nu_{n+1}(\{s\},dx)>0
\right\}
\cap \{s\le \tau_{q,n+1}\}
\setminus
\bigcup_{j=\ell}^{n}\bigcup_{m\ge1} [[T_{\ell,j,m}]] .
\]
Since the predictable \(\sigma\)-field is stable under countable unions
and set differences, each recursive remainder is again a thin
predictable set. Hence each such
remainder admits a predictable exhaustion, and the graphs
\([[T_{\ell,n,m}]]\), \(n\ge\ell\), \(m\ge1\), may be chosen pairwise
disjoint, up to evanescence.
For \(j\ge\ell\), define
\[
\gamma_j(s)
:=
\bigl(
1_{\{\cdot<\theta_{q,k,\ell,j}\}}
\bigr)_{s-},
\qquad
a_j(s)
:=
\int_{\{|x|>L(r,k)\}}
\nu_j(\{s\},dx),
\]
and set
\[
        w_j(s):=\gamma_j(s)a_j(s),
        \qquad
        w(s):=\sup_{j\ge\ell}w_j(s).
\]
The processes \(\gamma_j\) and \(a_j\) are predictable. Indeed,
\(a_j\) is the jump process of the predictable increasing process
\[
u\longmapsto
\int_{(0,u]}\int_{\{|x|>L(r,k)\}}
\nu_j(ds,dx).
\]
Consequently, \(w_j\) and the countable supremum \(w\) are
predictable.

Every time \(s\in(0,t]\) for which \(w(s)>0\) belongs to at least one
of the thin predictable sets exhausted by the recursively constructed
families \(\{T_{\ell,n,m}\}_{m\ge1}\). Indeed, \(w(s)>0\) implies that
\(w_j(s)>0\) for some \(j\ge\ell\), and then
\[
        a_j(s)>0,
        \qquad
        s\le\theta_{q,k,\ell,j}
        \le\tau_{q,j}.
\]
Because the recursive construction removes every graph already
selected at an earlier stage, each such time occurs on exactly one of
the graphs \([[T_{\ell,n,m}]]\). Therefore
\[
\sum_{0<s\le t}w(s)
=
\sum_{n=\ell}^{\infty}
\sum_{m=1}^{\infty}
1_{\{T_{\ell,n,m}\le t\}}
w(T_{\ell,n,m})
\qquad\text{a.s.}
\]
Notice that this identity concerns the supremum \(w\) at the selected
time. It does not assert that the graph assigned to the index \(n\)
represents all individual \(j\)-terms occurring at that time.

By construction, no jump magnitude \(|\Delta X_s|\) before
\(\zeta_{\ell,k}\) lies within distance \(\delta(k,r)\) of the threshold
\(L(r,k)\), while for all \(n\ge \ell\),
\[
\sup_{s<\zeta_{\ell,k}} |\Delta(X^n-X)_s|\le \delta(k,r).
\]
Therefore, by the threshold-separated alignment claim, outside a null
set, for every \(n\ge\ell\),
\[
\{s\in\mathcal A_n\cap[0,\zeta_{\ell,k}):|\Delta X^n_s|>L(r,k)\}
=
\{s\in\mathcal A\cap[0,\zeta_{\ell,k}):|\Delta X_s|>L(r,k)\}.
\]
Restricting both sides to \([0,\theta_{q,k,\ell,n})\), and using
\(\theta_{q,k,\ell,n}\le\zeta_{\ell,k}\), gives
\[
\{s\in \mathcal A_n\cap[0,\theta_{q,k,\ell,n}):|\Delta X^n_s|>L(r,k)\}
=
\{s\in \mathcal A\cap[0,\theta_{q,k,\ell,n}):|\Delta X_s|>L(r,k)\}
\quad\text{a.s.}
\]
This proves assertion~(1).

We next prove the counting bound
\[
\textbf{Claim:}\qquad
\sum_{n=\ell}^{\infty}
\sum_{m=1}^{\infty}
1_{\{T_{\ell,n,m}\le t\}}
\sup_{j\ge\ell}
\left[
\gamma_j(T_{\ell,n,m})
1_{\{|\Delta X^j_{T_{\ell,n,m}}|>L(r,k)\}}
\right]
\le q+1
\qquad\text{a.s.}
\]

\begin{proof}[Proof of the counting bound]
Applying the support-implication argument from
Claim~\ref{claim:exact-jump-matching-prob} to the countable family of
predictable stopping times
\[
        \{T_{\ell,n,m}:n\ge\ell,\ m\ge1\}
\]
and to every \(X^j\), \(j\ge\ell\), we may work outside a single null
set on which
\[
|\Delta X^j_{T_{\ell,n,m}}|>0
\quad\Longrightarrow\quad
T_{\ell,n,m}\in\mathcal A_j
\]
holds simultaneously for all \(j,n,m\). We also exclude the null set
on which the threshold-separated matching identities on
\([0,\zeta_{\ell,k})\) fail.

Fix an outcome outside these null sets. Suppose that, for some
\(n\ge\ell\) and \(m\ge1\),
\[
1_{\{T_{\ell,n,m}\le t\}}
\sup_{j\ge\ell}
\left[
\gamma_j(T_{\ell,n,m})
1_{\{|\Delta X^j_{T_{\ell,n,m}}|>L(r,k)\}}
\right]
=1.
\]
Then \(T_{\ell,n,m}\le t\). Since all the quantities inside the
supremum take values in \(\{0,1\}\), there exists \(j\ge\ell\) such
that, writing
\[
        T:=T_{\ell,n,m},
\]
we have
\[
        \gamma_j(T)=1,
        \qquad
        |\Delta X^j_T|>L(r,k).
\]
The support implication gives \(T\in\mathcal A_j\). Moreover,
\[
\gamma_j(T)=1
\quad\Longrightarrow\quad
T\le\theta_{q,k,\ell,j}
\le
\zeta_{\ell,k}\wedge\tau_q^X.
\]

If \(T<\zeta_{\ell,k}\), the exact matching identity gives
\[
        T\in\mathcal A,
        \qquad
        |\Delta X_T|>L(r,k).
\]
Since also \(T\le\tau_q^X\), this is one of the at most \(q\)
distinct times counted by \(N\) no later than \(\tau_q^X\).

If \(T=\zeta_{\ell,k}\), the matching identity on
\([0,\zeta_{\ell,k})\) need not apply. However, the graphs
\([[T_{\ell,n,m}]]\) are pairwise disjoint. Hence, for each fixed
outcome, at most one pair \((n,m)\) can contribute at
\(\zeta_{\ell,k}\).

Thus at most \(q\) nonzero terms arise strictly before
\(\zeta_{\ell,k}\), and at most one additional term can arise at
\(\zeta_{\ell,k}\). Therefore
\[
\sum_{n=\ell}^{\infty}
\sum_{m=1}^{\infty}
1_{\{T_{\ell,n,m}\le t\}}
\sup_{j\ge\ell}
\left[
\gamma_j(T_{\ell,n,m})
1_{\{|\Delta X^j_{T_{\ell,n,m}}|>L(r,k)\}}
\right]
\le q+1.
\]
This proves the claim.
\end{proof}

Fix \(n\ge\ell\), \(m\ge1\), and write
\[
        T:=T_{\ell,n,m}.
\]
For every \(j\ge\ell\), define
\[
Z_j(T)
:=
1_{\{T\le t\}}
\gamma_j(T)
1_{\{|\Delta X^j_T|>L(r,k)\}}.
\]
Since \(T\) is predictable and
\(\gamma_j(T)1_{\{T\le t\}}\) is
\(\mathcal F_{T-}\)-measurable, the conditional compensator identity
gives
\[
\begin{aligned}
1_{\{T\le t\}}
\gamma_j(T)
\int_{\{|x|>L(r,k)\}}
\nu_j(\{T\},dx)
&=
\E\left[
1_{\{T\le t\}}
\gamma_j(T)
1_{\{|\Delta X^j_T|>L(r,k)\}}
\bigm|
\mathcal F_{T-}
\right]
\\
&=
\E\left[
Z_j(T)
\bigm|
\mathcal F_{T-}
\right].
\end{aligned}
\]
Since the family of indices \((n,m,j)\) is countable, these identities
may be chosen to hold simultaneously.

It follows that
\[
\begin{aligned}
1_{\{T\le t\}}w(T)
&=
\sup_{j\ge\ell}
\E\left[
Z_j(T)
\bigm|
\mathcal F_{T-}
\right]
\\
&\le
\E\left[
\sup_{j\ge\ell}Z_j(T)
\bigm|
\mathcal F_{T-}
\right].
\end{aligned}
\]
Here the inequality follows from the monotonicity of conditional
expectation, since
\[
        Z_j(T)\le\sup_{i\ge\ell}Z_i(T)
        \qquad\text{for every }j\ge\ell.
\]

Using the exhaustion identity for \(w\), conditional expectation,
monotone convergence, and the counting bound proved above, we obtain
\[
\begin{aligned}
\E\left[
\sum_{0<s\le t}w(s)
\right]
&=
\E\left[
\sum_{n=\ell}^{\infty}
\sum_{m=1}^{\infty}
1_{\{T_{\ell,n,m}\le t\}}
w(T_{\ell,n,m})
\right]
\\
&\le
\E\left[
\sum_{n=\ell}^{\infty}
\sum_{m=1}^{\infty}
\sup_{j\ge\ell}
Z_j(T_{\ell,n,m})
\right]
\le
q+1.
\end{aligned}
\]
Consequently,
\[
\sum_{0<s\le t}w(s)
<\infty
\qquad\text{a.s.}
\]
Equivalently,
\[
\sum_{0<s\le t}
\sup_{n\ge\ell}
\bigl(
1_{\{\cdot<\theta_{q,k,\ell,n}\}}
\bigr)_{s-}
\int_{\{|x|>L(r,k)\}}
\nu_n(\{s\},dx)
<\infty
\qquad\text{a.s.}
\]
Choose a predictable exhaustion
\(\{U_l\}_{l\in\N}\) of
\(\mathcal A\cap\{s\le\tau_q^X\}\)
whose graphs are pairwise disjoint, up to evanescence. Then, by
monotone convergence,
\begin{align*}
\E\left[\sum_{s\le \tau_{q}^X} \int_{|x|>L(r,k)}\nu(\{s\},dx)\right]
&=
\E\left[\sum_{l=1}^\infty \int_{|x|>L(r,k)}\nu(\{U_l\},dx)\right]
\\
&=
\sum_{l=1}^\infty\E\left[\int_{|x|>L(r,k)}\nu(\{U_l\},dx)\right]
\\
&=
\sum_{l=1}^\infty\E\left[\int_{|x|>L(r,k)}\mu(\{U_l\},dx)\right]
\\
&=
\E\left[\sum_{l=1}^\infty 1_{|\Delta X_{U_l}|>L(r,k)}1_{U_l\le \tau_{q}^X}\right]
=
\E\left[\sum_{s\in\mathcal{A}} 1_{|\Delta X_{s}| > L(r,k)}1_{s\le \tau_{q}^X}\right]\le q,
\end{align*}
Here we used that, on \(\{U_l<\infty\}\),
\[
\int_{\{|x|>L(r,k)\}}\mu(\{U_l\},dx)
=
1_{\{|\Delta X_{U_l}|>L(r,k)\}},
\]
whereas both sides in the preceding sum are understood to be zero on
\(\{U_l=\infty\}\).
This implies that $$\sum_{s\le \tau_{q}^X} \int_{|x|>L(r,k)}\nu(\{s\},dx)<\infty\quad\textsf{a.s..}$$ 
Since \(\theta_{q,k,\ell,n}\le \tau_q^X\) for all \(n\ge\ell\), the factor
$
\sup_{n\ge\ell}
\bigl(1_{\{\cdot<\theta_{q,k,\ell,n}\}}\bigr)_{s-}
$
can be nonzero only at times \(s\le \tau_q^X\). Hence
\[
\sum_{s\le t}\sup_{n\ge\ell}
\bigl(1_{\{\cdot<\theta_{q,k,\ell,n}\}}\bigr)_{s-}
\int_{|x|>L(r,k)}\nu(\{s\},dx)
\le
\sum_{s\le \tau_q^X}
\int_{|x|>L(r,k)}\nu(\{s\},dx)
<\infty \quad\textsf{a.s.}.
\]
Finally, for every \(s\le t\),
\[
\begin{aligned}
&\sup_{n\ge\ell}
\bigl(
1_{\{\cdot<\theta_{q,k,\ell,n}\}}
\bigr)_{s-}
\left(
\int_{\{|x|>L(r,k)\}}\nu_n(\{s\},dx)
+
\int_{\{|x|>L(r,k)\}}\nu(\{s\},dx)
\right)
\\
&\qquad\le
w(s)
+
\sup_{n\ge\ell}
\bigl(
1_{\{\cdot<\theta_{q,k,\ell,n}\}}
\bigr)_{s-}
\int_{\{|x|>L(r,k)\}}\nu(\{s\},dx).
\end{aligned}
\]
The two terms on the right-hand side are summable over \(s\le t\)
almost surely by the preceding estimates. Hence
\[
\sum_{s\le t}
\sup_{n\ge\ell}
\bigl(
1_{\{\cdot<\theta_{q,k,\ell,n}\}}
\bigr)_{s-}
\left(
\int_{\{|x|>L(r,k)\}}\nu_n(\{s\},dx)
+
\int_{\{|x|>L(r,k)\}}\nu(\{s\},dx)
\right)
<\infty
\qquad\textsf{a.s.}
\]
This proves assertion~(3) and completes the proof.
\end{proof}

\begin{proof}[Proof of Corollary~\ref{cor:annular-compensator-atom-mass-tightness}]
For every \(x\in E\),
\[
        |x|^p\ge b^p.
\]
Consequently, for every \(n\ge1\),
\[
\begin{aligned}
\sum_{0<s\le t}\nu^n(\{s\},E)
&\le
\nu^n((0,t]\times E)
\\
&\le
b^{-p}
\int_{(0,t]}\int_E
|x|^p\,\nu^n(ds,dx)
\\
&\le
b^{-p}
\int_{(0,t]}\int_{\{|x|\le a\}}
|x|^p\,\nu^n(ds,dx).
\end{aligned}
\]
The family of random variables on the last line is bounded in
probability by
Lemma~\ref{lem:bounded-mark-predictable-p-energy}. Hence
\[
\left\{
        \sum_{0<s\le t}\nu^n(\{s\},E)
        :
        n\ge1
\right\}
\]
is bounded in probability.
\end{proof}

\subsection{Completion of the proof of Theorem \ref{thm:CMC-pvar-compensator-stability}}

\begin{proof}[Proof of Theorem~\ref{thm:CMC-pvar-compensator-stability}]
Fix \(t<\infty\), and let \((n_j)\) be an arbitrary subsequence.
There exists a further subsequence, still denoted by \((n_j)\), such
that
\[
        [X^{n_j}-X]^{(p)}_t
        \longrightarrow0
        \qquad\text{a.s.}
\]
We work along this further subsequence and suppress the subsequence
notation. For notational convenience, throughout the proof we write
\[
        \mu_n:=\mu^n,
        \qquad
        \nu_n:=\nu^n.
\]
At the end, the subsequence criterion for convergence in probability
will yield the assertion for the original sequence.

Since \(A^n-A\) is a finite-variation process and \(p>1\), its
continuous part has zero \(p\)-th power variation. Hence
\[
[A^n-A]^{(p)}_t
=
\sum_{s\le t}
\left|
\int_{\mathbb R}f_n(s,x)\,\nu_n(\{s\},dx)
-
\int_{\mathbb R}f(s,x)\,\nu(\{s\},dx)
\right|^p.
\]
The cross-atomic integral
\[
        \int_{\mathbb R}f_n(s,x)\,\nu(\{s\},dx)
\]
is well defined: on every bounded mark region this follows from
\({\rm (LG}_0{\rm )}\), while the tail is covered by the quantities
\(\widehat\beta_n^{(a)}\) appearing in \((\mathrm{BE}_p)\).

Therefore, by the \(\ell^p\)-triangle inequality,
\begin{align}
\bigl([A^n-A]^{(p)}_t\bigr)^{1/p}
&\le
\left(
\sum_{s\le t}
\left|
\int_{\mathbb R}
\bigl(f_n(s,x)-f(s,x)\bigr)\,
\nu(\{s\},dx)
\right|^p
\right)^{1/p}
\nonumber\\
&\quad+
\left(
\sum_{s\le t}
\left|
\int_{\mathbb R}
f_n(s,x)\,
\bigl(\nu-\nu_n\bigr)(\{s\},dx)
\right|^p
\right)^{1/p}.
\label{eq:pvar-main-split}
\end{align}

\medskip
\noindent
\textbf{Step 1: The fixed-compensator integrand term.}
\medskip

We first prove that
\[
I_n(t)
:=
\left(
\sum_{s\le t}
\left|
\int_{\mathbb R}
\bigl(f_n(s,x)-f(s,x)\bigr)\,
\nu(\{s\},dx)
\right|^p
\right)^{1/p}
\longrightarrow0
\]
in probability.

Let
\[
        d:=R_0,
\]
where \(R_0\) is the radius appearing in
\({\rm (LG}_0{\rm )}\), and fix \(a>d\). By the
\(\ell^p\)-triangle inequality,
\[
        I_n(t)
        \le
        I_n^{(0,d)}(t)
        +
        I_n^{(d,a)}(t)
        +
        I_n^{(a,\infty)}(t),
\]
where
\[
I_n^{(0,d)}(t)
:=
\left(
\sum_{s\le t}
\left|
\int_{\{|x|\le d\}}
(f_n-f)(s,x)\,\nu(\{s\},dx)
\right|^p
\right)^{1/p},
\]
\[
I_n^{(d,a)}(t)
:=
\left(
\sum_{s\le t}
\left|
\int_{\{d<|x|\le a\}}
(f_n-f)(s,x)\,\nu(\{s\},dx)
\right|^p
\right)^{1/p},
\]
and
\[
I_n^{(a,\infty)}(t)
:=
\left(
\sum_{s\le t}
\left|
\int_{\{|x|>a\}}
(f_n-f)(s,x)\,\nu(\{s\},dx)
\right|^p
\right)^{1/p}.
\]

We treat these three terms separately.

\medskip
\noindent
\textit{Step 1a: The contribution from \(\{|x|\le d\}\).}
\medskip

Since
\[
        \nu(\{s\},\mathbb R)\le1
\]
up to evanescence, Jensen's inequality gives, at every compensator
atom,
\[
\begin{aligned}
\left|
\int_{\{|x|\le d\}}
(f_n-f)(s,x)\,\nu(\{s\},dx)
\right|^p
&\le
\nu(\{s\},\{|x|\le d\})^{p-1}
\int_{\{|x|\le d\}}
|f_n(s,x)-f(s,x)|^p\,\nu(\{s\},dx)
\\
&\le
\int_{\{|x|\le d\}}
|f_n(s,x)-f(s,x)|^p\,\nu(\{s\},dx).
\end{aligned}
\]
Consequently,
\begin{align*}
\bigl(I_n^{(0,d)}(t)\bigr)^p
&\le
\sum_{s\le t}\int_{\{|x|\le d\}}
|f_n(s,x)-f(s,x)|^p\,\nu(\{s\},dx)
\\
&\le
\int_0^t\int_{\{|x|\le d\}}
|f_n(s,x)-f(s,x)|^p\,\nu(ds,dx).
\end{align*}

For \(m,q\in\mathbb N\), set
\[
        C_t^*
        :=
        \sup_{0\le u\le t}C^{R_0}_u,
        \qquad
        \chi_m(s)
        :=
        1_{\{C^{R_0}_s\le m\}},
\]
and define
\[
        \sigma_q
        :=
        \inf\{u>0:[X]^{(p)}_u\ge q\},
        \qquad
        \gamma_q(s)
        :=
        \bigl(1_{\{\cdot<\sigma_q\}}\bigr)_{s-},
\]
with \(\inf\varnothing:=\infty\). Then
$
        \gamma_q(s)=1_{\{s\le\sigma_q\}},
$
and both \(\chi_m\) and \(\gamma_q\) are predictable.

Set
\[
Z_n^{m,q}
:=
\int_0^t
\chi_m(s)\gamma_q(s)
\int_{\{|x|\le d\}}
|f_n(s,x)-f(s,x)|^p\,\nu(ds,dx).
\]
The compensator identity gives
\[
\begin{aligned}
\E[Z_n^{m,q}]
=
\E\left[
\sum_{s\le t}
\chi_m(s)\gamma_q(s)
|f_n(s,\Delta X_s)-f(s,\Delta X_s)|^p
1_{\{|\Delta X_s|\le d\}}
\right].
\end{aligned}
\]

By \({\rm (LU)}\), outside a fixed null set,
\[
        f_n(s,\Delta X_s)-f(s,\Delta X_s)
        \longrightarrow0
\]
simultaneously for every \(s\le t\). Moreover, on
\(\{\chi_m(s)=1\}\), condition \({\rm (LG}_0{\rm )}\) gives
\[
|f_n(s,\Delta X_s)-f(s,\Delta X_s)|
\le
m|\Delta X_s|
\]
whenever \(|\Delta X_s|\le d=R_0\).

Since the predictable localization includes at most the single
boundary time \(\sigma_q\),
\[
\begin{aligned}
\sum_{s\le t}
\chi_m(s)\gamma_q(s)
|f_n(s,\Delta X_s)-f(s,\Delta X_s)|^p
1_{\{|\Delta X_s|\le d\}}
&\le
m^p
\sum_{s\le t}
\gamma_q(s)
|\Delta X_s|^p
1_{\{|\Delta X_s|\le d\}}
\\
&\le
m^p(q+d^p).
\end{aligned}
\]
Indeed, the jumps strictly before \(\sigma_q\) contribute at most
\(q\), while the possible boundary jump contributes at most \(d^p\).
Dominated convergence therefore yields
$
        \E[Z_n^{m,q}]\longrightarrow0.
$
In particular,
$
        Z_n^{m,q}\longrightarrow0
$
in probability.
Define
\[
I_{n,m,q}^{(0,d)}(t)
:=
\left(
\sum_{s\le t}
\chi_m(s)\gamma_q(s)
\left|
\int_{\{|x|\le d\}}
(f_n-f)(s,x)\,\nu(\{s\},dx)
\right|^p
\right)^{1/p}.
\]
The preceding atomwise Jensen estimate gives
$
        \bigl(I_{n,m,q}^{(0,d)}(t)\bigr)^p
        \le Z_n^{m,q}.
$
Thus
$
        I_{n,m,q}^{(0,d)}(t)
        \longrightarrow0$
 in probability.
On the event
$
        \{C_t^*\le m\}\cap\{\sigma_q>t\},
$
we have
$
        I_n^{(0,d)}(t)=I_{n,m,q}^{(0,d)}(t).
$
Consequently,
\[
\begin{aligned}
\P\bigl(I_n^{(0,d)}(t)>\varepsilon\bigr)
&\le
\P\bigl(I_{n,m,q}^{(0,d)}(t)>\varepsilon\bigr)
+
\P(C_t^*>m)
+
\P([X]^{(p)}_t\ge q).
\end{aligned}
\]
Taking the upper limit as \(n\to\infty\) in the preceding
inequality yields, for every \(m,q\in\mathbb N\) and
\(\varepsilon>0\),
\[
\limsup_{n\to\infty}
\P\bigl(I_n^{(0,d)}(t)>\varepsilon\bigr)
\le
\P(C_t^*>m)
+
\P([X]^{(p)}_t\ge q).
\]
Since \([X]^{(p)}_t<\infty\) almost surely,
\[
        \P([X]^{(p)}_t\ge q)\longrightarrow0
        \qquad\text{as }q\to\infty.
\]
Since \(C^{R_0}\) is locally bounded,
\[
        \P(C_t^*>m)\longrightarrow0
        \qquad\text{as }m\to\infty.
\]
Consequently,
\[
        I_n^{(0,d)}(t)
        \longrightarrow0
        \qquad\text{in probability}.
\]

\medskip
\noindent
\textit{Step 1b: The contribution from \(\{d<|x|\le a\}\).}
\medskip

For \(s\le t\), set
\[
u_n^{d,a}(s)
:=
\int_{\{d<|x|\le a\}}
(f_n(s,x)-f(s,x))\,\nu(\{s\},dx).
\]
Since
\[
        \nu(\{s\},\mathbb R)\le1,
\]
we have
\[
|u_n^{d,a}(s)|
\le
\sup_{|x|\le a}|f_n(s,x)-f(s,x)|.
\]
Thus \({\rm (LU)}\) gives, outside a fixed null set,
$
        u_n^{d,a}(s)\longrightarrow0
$
simultaneously for every \(s\le t\).

By the notation in \((\mathrm{BE}_p)\),
\[
\int_{\{d<|x|\le a\}}
f_n(s,x)\,\nu(\{s\},dx)
=
\widehat\beta_n^{(d)}(s)
-
\widehat\beta_n^{(a)}(s)
\]
and
\[
\int_{\{d<|x|\le a\}}
f(s,x)\,\nu(\{s\},dx)
=
\beta^{(d)}(s)
-
\beta^{(a)}(s).
\]
Hence
\[
\begin{aligned}
|u_n^{d,a}(s)|^p
&\le
4^{p-1}
\Bigl(
|\widehat\beta_n^{(d)}(s)|^p
+
|\widehat\beta_n^{(a)}(s)|^p
+
|\beta^{(d)}(s)|^p
+
|\beta^{(a)}(s)|^p
\Bigr).
\end{aligned}
\]
Since \(a\ge d\), condition \((\mathrm{BE}_p)\), applied with
\(c=d\), gives
\[
\begin{aligned}
|\widehat\beta_n^{(d)}(s)|^p
&\le K_s^{(d)}\alpha^{(d)}(s),\\
|\widehat\beta_n^{(a)}(s)|^p
&\le K_s^{(d)}\alpha^{(a)}(s)
\le K_s^{(d)}\alpha^{(d)}(s),\\
|\beta^{(d)}(s)|^p
&\le K_s^{(d)}\alpha^{(d)}(s),\\
|\beta^{(a)}(s)|^p
&\le K_s^{(d)}\alpha^{(a)}(s)
\le K_s^{(d)}\alpha^{(d)}(s).
\end{aligned}
\]
Consequently,
\[
        |u_n^{d,a}(s)|^p
        \le
        4^pK_s^{(d)}\alpha^{(d)}(s).
\]

We record that
$
        \sum_{s\le t}\alpha^{(d)}(s)<\infty
        \text{ a.s..}
$
Indeed, define
\[
N_u^{(d)}
:=
\int_0^u\int_{\{|x|>d\}}\mu(ds,dx)
\]
and
\[
\rho_q^{(d)}
:=
\inf\{u>0:N_u^{(d)}\ge q\}.
\]
Since \(X\) is càdlàg,
$
        N_t^{(d)}<\infty
$
a.s., and hence
$
        \P(\rho_q^{(d)}>t)\longrightarrow1.
$
Set
\[
        \gamma_q^{(d)}(s)
        :=
        \bigl(1_{\{\cdot<\rho_q^{(d)}\}}\bigr)_{s-}.
\]
Then \(\gamma_q^{(d)}\) is predictable and
$
        \gamma_q^{(d)}(s)
        =
        1_{\{s\le\rho_q^{(d)}\}}.
$
The compensator identity gives
\[
\begin{aligned}
&\E\left[
\int_0^t
\gamma_q^{(d)}(s)
\int_{\{|x|>d\}}\nu(ds,dx)
\right]
=
\E\left[
\sum_{s\le t}
\gamma_q^{(d)}(s)
1_{\{|\Delta X_s|>d\}}
\right]
\le q.
\end{aligned}
\]
It follows that
\[
\int_0^t
\gamma_q^{(d)}(s)
\int_{\{|x|>d\}}\nu(ds,dx)
<\infty
\qquad\text{a.s.}
\]
for every \(q\). Since \(\rho_q^{(d)}>t\) eventually almost surely,
\[
\int_0^t\int_{\{|x|>d\}}\nu(ds,dx)<\infty
\qquad\text{a.s.}
\]
Therefore
\[
\sum_{s\le t}\alpha^{(d)}(s)
\le
\int_0^t\int_{\{|x|>d\}}\nu(ds,dx)
<\infty
\qquad\text{a.s.}
\]

For \(m\in\mathbb N\), define
\[
\pi_m^{(d)}
:=
\inf\{u>0:K_u^{(d)}>m\}.
\]
On \(\{s<\pi_m^{(d)}\}\),
$
        |u_n^{d,a}(s)|^p
        \le
        4^pm\alpha^{(d)}(s).
$
The right-hand side is summable over \(s\le t\), almost surely, and
does not depend on \(n\). Dominated convergence for series therefore
gives
$
\sum_{s\le t}
1_{\{s<\pi_m^{(d)}\}}
|u_n^{d,a}(s)|^p
\longrightarrow0
\text{ a.s..}
$
Set
\[
I_{n,m}^{(d,a)}(t)
:=
\left(
\sum_{s\le t}
1_{\{s<\pi_m^{(d)}\}}
|u_n^{d,a}(s)|^p
\right)^{1/p}.
\]
Then
$
        I_{n,m}^{(d,a)}(t)\longrightarrow0
        \text{ a.s..}
$
On \(\{\pi_m^{(d)}>t\}\),
$
        I_n^{(d,a)}(t)=I_{n,m}^{(d,a)}(t).
$
Thus
\[
\P\bigl(I_n^{(d,a)}(t)>\varepsilon\bigr)
\le
\P\bigl(I_{n,m}^{(d,a)}(t)>\varepsilon\bigr)
+
\P(\pi_m^{(d)}\le t).
\]
For every fixed \(m\in\mathbb N\), \(a>d\), and
\(\varepsilon>0\), taking the upper limit as \(n\to\infty\)
in the preceding inequality gives
\[
\limsup_{n\to\infty}
\P\bigl(I_n^{(d,a)}(t)>\varepsilon\bigr)
\le
\P(\pi_m^{(d)}\le t).
\]
Since \(K^{(d)}\) is locally bounded,
\[
        \P(\pi_m^{(d)}\le t)\longrightarrow0
        \qquad\text{as }m\to\infty.
\]
Therefore, for every fixed \(a>d\),
\[
        I_n^{(d,a)}(t)
        \longrightarrow0
        \qquad\text{in probability}.
\]

\medskip
\noindent
\textit{Step 1c: The contribution from \(\{|x|>a\}\).}
\medskip

For \(s\le t\), set
\[
v_n^a(s)
:=
\int_{\{|x|>a\}}
(f_n(s,x)-f(s,x))\,\nu(\{s\},dx).
\]
Then
\[
        v_n^a(s)
        =
        \widehat\beta_n^{(a)}(s)-\beta^{(a)}(s).
\]
Using
$
|u-v|^p
\le
2^{p-1}\bigl(|u|^p+|v|^p\bigr)
$
and \((\mathrm{BE}_p)\), applied with \(c=d\), we obtain
\[
\begin{aligned}
|v_n^a(s)|^p
&\le
2^{p-1}
\left(
|\widehat\beta_n^{(a)}(s)|^p
+
|\beta^{(a)}(s)|^p
\right)
\le
2^pK_s^{(d)}\alpha^{(a)}(s).
\end{aligned}
\]
On \(\{s<\pi_m^{(d)}\}\),
$
        |v_n^a(s)|^p
        \le
        2^pm\alpha^{(a)}(s).
$
For every fixed \(s\),
\[
        \alpha^{(a)}(s)
        =
        \nu(\{s\},\{|x|>a\})
        \longrightarrow0
\]
as \(a\to\infty\). Moreover,
\[
        0\le\alpha^{(a)}(s)\le\alpha^{(d)}(s),
        \qquad a\ge d,
\]
and
$
        \sum_{s\le t}\alpha^{(d)}(s)<\infty
        \text{ a.s..}
$
Dominated convergence for series therefore gives
\[
        \sum_{s\le t}\alpha^{(a)}(s)
        \longrightarrow0
        \qquad\text{a.s. as }a\to\infty.
\]
Consequently,
\[
\sum_{s\le t}
1_{\{s<\pi_m^{(d)}\}}
|v_n^a(s)|^p
\le
2^pm\sum_{s\le t}\alpha^{(a)}(s)
\longrightarrow0
\]
almost surely as \(a\to\infty\), uniformly in \(n\). Hence
\[
\begin{aligned}
\P\bigl(I_n^{(a,\infty)}(t)>\varepsilon\bigr)
&\le
\P\left(
2^pm\sum_{s\le t}\alpha^{(a)}(s)>\varepsilon^p
\right)
+
\P(\pi_m^{(d)}\le t).
\end{aligned}
\]
For every fixed \(m\in\mathbb N\), the almost sure convergence above
implies
\[
\limsup_{a\to\infty}
\sup_{n\ge1}
\P\bigl(I_n^{(a,\infty)}(t)>\varepsilon\bigr)
\le
\P(\pi_m^{(d)}\le t).
\]
Letting \(m\to\infty\) and using the local boundedness of
\(K^{(d)}\), we obtain
\[
\lim_{a\to\infty}
\sup_{n\ge1}
\P\bigl(I_n^{(a,\infty)}(t)>\varepsilon\bigr)
=
0.
\]

For fixed \(a>d\), the terms \(I_n^{(0,d)}(t)\) and
\(I_n^{(d,a)}(t)\) converge to zero in probability. Therefore
\[
\begin{aligned}
\limsup_{n\to\infty}\P(I_n(t)>\varepsilon)
&\le
\limsup_{n\to\infty}
\P\left(
I_n^{(a,\infty)}(t)>\frac{\varepsilon}{3}
\right).
\end{aligned}
\]
Letting \(a\to\infty\) yields
\[
        I_n(t)\longrightarrow0
        \qquad\text{in probability}.
\]

\medskip
\noindent
\textbf{Step 2: The local part of the \(\nu-\nu_n\) integral.}
\medskip

Let
\[
        C^\circ:=C^{R_0}.
\]
Then \(C^\circ\) is locally bounded and
\[
        |f_n(s,x)|+|f(s,x)|
        \le C^\circ_s|x|,
        \qquad |x|\le R_0.
\]

For \(0<r<\min\{1,R_0/2\}\) and \(k\ge1\), write
\[
        L:=L(r,k).
\]
By the \(\ell^p\)-triangle inequality,
\begin{align}
\left(
\sum_{s\le t}
\left|
\int_{\mathbb R}
f_n(s,x)(\nu-\nu_n)(\{s\},dx)
\right|^p
\right)^{1/p}
&\le
\left(
\sum_{s\le t}
\left|
\int_{|x|\le L}
f_n(s,x)\,\nu(\{s\},dx)
\right|^p
\right)^{1/p}
\nonumber\\
&+
\left(
\sum_{s\le t}
\left|
\int_{|x|\le L}
f_n(s,x)\,\nu_n(\{s\},dx)
\right|^p
\right)^{1/p}
\nonumber\\
&+
\left(
\sum_{s\le t}
\left|
\int_{|x|>L}
f_n(s,x)(\nu-\nu_n)(\{s\},dx)
\right|^p
\right)^{1/p}.
\label{eq:pvar-small-large-split}
\end{align}

Since
$
        \nu_n(\{s\},\mathbb R)\le1,
$
Jensen's inequality gives
\[
\left|
\int_{|x|\le L}
f_n(s,x)\,\nu_n(\{s\},dx)
\right|^p
\le
\int_{|x|\le L}
|f_n(s,x)|^p\,\nu_n(\{s\},dx).
\]
Therefore
\begin{align}
&\left(
\sum_{s\le t}
\left|
\int_{|x|\le L}
f_n(s,x)\,\nu_n(\{s\},dx)
\right|^p
\right)^{1/p}
\le
\left(
\int_0^t
\int_{|x|\le L}
|f_n(s,x)|^p\,\nu_n(ds,dx)
\right)^{1/p}.
\label{eq:pvar-small-compensator-Jensen}
\end{align}

We first record the corresponding pathwise estimate for the jump
measure. On the localization on which \(C^\circ\le m\),
\[
\left(
\int_0^t
\int_{|x|\le L}
|f_n(s,x)|^p\,\mu_n(ds,dx)
\right)^{1/p}
\le
m
\left(
\sum_{s\le t}
|\Delta X_s^n|^p
1_{\{|\Delta X_s^n|\le L\}}
\right)^{1/p}.
\]
By the \(p\)-th power Threshold Isolation Principle, on
\(A_k(r)\cap B_n^{(p)}(k,r)\),
\begin{align}
&\sum_{s\le t}
|\Delta X_s^n|^p
1_{\{|\Delta X_s^n|\le L\}}
\le
2^{p-1}[X^n-X]^{(p)}_t
+
2^{p-1}
\sum_{s\le t}
|\Delta X_s|^p
1_{\{|\Delta X_s|\le L\}}.
\label{eq:pvar-threshold-small-jump-bound}
\end{align}
Since \(L<r\),
\[
\sum_{s\le t}
|\Delta X_s|^p
1_{\{|\Delta X_s|\le L\}}
\le
\sum_{s\le t}
|\Delta X_s|^p
1_{\{|\Delta X_s|\le r\}},
\]
and the right-hand side converges almost surely to zero as
\(r\downarrow0\).

We now transfer this estimate from \(\mu_n\) to \(\nu_n\). For
\(m\in\mathbb N\), set
\[
        \chi_m^\circ(s):=1_{\{C^\circ_s\le m\}},
        \qquad
        C_t^{\circ,*}:=\sup_{0\le u\le t}C^\circ_u.
\]
Define
\[
\rho_k
:=
\inf\left\{
s>0:
|\Delta X_s|
\in[L-\delta(k,r),L+\delta(k,r)]
\right\},
\]
\[
\eta_{n,k}
:=
\inf\left\{
s>0:
|\Delta(X^n-X)_s|\ge\delta(k,r)
\right\},
\]
and
\[
\sigma_q
:=
\inf\{s>0:[X]^{(p)}_s\ge q\},
\qquad
\sigma_{q,n}
:=
\inf\{s>0:[X^n]^{(p)}_s\ge q\}.
\]
Set
\[
\tau_{q,n,k}
:=
\rho_k\wedge\eta_{n,k}\wedge\sigma_q\wedge\sigma_{q,n}
\]
and
\[
\Gamma_{q,n,k}(s)
:=
\bigl(1_{\{\cdot<\tau_{q,n,k}\}}\bigr)_{s-}
=
1_{\{s\le\tau_{q,n,k}\}}.
\]
Both \(\chi_m^\circ\) and \(\Gamma_{q,n,k}\) are predictable.

Define
\[
Z_n^{m,q,r,k}
:=
\int_0^t
\chi_m^\circ(s)\Gamma_{q,n,k}(s)
\int_{|x|\le L}
|f_n(s,x)|^p\,\nu_n(ds,dx).
\]
By the compensator identity,
\begin{align}
\E[Z_n^{m,q,r,k}]
&\le
m^p
\E\left[
\sum_{s\le t}
\Gamma_{q,n,k}(s)
|\Delta X_s^n|^p
1_{\{|\Delta X_s^n|\le L\}}
\right].
\label{eq:pvar-stopped-small-first}
\end{align}

The predictable localization includes at most the single boundary time
\(\tau_{q,n,k}\). Hence
\[
\begin{aligned}
&\sum_{s\le t}
\Gamma_{q,n,k}(s)
|\Delta X_s^n|^p
1_{\{|\Delta X_s^n|\le L\}}
\le
\sum_{s\le t}
1_{\{s<\tau_{q,n,k}\}}
|\Delta X_s^n|^p
1_{\{|\Delta X_s^n|\le L\}}
+
L^p.
\end{aligned}
\]
For \(s<\tau_{q,n,k}\), the threshold classification is preserved,
so
\[
1_{\{|\Delta X_s^n|\le L\}}
\le
1_{\{|\Delta X_s|\le L\}}.
\]

Set
\[
Q_n^{q,r,k}
:=
\sum_{s\le t}
1_{\{s<\tau_{q,n,k}\}}
|\Delta(X^n-X)_s|^p
\]
and
\[
S^{q,r}
:=
\sum_{s\le t}
1_{\{s<\sigma_q\}}
|\Delta X_s|^p
1_{\{|\Delta X_s|\le r\}}.
\]
Using
$
|u+v|^p
\le
2^{p-1}\bigl(|u|^p+|v|^p\bigr)
$
and \(L<r\), we obtain
\begin{align}
\E[Z_n^{m,q,r,k}]
&\le
2^{p-1}m^p\E[Q_n^{q,r,k}]
+
2^{p-1}m^p\E[S^{q,r}]
+
m^pL^p.
\label{eq:pvar-stopped-small-estimate}
\end{align}

We have
\[
        Q_n^{q,r,k}
        \le
        [X^n-X]^{(p)}_t
        \longrightarrow0
        \qquad\text{in probability}.
\]
Moreover, before \(\tau_{q,n,k}\),
\[
[X]^{(p)}_{\tau_{q,n,k}-}<q,
\qquad
[X^n]^{(p)}_{\tau_{q,n,k}-}<q.
\]
Therefore
\[
\begin{aligned}
Q_n^{q,r,k}
&\le
2^{p-1}
\left(
\sum_{s<\tau_{q,n,k}}|\Delta X_s^n|^p
+
\sum_{s<\tau_{q,n,k}}|\Delta X_s|^p
\right)
\le
2^pq.
\end{aligned}
\]
It follows by bounded convergence in probability that
$
        \E[Q_n^{q,r,k}]\longrightarrow0.
$
Furthermore,
\[
        S^{q,r}\longrightarrow0
        \qquad\text{a.s. as }r\downarrow0,
        \qquad
        0\le S^{q,r}\le q,
\]
and hence
$
        \E[S^{q,r}]\longrightarrow0.
$
Finally,
$
        L^p\le r^p\longrightarrow0.
$
On the event
$
        \{C_t^{\circ,*}\le m\}
        \cap
        \{\tau_{q,n,k}>t\},
$
we have
\[
Z_n^{m,q,r,k}
=
\int_0^t
\int_{|x|\le L}
|f_n(s,x)|^p\,\nu_n(ds,dx).
\]
Consequently,
\begin{align}
&\P\left(
\int_0^t
\int_{|x|\le L}
|f_n(s,x)|^p\,\nu_n(ds,dx)
\ge\varepsilon
\right)
\le
\P(Z_n^{m,q,r,k}\ge\varepsilon)
+
\P(C_t^{\circ,*}>m)
+
\P(\tau_{q,n,k}\le t).
\label{eq:pvar-remove-small-localization}
\end{align}
Furthermore,
\begin{align}
\P(\tau_{q,n,k}\le t)
&\le
\P(A_k(r)^c)
+
\P(B_n^{(p)}(k,r)^c)
+
\P([X]^{(p)}_t\ge q)
+
\P([X^n]^{(p)}_t\ge q).
\label{eq:pvar-small-exit-bound}
\end{align}

For every \(\varepsilon>0\), the atomwise Jensen estimate,
Markov's inequality, \eqref{eq:pvar-stopped-small-estimate},
and \eqref{eq:pvar-small-exit-bound} give
\begin{align}
\P\left(
\left(
\sum_{s\le t}
\left|
\int_{|x|\le L}
f_n(s,x)\,\nu_n(\{s\},dx)
\right|^p
\right)^{1/p}
>
\varepsilon
\right)
&\le
\frac{
2^{p-1}m^p\E[Q_n^{q,r,k}]
+
2^{p-1}m^p\E[S^{q,r}]
+
m^pr^p
}{
\varepsilon^p
}
\nonumber\\
&+
\P(C_t^{\circ,*}>m)
+
\P(A_k(r)^c)
+
\P(B_n^{(p)}(k,r)^c)
\nonumber\\
&+
\P([X]^{(p)}_t\ge q)
+
\P([X^n]^{(p)}_t\ge q).
\label{eq:pvar-small-nun-probability}
\end{align}

For every fixed \(m,q,r,k\),
\[
        \E[Q_n^{q,r,k}]\longrightarrow0
        \qquad\text{and}\qquad
        \P(B_n^{(p)}(k,r)^c)\longrightarrow0
\]
as \(n\to\infty\). For every fixed \(m,q\),
\[
        \E[S^{q,r}]\longrightarrow0,
        \qquad
        r^p\longrightarrow0
\]
as \(r\downarrow0\), while, for every fixed \(r>0\),
\[
        \P(A_k(r)^c)\longrightarrow0
        \qquad\text{as }k\to\infty.
\]
Finally,
\[
[X^n]^{(p)}_t
\le
2^{p-1}
\left(
[X^n-X]^{(p)}_t+[X]^{(p)}_t
\right),
\]
so that
$
        \{[X^n]^{(p)}_t:n\ge1\}
$
is bounded in probability.

The term involving \(\nu\) is handled similarly but does not require
the \(n\)-dependent threshold localization. Set
\[
        \Gamma_q(s)
        :=
        \bigl(1_{\{\cdot<\sigma_q\}}\bigr)_{s-}
        =
        1_{\{s\le\sigma_q\}}
\]
and
\[
Z_{n,\nu}^{m,q,r,k}
:=
\int_0^t
\chi_m^\circ(s)\Gamma_q(s)
\int_{|x|\le L}
|f_n(s,x)|^p\,\nu(ds,dx).
\]
The compensator identity gives
\[
\begin{aligned}
\E[Z_{n,\nu}^{m,q,r,k}]
&\le
m^p
\E\left[
\sum_{s\le t}
\Gamma_q(s)
|\Delta X_s|^p
1_{\{|\Delta X_s|\le L\}}
\right]
\le
m^p\left(
\E[S^{q,r}]+r^p
\right).
\end{aligned}
\]
The term \(r^p\) accounts for the possible single boundary jump at
\(\sigma_q\). The right-hand side converges to zero as
\(r\downarrow0\).

For every \(\varepsilon>0\), the atomwise Jensen estimate,
Markov's inequality, and the preceding expectation bound give
\begin{align}
&\P\left(
\left(
\sum_{s\le t}
\left|
\int_{|x|\le L}
f_n(s,x)\,\nu(\{s\},dx)
\right|^p
\right)^{1/p}
>
\varepsilon
\right)
\le
\frac{
m^p\bigl(\E[S^{q,r}]+r^p\bigr)
}{
\varepsilon^p
}
+
\P(C_t^{\circ,*}>m)
+
\P([X]^{(p)}_t\ge q).
\label{eq:pvar-small-nu-probability}
\end{align}

\medskip
\noindent
\textbf{Step 3: The part of the \(\nu-\nu_n\) integral above \(L\).}
\medskip

For each \(\omega\), the set of nonzero jump magnitudes
\[
\mathcal J_t(\omega)
:=
\{|\Delta X_s(\omega)|:0<s\le t,\ \Delta X_s(\omega)\ne0\}
\]
is countable. Hence, by Fubini's theorem,
\[
\int_d^\infty
\P\left(
\exists s\le t:|\Delta X_s|=a
\right)\,da
=
0.
\]
Consequently, for Lebesgue-almost every \(a>d\),
\[
\P\left(
\exists s\le t:|\Delta X_s|=a
\right)=0.
\]
We let \(a\to\infty\) through such regular radii. Fix henceforth a
regular radius \(a>d\).

Define
\[
\sigma_k
:=
\inf\left\{
s\in[0,t]:
\bigl||\Delta X_s|-L\bigr|\le\delta(k,r)
\right\},
\qquad
\inf\varnothing:=\infty.
\]
Then
$
        \{\sigma_k\le t\}=A_k(r)^c
$
and hence, by the Threshold Isolation Principle,
\[
        \P(\sigma_k\le t)
        =
        \P(A_k(r)^c)
        \longrightarrow0
        \qquad\text{as }k\to\infty.
\]

Define also
\[
\iota_a
:=
\inf\left\{
u\in(0,t]:
|\Delta X_u|\ge a
\right\},
\qquad
\inf\varnothing:=\infty.
\]

Let
\[
\begin{aligned}
\mathcal D_t
:={}&
\left\{
(s,\omega):
0<s\le t,\
\nu(\omega,\{s\},\mathbb R)>0
\right\}
\cup
\bigcup_{n\ge1}
\left\{
(s,\omega):
0<s\le t,\
\nu_n(\omega,\{s\},\mathbb R)>0
\right\}.
\end{aligned}
\]
This is a thin predictable set. Choose a predictable exhaustion
$
        \{T_h\}_{h\ge1}
$
of \(\mathcal D_t\) whose graphs are pairwise disjoint. All expressions
evaluated at \(T_h\) are understood to vanish on \(\{T_h=\infty\}\).

Fix \(h,j\ge1\), and set
$
        E_j(T_h):=\{C^a_{T_h}\le j\}.
$
Since \(C^a\) and \(T_h\) are predictable,
$
        E_j(T_h)\in\mathcal F_{T_h-}.
$
Consequently,
\[
\left|
1_{E_j(T_h)}
\int_{\{L<|x|\le a\}}
f_n(T_h,x)(\mu-\mu_n)(\{T_h\},dx)
\right|
\le 2ja.
\]
Hence
\[
1_{E_j(T_h)}
\int_{\{L<|x|\le a\}}
f_n(T_h,x)(\mu-\mu_n)(\{T_h\},dx)
\in
L^1(\Omega,\mathcal F,\P)
\cap
L^p(\Omega,\mathcal F,\P).
\]

Applying the compensator identity at \(T_h\) separately to
\((\mu,\nu)\) and \((\mu_n,\nu_n)\) gives
\begin{align}
&1_{E_j(T_h)}
\int_{\{L<|x|\le a\}}
f_n(T_h,x)(\nu-\nu_n)(\{T_h\},dx)
=
\E\left[
1_{E_j(T_h)}
\int_{\{L<|x|\le a\}}
f_n(T_h,x)(\mu-\mu_n)(\{T_h\},dx)
\Bigm|\mathcal F_{T_h-}
\right].
\label{eq:pvar-localized-annulus-identity}
\end{align}
Conditional Jensen's inequality yields
\begin{align}
&1_{E_j(T_h)}
\left|
\int_{\{L<|x|\le a\}}
f_n(T_h,x)(\nu-\nu_n)(\{T_h\},dx)
\right|^p
\le
\E\left[
1_{E_j(T_h)}
\left|
\int_{\{L<|x|\le a\}}
f_n(T_h,x)(\mu-\mu_n)(\{T_h\},dx)
\right|^p
\Bigm|\mathcal F_{T_h-}
\right].
\label{eq:pvar-localized-annulus-Jensen}
\end{align}

We next establish the stopped atomwise convergence. Work on the common
full-measure event on which
\\$
        [X^n-X]^{(p)}_t\longrightarrow0,
$
\({\rm (LU)}\) holds simultaneously for all \(s\le t\),
\({\rm (JC)}\) holds simultaneously for all \(s\le t\), and no jump
of \(X\) has magnitude exactly \(a\).

On
$
        \{T_h<\sigma_k\wedge\iota_a\},
$
we have
\[
\bigl||\Delta X_{T_h}|-L\bigr|>\delta(k,r),
\qquad
|\Delta X_{T_h}|<a.
\]
Moreover,
$
        \Delta X^n_{T_h}\longrightarrow\Delta X_{T_h},
$
because
\[
\sup_{u\le t}|\Delta(X^n-X)_u|
\le
\bigl([X^n-X]^{(p)}_t\bigr)^{1/p}
\longrightarrow0.
\]

If
$
        |\Delta X_{T_h}|<L-\delta(k,r),
$
then
$
        |\Delta X^n_{T_h}|<L
$
for all sufficiently large \(n\). Consequently,
\[
\int_{\{L<|x|\le a\}}
f_n(T_h,x)(\mu-\mu_n)(\{T_h\},dx)
=
0
\]
eventually.

If instead
$
        |\Delta X_{T_h}|>L+\delta(k,r),
$
then
$
        L<|\Delta X_{T_h}|<a,
$
and, eventually,
$
        L<|\Delta X^n_{T_h}|<a.
$
Therefore
\[
\begin{aligned}
&\int_{\{L<|x|\le a\}}
f_n(T_h,x)(\mu-\mu_n)(\{T_h\},dx)
=
f_n(T_h,\Delta X_{T_h})
-
f_n(T_h,\Delta X^n_{T_h})
\end{aligned}
\]
eventually. Furthermore,
\[
\begin{aligned}
|f_n(T_h,\Delta X^n_{T_h})
-
f_n(T_h,\Delta X_{T_h})|
&\le
|f_n(T_h,\Delta X^n_{T_h})
-
f(T_h,\Delta X^n_{T_h})|
+
|f(T_h,\Delta X^n_{T_h})
-
f(T_h,\Delta X_{T_h})|
\\
&+
|f(T_h,\Delta X_{T_h})
-
f_n(T_h,\Delta X_{T_h})|.
\end{aligned}
\]
The first and third terms converge to zero by \({\rm (LU)}\), while
the middle term converges to zero by \({\rm (JC)}\). Thus
\begin{align}
&1_{\{T_h<\sigma_k\wedge\iota_a\}}
\int_{\{L<|x|\le a\}}
f_n(T_h,x)(\mu-\mu_n)(\{T_h\},dx)
\longrightarrow0
\qquad\text{a.s.}
\label{eq:pvar-stopped-atomwise-annulus}
\end{align}

On the boundary
$
        \{T_h=\iota_a<\sigma_k\},
$
regularity of \(a\) gives
$
        |\Delta X_{T_h}|>a$ a.s..
Since
$
        \Delta X^n_{T_h}\longrightarrow\Delta X_{T_h},
$
we also have
$
        |\Delta X^n_{T_h}|>a
$
for all sufficiently large \(n\). Consequently, on
$
        \{T_h=\iota_a<\sigma_k\},
$
we have
\[
\int_{\{L<|x|\le a\}}
f_n(T_h,x)\,\mu(\{T_h\},dx)
=
0
\]
for every \(n\), and
\[
\int_{\{L<|x|\le a\}}
f_n(T_h,x)\,\mu_n(\{T_h\},dx)
=
0
\]
for all sufficiently large \(n\).
 Hence
\begin{align}
&1_{\{T_h=\iota_a<\sigma_k\}}
\int_{\{L<|x|\le a\}}
f_n(T_h,x)(\mu-\mu_n)(\{T_h\},dx)
\longrightarrow0
\qquad\text{a.s.}
\label{eq:pvar-iota-boundary}
\end{align}

It remains to control \(T_h=\sigma_k\). Since
\[
        |\Delta X_{\sigma_k}|
        \le L+\delta(k,r)<r
\]
and
$
        \sup_{u\le t}|\Delta(X^n-X)_u|
        \longrightarrow0,
$
we have, eventually,
\[
        |\Delta X^n_{\sigma_k}|<2r\le R_0.
\]
Consequently, on
$
        \{C^\circ_{\sigma_k}\le m\},
$
condition \({\rm (LG}_0{\rm )}\) gives
\[
\begin{aligned}
\left|
\int_{\{L<|x|\le a\}}
f_n(\sigma_k,x)(\mu-\mu_n)(\{\sigma_k\},dx)
\right|
&\le
|f_n(\sigma_k,\Delta X_{\sigma_k})|
+
|f_n(\sigma_k,\Delta X^n_{\sigma_k})|
\\
&\le
m|\Delta X_{\sigma_k}|
+
m|\Delta X^n_{\sigma_k}|
\le 3mr
\end{aligned}
\]
eventually.

Combining the preceding three cases gives
\begin{align}
&\limsup_{n\to\infty}
1_{\{C^\circ_{T_h}\le m\}}
\bigl(1_{\{\cdot<\sigma_k\wedge\iota_a\}}\bigr)_{T_h-}
\left|
\int_{\{L<|x|\le a\}}
f_n(T_h,x)(\mu-\mu_n)(\{T_h\},dx)
\right|
\le
3mr\,1_{\{T_h=\sigma_k\le t\}}
\qquad\text{a.s.}
\label{eq:pvar-pathwise-boundary}
\end{align}

The variables
\[
1_{E_j(T_h)},\qquad
1_{\{C^\circ_{T_h}\le m\}},\qquad
\bigl(1_{\{\cdot<\sigma_k\wedge\iota_a\}}\bigr)_{T_h-}
\]
are \(\mathcal F_{T_h-}\)-measurable. Multiplying
\eqref{eq:pvar-localized-annulus-identity} by the final two variables,
using
\[
|\E[Z\mid\mathcal F_{T_h-}]|
\le
\E[|Z|\mid\mathcal F_{T_h-}],
\]
we first obtain
\[
\begin{aligned}
&1_{E_j(T_h)}
\limsup_{n\to\infty}
1_{\{C^\circ_{T_h}\le m\}}
\left|
\bigl(1_{\{\cdot<\sigma_k\wedge\iota_a\}}\bigr)_{T_h-}
\int_{\{L<|x|\le a\}}
f_n(T_h,x)(\nu-\nu_n)(\{T_h\},dx)
\right|
\\
&\quad\le
\limsup_{n\to\infty}
\E\left[
1_{E_j(T_h)}
1_{\{C^\circ_{T_h}\le m\}}
\bigl(1_{\{\cdot<\sigma_k\wedge\iota_a\}}\bigr)_{T_h-}
\left|
\int_{\{L<|x|\le a\}}
f_n(T_h,x)(\mu-\mu_n)(\{T_h\},dx)
\right|
\Bigm|\mathcal F_{T_h-}
\right].
\end{aligned}
\]
The random variables inside the conditional expectations are
nonnegative and bounded by \(2ja\). The conditional reverse Fatou's lemma
therefore gives
\[
\begin{aligned}
&1_{E_j(T_h)}
\limsup_{n\to\infty}
1_{\{C^\circ_{T_h}\le m\}}
\left|
\bigl(1_{\{\cdot<\sigma_k\wedge\iota_a\}}\bigr)_{T_h-}
\int_{\{L<|x|\le a\}}
f_n(T_h,x)(\nu-\nu_n)(\{T_h\},dx)
\right|
\le
3mr\,1_{E_j(T_h)}
\E\left[
1_{\{T_h=\sigma_k\le t\}}
\Bigm|\mathcal F_{T_h-}
\right].
\end{aligned}
\]
Since \(C^a\) is finite-valued,
$
        E_j(T_h)\uparrow\Omega
$ a.s.. Letting \(j\to\infty\) and then raising both sides to the power \(p\)
gives
\begin{align}
&\limsup_{n\to\infty}
1_{\{C^\circ_{T_h}\le m\}}
\left|
\bigl(1_{\{\cdot<\sigma_k\wedge\iota_a\}}\bigr)_{T_h-}
\int_{\{L<|x|\le a\}}
f_n(T_h,x)(\nu-\nu_n)(\{T_h\},dx)
\right|^p
\le
3^pm^pr^p
\left(
\E\left[
1_{\{T_h=\sigma_k\le t\}}
\Bigm|\mathcal F_{T_h-}
\right]
\right)^p.
\label{eq:pvar-conditional-boundary}
\end{align}

Since the graphs of the \(T_h\)'s are pairwise disjoint and
\(z^p\le z\) for \(0\le z\le1\),
\begin{align}
\E\left[
\sum_{h\ge1}
\left(
\E\left[
1_{\{T_h=\sigma_k\le t\}}
\Bigm|\mathcal F_{T_h-}
\right]
\right)^p
\right]
&\le
\sum_{h\ge1}
\P(T_h=\sigma_k\le t)
\nonumber\\
&\le
\P(\sigma_k\le t)
=
\P(A_k(r)^c).
\label{eq:pvar-boundary-probability-sum}
\end{align}
Apply Lemma \ref{complemma} and let
$
        \theta_{q,k,\ell,n}
$
be the corresponding stopping times. Define
\[
        \theta^a_{q,k,\ell,n}
        :=
        \theta_{q,k,\ell,n}\wedge\iota_a.
\]
Since
\(
\theta^a_{q,k,\ell,n}\le\theta_{q,k,\ell,n},
\)
part~(3) of Lemma~\ref{complemma} gives
\[
\sum_{s\le t}
\sup_{n\ge\ell}
\bigl(1_{\{\cdot<\theta^a_{q,k,\ell,n}\}}\bigr)_{s-}
\left(
\int_{|x|>L}\nu_n(\{s\},dx)
+
\int_{|x|>L}\nu(\{s\},dx)
\right)
<\infty
\]
almost surely.

Let
\[
c_{q,\ell,n}^{(m)}(s)
:=
\chi_m^\circ(s)
1_{\{s<\theta^a_{q,k,\ell,n}\}}
\left|
\int_{L<|x|\le a}
f_n(s,x)(\nu-\nu_n)(\{s\},dx)
\right|^p.
\]
On
\[
        \{C_t^{\circ,*}\le m\}
        \cap
        \{\theta^a_{q,k,\ell,n}>t\},
\]
we have
\[
\sum_{s\le t}c_{q,\ell,n}^{(m)}(s)
=
\sum_{s\le t}
\left|
\int_{L<|x|\le a}
f_n(s,x)(\nu-\nu_n)(\{s\},dx)
\right|^p.
\]

For brevity, write
\[
\gamma_n(s)
:=
\bigl(1_{\{\cdot<\theta^a_{q,k,\ell,n}\}}\bigr)_{s-}
\]
and define
\[
b(s)
:=
\sup_{n\ge\ell}
\gamma_n(s)
\left|
\int_{L<|x|\le a}
f_n(s,x)(\nu-\nu_n)(\{s\},dx)
\right|^p.
\]
By the notation introduced in \((\mathrm{BE}_p)\),
\begin{align*}
b(s)
&=
\sup_{n\ge\ell}
\gamma_n(s)
\Bigg|
\int_{\{|x|>L\}}
f_n(s,x)\,\nu(\{s\},dx)
-
\int_{\{|x|>a\}}
f_n(s,x)\,\nu(\{s\},dx)
\\
&\hspace{8em}
-
\int_{\{|x|>L\}}
f_n(s,x)\,\nu_n(\{s\},dx)
+
\int_{\{|x|>a\}}
f_n(s,x)\,\nu_n(\{s\},dx)
\Bigg|^p
\\
&\le
4^{p-1}
\sup_{n\ge\ell}
\gamma_n(s)
\Bigg[
\left|
\int_{\{|x|>L\}}
f_n(s,x)\,\nu(\{s\},dx)
\right|^p
+
\left|
\int_{\{|x|>a\}}
f_n(s,x)\,\nu(\{s\},dx)
\right|^p
\\
&\hspace{8em}
+
\left|
\int_{\{|x|>L\}}
f_n(s,x)\,\nu_n(\{s\},dx)
\right|^p
+
\left|
\int_{\{|x|>a\}}
f_n(s,x)\,\nu_n(\{s\},dx)
\right|^p
\Bigg]
\\
&\le
4^{p-1}K_s^{(L)}
\sup_{n\ge\ell}
\gamma_n(s)
\Bigl(
\alpha^{(L)}(s)
+
\alpha^{(a)}(s)
+
\alpha_n^{(L)}(s)
+
\alpha_n^{(a)}(s)
\Bigr)
\\
&\le
2\cdot4^{p-1}K_s^{(L)}
\sup_{n\ge\ell}
\gamma_n(s)
\left(
\alpha_n^{(L)}(s)
+
\alpha^{(L)}(s)
\right).
\end{align*}
Since \(K^{(L)}\) is locally bounded,
\[
        M_t^{(L)}
        :=
        \sup_{0\le u\le t}K_u^{(L)}
        <\infty
        \qquad\text{a.s.}
\]
Hence
\[
b(s)
\le
2\cdot4^{p-1}M_t^{(L)}
\sup_{n\ge\ell}\gamma_n(s)
\left(
\alpha_n^{(L)}(s)+\alpha^{(L)}(s)
\right).
\]
Part~(3) of Lemma~\ref{complemma} therefore yields
\[
        \sum_{s\le t}b(s)<\infty
        \qquad\text{a.s.}
\]

Since
$
        \theta^a_{q,k,\ell,n}
        \le
        \sigma_k\wedge\iota_a,
$
the atomwise estimate \eqref{eq:pvar-conditional-boundary} gives,
for every \(h\ge1\),
\[
\begin{aligned}
\limsup_{n\to\infty}
c_{q,\ell,n}^{(m)}(T_h)
&\le
3^pm^pr^p
\left(
\E\left[
1_{\{T_h=\sigma_k\le t\}}
\Bigm|\mathcal F_{T_h-}
\right]
\right)^p
\end{aligned}
\]
almost surely. Moreover,
$
        c_{q,\ell,n}^{(m)}(s)\le b(s).
$
The reverse Fatou inequality for series therefore gives
\[
\begin{aligned}
\limsup_{n\to\infty}
\sum_{s\le t}c_{q,\ell,n}^{(m)}(s)
&\le
3^pm^pr^p
\sum_{h\ge1}
\left(
\E\left[
1_{\{T_h=\sigma_k\le t\}}
\Bigm|\mathcal F_{T_h-}
\right]
\right)^p
\end{aligned}
\]
almost surely. Consequently, by
\eqref{eq:pvar-boundary-probability-sum},
\[
\limsup_{n\to\infty}
\P\left(
\sum_{s\le t}c_{q,\ell,n}^{(m)}(s)>\varepsilon
\right)
\le
\frac{3^pm^pr^p}{\varepsilon}
\P(A_k(r)^c).
\]

On the event
\[
        \{C_t^{\circ,*}\le m\}
        \cap
        \{\theta^a_{q,k,\ell,n}>t\},
\]
the stopped sum equals the unrestricted bounded-annulus sum. Hence
\begin{align}
\P\left(
\sum_{s\le t}
\left|
\int_{L<|x|\le a}
f_n(s,x)(\nu-\nu_n)(\{s\},dx)
\right|^p
>\varepsilon
\right)
&\le
\P(C_t^{\circ,*}>m)
+
\P(\theta_{q,k,\ell,n}\le t)\nonumber
\\
&+
\P(\iota_a\le t)
+
\P\left(
\sum_{s\le t}c_{q,\ell,n}^{(m)}(s)>\varepsilon
\right).
\label{eq:pvar-bounded-annulus-localization}
\end{align}
Taking upper limits gives
\begin{align}
\limsup_{n\to\infty}
\P\left(
\sum_{s\le t}
\left|
\int_{L<|x|\le a}
f_n(s,x)(\nu-\nu_n)(\{s\},dx)
\right|^p
>\varepsilon
\right)
&\le
\P(C_t^{\circ,*}>m)
+
\sup_{n\ge\ell}
\P(\theta_{q,k,\ell,n}\le t)
\nonumber\\
&+
\P(\iota_a\le t)+
\frac{3^pm^pr^p}{\varepsilon}
\P(A_k(r)^c).
\label{eq:pvar-bounded-annulus-limsup}
\end{align}

It remains to control the compensator tail \(|x|>a\). At every
\(s\le t\),
\[
\int_{\{|x|>a\}}
f_n(s,x)(\nu-\nu_n)(\{s\},dx)
=
\widehat\beta_n^{(a)}(s)-\beta_n^{(a)}(s).
\]

For \(M,q\in\mathbb N\), set
\[
        K_t^{(d),*}
        :=
        \sup_{0\le u\le t}K_u^{(d)},
        \qquad
        \chi_M^{(d)}(s)
        :=
        1_{\{K_s^{(d)}\le M\}}.
\]
Define
\[
\rho_q
:=
\inf\{u>0:[X]^{(p)}_u\ge q\},
\qquad
\rho_{q,n}
:=
\inf\{u>0:[X^n]^{(p)}_u\ge q\},
\]
and
\[
\gamma_q(s)
:=
\bigl(1_{\{\cdot<\rho_q\}}\bigr)_{s-},
\qquad
\gamma_{q,n}(s)
:=
\bigl(1_{\{\cdot<\rho_{q,n}\}}\bigr)_{s-}.
\]
All three factors
\(\chi_M^{(d)},\gamma_q,\gamma_{q,n}\) are predictable.

Set
\[
\widetilde R_{M,q}^{n,a}
:=
\sum_{s\le t}
\chi_M^{(d)}(s)
\gamma_q(s)\gamma_{q,n}(s)
\left|
\widehat\beta_n^{(a)}(s)-\beta_n^{(a)}(s)
\right|^p.
\]
By \((\mathrm{BE}_p)\),
\[
\widetilde R_{M,q}^{n,a}
\le
2^{p-1}M
\sum_{s\le t}\gamma_q(s)\alpha^{(a)}(s)
+
2^{p-1}M
\sum_{s\le t}\gamma_{q,n}(s)\alpha_n^{(a)}(s).
\]

The compensator identity gives
\[
\begin{aligned}
\E\left[
\sum_{s\le t}\gamma_q(s)\alpha^{(a)}(s)
\right]
&\le
\E\left[
\sum_{s\le t}
\gamma_q(s)1_{\{|\Delta X_s|>a\}}
\right].
\end{aligned}
\]
The jumps strictly before \(\rho_q\) contribute at most \(q/a^p\),
while the predictable localization can include the single boundary
jump at \(\rho_q\). Therefore
\[
\sum_{s\le t}
\gamma_q(s)1_{\{|\Delta X_s|>a\}}
\le
\frac{q}{a^p}
+
1_{\{\rho_q\le t\}}.
\]
Hence
\[
\E\left[
\sum_{s\le t}\gamma_q(s)\alpha^{(a)}(s)
\right]
\le
\frac{q}{a^p}
+
\P(\rho_q\le t).
\]
Similarly,
\[
\E\left[
\sum_{s\le t}\gamma_{q,n}(s)\alpha_n^{(a)}(s)
\right]
\le
\frac{q}{a^p}
+
\P(\rho_{q,n}\le t).
\]
It follows that
\[
\begin{aligned}
\E[\widetilde R_{M,q}^{n,a}]
&\le
\frac{2^pMq}{a^p}
+
2^{p-1}M
\left(
\P(\rho_q\le t)
+
\P(\rho_{q,n}\le t)
\right).
\end{aligned}
\]

On the event
\[
        \{K_t^{(d),*}\le M\}
        \cap
        \{\rho_q>t\}
        \cap
        \{\rho_{q,n}>t\},
\]
we have
\[
\widetilde R_{M,q}^{n,a}
=
\sum_{s\le t}
\left|
\widehat\beta_n^{(a)}(s)-\beta_n^{(a)}(s)
\right|^p.
\]
Thus Markov's inequality gives
\[
\begin{aligned}
\P\left(
\sum_{s\le t}
\left|
\widehat\beta_n^{(a)}(s)-\beta_n^{(a)}(s)
\right|^p
>\varepsilon
\right)
&\le
\P(K_t^{(d),*}>M)
+
\P(\rho_q\le t)
+
\P(\rho_{q,n}\le t)
\\
&+
\frac{2^pMq}{\varepsilon a^p}
+
\frac{2^{p-1}M}{\varepsilon}
\left(
\P(\rho_q\le t)
+
\P(\rho_{q,n}\le t)
\right).
\end{aligned}
\]
Fix \(\varepsilon>0\) and \(\zeta>0\). Since \(K^{(d)}\) is locally
bounded, choose \(M\in\mathbb N\) such that
\[
        \P(K_t^{(d),*}>M)<\zeta.
\]
The family
$
        \{[X^n]^{(p)}_t:n\ge1\}
$
is bounded in probability. Hence \(q\in\mathbb N\) may be chosen so
large that
\[
\left(
1+\frac{2^{p-1}M}{\varepsilon}
\right)
\left(
\P([X]^{(p)}_t\ge q)
+
\sup_{n\ge1}\P([X^n]^{(p)}_t\ge q)
\right)
<
\zeta.
\]
For these fixed \(M\) and \(q\), choose \(a_0>d\) such that
\[
        \frac{2^pMq}{\varepsilon a^p}<\zeta
        \qquad\text{for every }a\ge a_0.
\]
The preceding probability estimate then gives, for every regular
\(a\ge a_0\),
\[
\limsup_{n\to\infty}
\P\left(
\sum_{s\le t}
\left|
\widehat\beta_n^{(a)}(s)-\beta_n^{(a)}(s)
\right|^p
>
\varepsilon
\right)
\le
3\zeta.
\]
Since \(\zeta>0\) was arbitrary,
\begin{equation}
\lim_{\substack{a\to\infty\\a\ {\rm regular}}}
\limsup_{n\to\infty}
\P\left(
\sum_{s\le t}
\left|
\int_{\{|x|>a\}}
f_n(s,x)(\nu-\nu_n)(\{s\},dx)
\right|^p
>
\varepsilon
\right)
=
0.
\label{eq:pvar-tail-limsup}
\end{equation}

We now combine the estimates. Define
\[
U_{n,\nu}^{L}(t)
:=
\left(
\sum_{s\le t}
\left|
\int_{\{|x|\le L\}}
f_n(s,x)\,\nu(\{s\},dx)
\right|^p
\right)^{1/p}
\]
and
\[
U_{n,\nu_n}^{L}(t)
:=
\left(
\sum_{s\le t}
\left|
\int_{\{|x|\le L\}}
f_n(s,x)\,\nu_n(\{s\},dx)
\right|^p
\right)^{1/p}.
\]
By \eqref{eq:pvar-small-large-split} and the
\(\ell^p\)-triangle inequality,
\begin{align}
J_n(t)
:=
\left(
\sum_{s\le t}
\left|
\int_{\mathbb R}
f_n(s,x)(\nu-\nu_n)(\{s\},dx)
\right|^p
\right)^{1/p}
&\le
U_{n,\nu}^{L}(t)
+
U_{n,\nu_n}^{L}(t)
\nonumber\\
&+
\left(
\sum_{s\le t}
\left|
\int_{L<|x|\le a}
f_n(s,x)(\nu-\nu_n)(\{s\},dx)
\right|^p
\right)^{1/p}
\nonumber\\
&+
\left(
\sum_{s\le t}
\left|
\int_{\{|x|>a\}}
f_n(s,x)(\nu-\nu_n)(\{s\},dx)
\right|^p
\right)^{1/p}.
\label{eq:pvar-final-four-term-decomposition}
\end{align}

Fix \(\varepsilon>0\) and \(\eta>0\). Since \(C^\circ\) is locally
bounded, choose \(m\in\mathbb N\) such that
$
        \P(C_t^{\circ,*}>m)<\frac{\eta}{8}.
$
Since
$
        \{[X^n]^{(p)}_t:n\ge1\}
$
is bounded in probability, choose \(Q\in\mathbb N\) such that
\[
\P([X]^{(p)}_t\ge Q)
+
\sup_{n\ge1}\P([X^n]^{(p)}_t\ge Q)
<
\frac{\eta}{8}.
\]

For these fixed \(m\) and \(Q\), choose
\[
        0<r<\min\{1,R_0/2\}
\]
so small that
\[
\frac{4^p}{\varepsilon^p}
\left(
2^{p-1}m^p\E[S^{Q,r}]
+
m^pr^p
\right)
<
\frac{\eta}{8} \quad\text{ and } \quad \frac{3^pm^pr^p}{\varepsilon^p/4^p}\le1.
\]
Apply part~(2) of Lemma~\ref{complemma} with
error level \(\eta/8\). There exists \(k_0\) such that, for every
\(k\ge k_0\), there exists \(\ell_0\) such that, for every
\(\ell\ge\ell_0\), there exists \(q_0\) for which
\[
        \sup_{n\ge\ell}
        \P(\theta_{q,k,\ell,n}\le t)
        <
        \frac{\eta}{8},
        \qquad q\ge q_0.
\]
Choose \(k\ge k_0\) sufficiently large that also
$
        \P(A_k(r)^c)<\frac{\eta}{8},
$
then choose \(\ell\ge\ell_0\) and \(q\ge q_0\).

For the fixed values of \(Q,r,k\), we have
$
        \E[Q_n^{Q,r,k}]\longrightarrow0
$
and
$
        \P(B_n^{(p)}(k,r)^c)\longrightarrow0.
$
Hence there exists \(N\ge\ell\) such that, for every \(n\ge N\),
\[
\frac{4^p}{\varepsilon^p}
2^{p-1}m^p\E[Q_n^{Q,r,k}]
<
\frac{\eta}{8}
\quad\text{ and } \quad  \P(B_n^{(p)}(k,r)^c)<\frac{\eta}{8}.
\]

By \eqref{eq:pvar-tail-limsup}, a regular radius \(a>d\) can be
chosen sufficiently large that
\[
        \P(\iota_a\le t)<\frac{\eta}{8}
        \quad\text{ and } \quad \limsup_{n\to\infty}
\P\left(
\sum_{s\le t}
\left|
\int_{\{|x|>a\}}
f_n(s,x)(\nu-\nu_n)(\{s\},dx)
\right|^p
>
\frac{\varepsilon^p}{4^p}
\right)
<
\frac{\eta}{8}.
\]

Substituting these choices into
\eqref{eq:pvar-small-nu-probability} gives
\[
\limsup_{n\to\infty}
\P\left(
U_{n,\nu}^{L}(t)>\frac{\varepsilon}{4}
\right)
\le
\frac{3\eta}{8}.
\]
Similarly, \eqref{eq:pvar-small-nun-probability} gives
\[
\limsup_{n\to\infty}
\P\left(
U_{n,\nu_n}^{L}(t)>\frac{\varepsilon}{4}
\right)
\le
\frac{6\eta}{8}.
\]

Using \eqref{eq:pvar-bounded-annulus-limsup}, together with
\[
        \frac{3^pm^pr^p}{\varepsilon^p/4^p}\le1,
\]
we obtain
\[
\begin{aligned}
&\limsup_{n\to\infty}
\P\left(
\sum_{s\le t}
\left|
\int_{L<|x|\le a}
f_n(s,x)(\nu-\nu_n)(\{s\},dx)
\right|^p
>
\frac{\varepsilon^p}{4^p}
\right)
\le
\frac{4\eta}{8}.
\end{aligned}
\]

Consequently, the union bound applied to
\eqref{eq:pvar-final-four-term-decomposition} yields
\[
\begin{aligned}
\limsup_{n\to\infty}\P(J_n(t)>\varepsilon)
&\le
\frac{3\eta}{8}
+
\frac{6\eta}{8}
+
\frac{4\eta}{8}
+
\frac{\eta}{8}
=
\frac{7\eta}{4}.
\end{aligned}
\]
Since \(\eta>0\) was arbitrary,
\[
        J_n(t)\longrightarrow0
        \qquad\text{in probability}
\]
along the chosen further subsequence.

Step 1 gives, along the same further subsequence,
\[
\left(
\sum_{s\le t}
\left|
\int_{\mathbb R}
(f_n(s,x)-f(s,x))\,\nu(\{s\},dx)
\right|^p
\right)^{1/p}
\longrightarrow0
\]
in probability. Hence \eqref{eq:pvar-main-split} yields
\[
        [A^n-A]^{(p)}_t
        \longrightarrow0
        \qquad\text{in probability}
\]
along the chosen further subsequence.

Since the initial subsequence was arbitrary, the subsequence criterion
gives
\[
        [A^n-A]^{(p)}_t
        \longrightarrow0
        \qquad\text{in probability}
\]
for the original sequence.
\end{proof}

\section*{Funding}
The author declares that no funds, grants, or other support were
received during the preparation of this manuscript.

\end{document}